\documentclass[12pt]{amsart}
\usepackage{amssymb,amscd}

\let\frak\mathfrak

\def\>{\relax\ifmmode\mskip.666667\thinmuskip\relax\else\kern.111111em\fi}
\def\<{\relax\ifmmode\mskip-.333333\thinmuskip\relax\else\kern-.0555556em\fi}
\def\vsk#1>{\vskip#1\baselineskip}
\def\vv#1>{\vadjust{\vsk#1>}\ignorespaces}
\def\vvn#1>{\vadjust{\nobreak\vsk#1>\nobreak}\ignorespaces}
 \let\alb\allowbreak
\def\fratop{\genfrac{}{}{0pt}1}
\def\satop#1#2{\fratop{\scriptstyle#1}{\scriptstyle#2}}
\let\dsize\displaystyle  \let\ssize\scriptstyle
\let\sssize\scriptscriptstyle \def\tfrac{\textstyle\frac}
\let\phan\phantom \let\vp\vphantom

\let\Medskip\medskip
\def\medskip{\par\Medskip}
\let\Bigskip\bigskip
\def\bigskip{\par\Bigskip}

\let\Maketitle\maketitle
\def\maketitle{\Maketitle\thispagestyle{empty}\let\maketitle\empty}

\newtheorem{thm}{Theorem}[section]
\newtheorem{cor}[thm]{Corollary}
\newtheorem{lem}[thm]{Lemma}
\newtheorem{prop}[thm]{Proposition}

\numberwithin{equation}{section}

\theoremstyle{definition}
\newtheorem*{rem}{Remark}
\newtheorem*{example}{Example}

\let\mc\mathcal
\let\nc\newcommand

\let\al\alpha

\let\dl\delta
\let\Dl\Delta
\let\eps\varepsilon

\let\ka\kappa
\let\la\lambda

\let\pho\phi
\let\phi\varphi
\let\si\sigma

\let\Tht\Theta

\let\thi\vartheta

\let\der\partial
\let\Hat\widehat

\let\ox\otimes
\let\Tilde\widetilde
\let\bra\langle
\let\ket\rangle

\let\ge\geqslant
\let\geq\geqslant
\let\le\leqslant
\let\leq\leqslant

\let\on\operatorname
\let\bi\bibitem
\let\bs\boldsymbol
\let\xlto\xleftarrow
\let\xto\xrightarrow

\def\C{{\mathbb C}}
\def\Z{{\mathbb Z}}

\def\Ac{{\mc A}}
\def\B{{\mc B}}
\def\F{{\mc F}}
\def\Hc{{\mc H}}
\def\Jc{{\mc J}}

\def\Sc{{\mc S}}
\def\V{{\mc V}}

\def\+#1{^{\{#1\}}}
\def\lsym#1{#1\alb\dots\relax#1\alb}
\def\lc{\lsym,}
\def\lox{\lsym\ox}

\def\diag{\on{diag}}
\def\End{\on{End}}

\def\Hom{\on{Hom}}
\def\id{\on{id}}

\def\qdet{\on{qdet}}

\def\rank{\on{rank}}

\def\tbigoplus{\mathop{\textstyle{\bigoplus}}\limits}

\def\tr{\on{tr}}

\def\ii{i,\<\>i}
\def\ij{i,\<\>j}
\def\ik{i,\<\>k}
\def\il{i,\<\>l}
\def\ji{j,\<\>i}
\def\jj{j,\<\>j}
\def\jk{j,\<\>k}
\def\kj{k,\<\>j}
\def\kl{k,\<\>l}

\def\II{I\<\<,\<\>I}
\def\IJ{I\<\<,J}
\def\ioi{i+1,\<\>i}
\def\pp{p,\<\>p}
\def\ppo{p,\<\>p+1}
\def\pop{p+1,\<\>p}
\def\pci{p,\<\>i}
\def\pcj{p,\<\>j}
\def\poi{p+1,\<\>i}
\def\poj{p+1,\<\>j}

\def\gl{\mathfrak{gl}}
\def\gln{\mathfrak{gl}_N}
\def\sln{\mathfrak{sl}_N}
\def\slnn{\mathfrak{sl}_{\>n}}

\def\Ugln{U(\gln)}

\def\Yn{Y\<(\gln)}

\def\beq{\begin{equation}}
\def\eeq{\end{equation}}
\def\be{\begin{equation*}}
\def\ee{\end{equation*}}

\nc{\bea}{\begin{eqnarray*}}
\nc{\eea}{\end{eqnarray*}}
\nc{\bean}{\begin{eqnarray}}
\nc{\eean}{\end{eqnarray}}
\nc{\Ref}[1]{{\rm(\ref{#1})}}

\def\g{{\mathfrak g}}
\def\h{{\mathfrak h}}

\let\ga\gamma
\let\Ga\Gamma

\nc{\Il}{{\mc I_{\bs\la}}}
\nc{\bla}{{\bs\la}}
\nc{\Fla}{\F_\bla}
\nc{\tfl}{{T^*\Fla}}
\nc{\GL}{{GL_n(\C)}}
\nc{\GLC}{{GL_n(\C)\times\C^*}}

\let\aal\al % \def\aal{{\bs\al}}
\def\mub{{\bs\mu}}
\def\Dc{\check D}
\def\Di{{\tfrac1D}}
\def\Dci{{\tfrac1\Dc}}
\def\DV{\Di\V^-}
\def\DL{\DV_\bla}
\def\Vl{\V^+_\bla}
\def\DVe{\Dci\V^=}
\def\DLe{\DVe_\bla}

\def\xxx{x_1\lc x_n}
\def\yyy{y_1\lc y_n}
\def\zzz{z_1\lc z_n}
\def\Cx{\C[\xxx]}
\def\Czh{\C[\zzz,h]}
\def\Vz{V\<\ox\Czh}
\def\ty{\Tilde Y\<(\gln)}
\def\tb{\Tilde\B}
\def\IMA{{I^{\<\>\max}}}
\def\IMI{{I^{\<\>\min}}}
\def\CZH{\C[\zzz]^{\>S_n}\!\ox\C[h]}
\def\Sla{S_{\la_1}\!\lsym\times S_{\la_N}}

\def\Czhl{\C[\zzz]^{\>\Sla}\!\ox\C[h]}
\def\Czghl{\C[\<\>\zb\<\>;\Gmm\>]^{\>S_n\times\Sla}\!\ox\C[h]}

\def\Ct{\Tilde C}
\def\fc{\check f}
\def\Hg{{\mathfrak H}}

\def\Pit{\Tilde\Pi}
\def\Qc{\check Q}
\let\sd s %% \def\sd{\dot s}
\def\sh{\hat s}
\def\st{\tilde s}
\def\sih{\hat\si}

\def\Vh{\Hat V}
\def\Wh{\Hat W}

\def\ib{\bs i}
\def\jb{\bs j}
\def\kb{\bs k}
\def\iib{\ib,\<\>\ib}
\def\ijb{\ib,\<\>\jb}
\def\ikb{\ib,\<\>\kb}
\def\jib{\jb,\<\>\ib}

\def\kjb{\kb,\<\>\jb}
\def\zb{\bs z}
\def\zzi{z_1\lc z_i,z_{i+1}\lc z_n}
\def\zzii{z_1\lc z_{i+1},z_i\lc z_n}
\def\zzzn{z_n\lc z_1}
\def\xxi{x_1\lc x_i,x_{i+1}\lc x_n}
\def\xxii{x_1\lc x_{i+1},x_i\lc x_n}
\def\ip{\<\>i\>\prime}
\def\ipi{\>\prime\<\>i}
\def\jp{\prime j}
\def\iset{\{\<\>i\<\>\}}
\def\jset{\{\<\>j\<\>\}}
\def\IMIp{{I^{\min,i\prime}}}
\let\Gmm\Gamma

\def\phoo{\pho_{\>0}}
\def\Vy#1{V^{\bra#1\ket}}
\def\Vyh#1{\Vh^{\bra#1\ket}}
\def\Hh{\Hg_n}
\def\bmo{\nabla^{\>\sssize\mathrm{BMO}}}

\def\Bin{B^{\<\>\infty}}
\def\Bci{\B^{\<\>\infty}}
\def\tbi{\tb^{\<\>\infty}}
\def\Xin{X^\infty}

\def\Bb{\rlap{$\<\>\bar{\phan\B}$}\rlap{$\>\bar{\phan\B}$}\B}
\def\Hb{\rlap{$\bar{\,\,\phan H}$}\bar H}
\def\Hbc{\rlap{$\<\>\bar{\phan\Hc}$}\rlap{$\,\bar{\phan\Hc}$}\Hc}

\def\Bk{B^{\<\>\kk}}
\def\Bck{\B^{\<\>\kk}}
\def\Bbk{\Bb^{\<\>\kk}}
\def\Hck{\Hc^{\<\>\kk}}
\def\Hbck{\Hbc^{\<\>\kk}}

\def\mukp{\mu^{\kk+}}
\def\muke{\mu^{\kk=}}
\def\mukm{\mu^{\kk-}}

\def\nukp{\nu^{\<\>\kk+}}
\def\nuke{\nu^{\<\>\kk=}}
\def\nukm{\nu^{\<\>\kk-}}
\def\nukpm{\nu^{\<\>\kk\pm}}
\def\rhkm{\rho^{\<\>\kk-}}
\def\rhkpm{\rho^{\<\>\kk\pm}}
\def\tbk{\tb^{\<\>\kk}}
\def\tbkp{\tb^{\<\>\kk+}}
\def\tbkm{\tb^{\<\>\kk-}}
\def\tbkpm{\tb^{\<\>\kk\pm}}

\def\tbbm{\Bb^{\<\>\kk-}}
\def\Sk{S^{\>\kk}}
\def\Uk{U^\kk}
\def\Upk{U'^{\<\>\kk}}
\def\Wk{W^\kk}
\def\Whk{\Wh^\kk}
\def\Xk{X^\kk}
\def\Xkp{X^{\kk+}}
\def\Xkm{X^{\kk-}}

\def\Yk{Y^{\<\>\kk}}
\def\ddk_#1{\kk_{#1}\<\>\frac\der{\der\<\>\kk_{#1}}}
\def\zno{\big/\bra\>z_1\<\lsym+z_n\<=0\,\ket}
\def\vone{v_1\<\lox v_1}
\def\Hplus{\bigoplus_\bla H_\bla}
\def\Pone{\mathbb P^{\<\>1}}

\def\bul{\mathbin{\raise.2ex\hbox{$\sssize\bullet$}}}
\def\intt{\mathchoice
{\mathop{\raise.2ex\rlap{$\,\,\ssize\backslash$}{\intop}}\nolimits}
{\mathop{\raise.3ex\rlap{$\,\sssize\backslash$}{\intop}}\nolimits}
{\mathop{\raise.1ex\rlap{$\sssize\>\backslash$}{\intop}}\nolimits}
{\mathop{\rlap{$\sssize\<\>\backslash$}{\intop}}\nolimits}}

\def\nablas{\nabla^{\<\>\star}}
\def\nablab{\nabla^{\<\>\bul}}
\def\chib{\chi^{\<\>\bul}}

\def\zla{\zeta_{\>\bla}^{\vp:}}
\let\gak\gamma %% \def\gak{\ga^{\<\>\kk}}
\let\kk q %% Q
\let\kp\kappa %% p
\let\cc c
\def\kkk{\kk_1\lc\kk_N}
\def\kkn{\kk_1\lc\kk_n}
\let\Ko K
\def\Kh{\Hat\Ko}
\def\zzip{z_1\lc z_{i-1},z_i\<-\kp,z_{i+1}\lc z_n}

\def\GZ/{Gelfand-Zetlin}
\def\KZ/{{\slshape KZ\/}}
\def\qKZ/{{\slshape qKZ\/}}
\def\XXX/{{\slshape XXX\/}}

\begin{document}

\hrule width0pt
\vsk->

\title[Cohomology of the cotangent bundle of a flag variety as a Bethe algebra]
{Quantum cohomology of the cotangent bundle of\\[2pt]
a flag variety as a Yangian Bethe algebra}

\author
[V\<.\,Gorbounov, R.\,Rim\'anyi, V.\,Tarasov, A.\,Varchenko]
{V\<.\,Gorbounov$\>^*$, R.\,Rim\'anyi$\>^{\star}$,
V.\,Tarasov$\>^\circ$, A.\,Varchenko$\>^\diamond$}

\maketitle

\begin{center}
{\it $\kern-.4em^*\<$Institute of Mathematics, University of Aberdeen,
Aberdeen, AB24 3UE UK\/}

\vsk.5>
{\it $^{\star\,\diamond}\<$Department of Mathematics, University
of North Carolina at Chapel Hill\\ Chapel Hill, NC 27599-3250, USA\/}

\vsk.5>
{\it $\kern-.4em^\circ\<$Department of Mathematical Sciences,
Indiana University\,--\>Purdue University Indianapolis\kern-.4em\\
402 North Blackford St, Indianapolis, IN 46202-3216, USA\/}

\vsk.5>
{\it $^\circ\<$St.\,Petersburg Branch of Steklov Mathematical Institute\\
Fontanka 27, St.\,Petersburg, 191023, Russia\/}
\end{center}

{\let\thefootnote\relax
\footnotetext{\vsk-.8>\noindent
$^*\<${\sl E\>-mail}:\enspace v.gorbunov@abdn.ac.uk\\
$^\star\<${\sl E\>-mail}:\enspace rimanyi@email.unc.edu\>,
supported in part by NSF grant DMS-1200685\\
$^\circ\<${\sl E\>-mail}:\enspace vt@math.iupui.edu\>, vt@pdmi.ras.ru\>,
supported in part by NSF grant DMS-0901616\\
$^\diamond\<${\sl E\>-mail}:\enspace anv@email.unc.edu\>,
supported in part by NSF grant DMS-1101508}}

\begin{abstract}
We interpret the equivariant cohomology algebra
\,$H^*_{GL_n\times\C^*}(\tfl;\C)$ \,of the cotangent bundle of
a partial flag variety $\Fla$ parametrizing chains of subspaces
\,$0=F_0\subset F_1\subset\dots\subset F_N =\C^n$,
\,$\dim F_i/F_{i-1}=\la_i$, as the Yangian Bethe algebra \,$\Bci(\DL)$ of
the \,$\gln$-weight subspace $\DL\<$ of a $\Yn$-module $\DV$. Under this
identification the dynamical connection of \cite{TV1} turns into the quantum
connection of \cite{BMO} and \cite{MO}. As a result of this identification
we describe the algebra of quantum multiplication on
$H^*_{GL_n\times\C^*}(\tfl;\C)$ as the algebra of functions on fibers of
a discrete Wronski map. In particular this gives generators and relations
of that algebra. This identification also gives us hypergeometric solutions
of the associated quantum differential equation. That fact manifests the
Landau-Ginzburg mirror symmetry for the cotangent bundle of the flag variety.
\end{abstract}

\setcounter{footnote}{0}
\renewcommand{\thefootnote}{\arabic{footnote}}

\section{Introduction}

A Bethe algebra of a quantum integrable model is a commutative algebra of
linear operators (Hamiltonians) acting on the space of states of the model.
An interesting problem is to describe the Bethe algebra as the algebra of
functions on a suitable scheme. Such a description can be considered as
an instance of the geometric Langlands correspondence, see for example
\cite{MTV2, MTV3}. The $\gln$ \XXX/ model is a quantum integrable model.
The Bethe algebra \,$\Bck\<$ of the \XXX/ model is a commutative subalgebra
of the Yangian $\Yn$. The algebra \,$\Bck\<$ depends on the parameters
$\kk=(\kkk)\in\C^N$. Having a $Y(\gln)$-module $M$, one obtains the commutative
subalgebra $\Bck(M)\subset\End(M)$ as the image of \,$\Bck\<$. The geometric
interpretation of the algebra $\Bck(M)$ as the algebra of functions on a scheme
leads to interesting objects, see for example, \cite{MTV4}.

\goodbreak
\vsk.1>
In this paper we consider an extension of the Yangian \,$\Yn$, denoted $\ty$,
which is a subalgebra of \,$\Yn\ox\C[h]$, and we work with the corresponding
Bethe subalgebra \,$\tbk\subset\ty$.

\vsk.1>
One of the most interesting Yangian modules is the vector space
\,$\V=V\!\ox\Czh$ of \,$V\<$-valued polynomials in $\zzz$, where
$V=(\C^N)^{\ox n}$ is the tensor power of the standard vector representation
of $\gln$. We introduce on $\V$ two Yangian actions, called \,$\pho^\pm$,
and two actions of the symmetric group \,$S_n$, called \,$S^\pm_n$-actions.
The Yangian action $\pho^+$ commutes with the \,$S_n^+$-action. Hence,
the subspace \,$\V^+\!\subset\V$ of \,$S_n^+$-invariants is a Yangian module.
Similarly, the Yangian action \,$\pho^-$ commutes with the \,$S_n^-$-action.
Hence, the subspace \,$\V^-\!\subset\V$ of \,$S_n^-$-skew-invariants is
a Yangian module. The Yangian module structure on \,$\V^-$ induces a Yangian
module structure on \,$\DV=\{\Di f\ |\ f\<\in\V^-\}$, where
\,$D=\prod_{1\leq i<j\leq n}(z_i-z_j+h)$. The Bethe algebra \,$\tbk\<$
\vv.16>
preserves the $\gln$-weight decompositions \,$\V^+\!=\bigoplus_\bla\Vl$
and $\DV\!=\bigoplus_\bla\DL$, \,$\bla=(\la_1\lc\la_N)\in\Z^N_{\geq 0}$,
\,$|\bla|=n$.

\vsk.2>
In this paper we study the limit of the algebras \,$\tbk(\Vl)$, \,$\tbk(\DL)$
as \,$\kk$ tends to infinity so that $\kk_{i+1}/\kk_i\to0$ for all
$i=1\lc N-1$, \,and \,$\kk_N=1$. We show that in this limit each of the Bethe
algebras \,$\tbi(\Vl)$, \,$\tbi(\DL)$ can be identified with the equivariant
cohomology algebra \,$H_\bla:=H^*_{GL_n\times\C^*}(\tfl;\C)$ \,of the cotangent
bundle of the partial flag variety $\Fla$ parametrizing chains of subspaces
\vvn.2>
\be
0=F_0\subset F_1\lsym\subset F_N =\C^n,\qquad \dim F_i/F_{i-1}=\la_i\,.
\vv.2>
\ee
More precisely, we construct an algebra isomorphism
\vvn.1>
\,$\mu^+_\bla\<:H_\bla\to\tbi(\Vl)$ and a vector space isomorphism
\,$\nu^+_\bla\<:H_\bla\to\Vl$, which identify the
\,$\tbi(\Vl)$-module \,$\Vl$ with the action of the algebra
$H_\bla$ on itself by multiplication operators, see Theorem \ref{thm main+}.
\vvn.1>
Similarly, we construct an algebra isomorphism
\,$\mu^-_\bla\<:H_\bla\to\tbi(\DL)$ and a vector space isomorphism
\,$\nu^+_\bla\<:H_\bla\to\DL$, which identify the $\tbi(\DL)$-module
\vvn.1>
\,$\DL$ with the action of the algebra $H_\bla$ on itself by multiplication
operators, see Theorem \ref{thm main-}.

\vsk.2>
In Section \ref{elsewhere}, using the discrete Wronski determinant $\Wk(u)$,
see formula \Ref{W}, we introduce an algebra $\Hck_\bla$ and construct
vector space isomorphisms \,$\mukp_\bla\!:\Hck_\bla\to\>\Vl$,
\,$\mukm_\bla\!:\Hck_\bla\to\>\DL$ \,and algebra isomorphisms
\,$\nukp_\bla\!:\Hck_\bla\<\to\tb(\Vl)$\>,
\,$\nukm_\bla\!:\Hck_\bla\<\to\tb(\DL)$\>. The isomorphisms \,$\mukp_\bla\!$
and \,$\nukp_\bla\!$ identify the \,$\tbk(\Vl)$-module \,$\Vl$ and the regular
representation of \,$\Hck_\bla$, see Theorem \ref{mainK+}. Similarly,
the isomorphisms \,$\mukm_\bla\!$ and \,$\nukm_\bla\!$ identify the
\,$\tbk(\DL)$-module \,$\DL$ and the regular representation of \,$\Hck_\bla$,
see Theorem \ref{mainK-}. In particular, the algebras \,$\tbk(\Vl)$ and
$\tbk(\DL)$ are isomorphic, and the \,$\tbk(\Vl)$-module \,$\Vl$ is isomorphic
to the \,$\tbk(\DL)$-module \,$\DL$.

\vsk.2>
Under the isomorphisms \,$\nu^+_\bla\<:H_\bla\to\V^+$ and
\,$\nu^-_\bla\<:H_\bla\to\DV$ the subalgebras \,$\tbk(\Vl)$ and \,$\tbk(\DL)$
induce respectively commutative subalgebras \,$\tbkp_\bla\!$ and
\,$\tbkm_\bla\!$ of \,$\End(H_\bla)$. Theorem \ref{conj} says that the Bethe
algebra \,$\tbkm_\bla\!$ is the algebra \>$QH_{GL_n\times\C^*}(\tfl;\C)$ of
quantum multiplication on \>$H_\bla$ described in \cite{BMO, MO}
under the identification of notation \Ref{id par}. The proof of
Theorem \ref{conj} is based on \cite{RTV}, where it is shown that the Yangian
structure on the equivariant cohomology defined in \cite{MO} coincides with
the Yangian structure \>$\rho^-$ defined in Section~\ref{sec cohom and Yang}.
In Section \ref{sec Special} we give a direct proof of the special case of
Theorem \ref{conj} when \,$N=n$\,, \,$\bla=(1\lc1)$. The proof is based on
formulae from \cite{BMO}.

\vsk.1>
In the first version of this paper, see {\sf arXiv:1204.5138v1}, the statement
of Theorem \ref{conj} was formulated as a conjecture. The later paper \cite{MO}
contained, in particular, a description of the quantum cohomology algebra
\>$QH_{GL_n\times\C^*}(\tfl;\C)$. That allowed us to related the quantum
cohomology algebra and the Bethe algebra \,$\tbkm_\bla\!$.

\vsk.1>
Corollary \ref{conjcor} gives a description of the quantum cohomology algebra
\>$QH_{GL_n\times\C^*}(\tfl;\C)$ by generators and relations.
Theorem \ref{lim h to inf} shows that in the limit $h\to\infty$ such that
\,$h^{\la_i+\la_{i+1}}q_{i+1}/q_i$ are fixed for all \,$i=1,\dots, N-1$,
the description of \>$QH_{GL_n\times\C^*}(\tfl;\C)$ turns into the known
description of the quantum cohomology algebra \>$QH_{GL_n}(\Fla;\C)$ of
the partial flag variety $\Fla$, see \cite{AS}\>.

\vsk.2>
In more general terms, our identification of the quantum multiplication
on \,$H_\bla$ and the Yangian Bethe algebra \,$\tbkm\!$ indicates a relation
between quantum integrable chain models and quantum cohomology algebras, see
also the paper \cite{MO} (which appeared later than the first version of this
paper). That relation was discussed in physics literature in recent works
by Nekrasov and Shatashvili, see for example \cite{NS}.

\vsk.1>
Note that a relationship between quantum integrable models and quantum
cohomology\<\>/\<\>WZNW fusion rings has been shown also in \cite{K, KoS},
see references therein.

\vsk.2>
The identification of the quantum multiplication on \,$H_\bla$ and the Yangian
Bethe algebra \,$\tbkm\!$ gives a relation between interesting objects in both
theories. For example, the eigenvectors of the Bethe algebra is the main object
of the \XXX/ model and finding eigenvectors is the subject of the well
developed Bethe ansatz theory. Under the above identification the eigenvectors
correspond to idempotents of the quantum multiplication and the idempotents
are important in the theory of Frobenius manifolds, see \cite{D}. According
to the above identification, the idempotents of the quantum cohomology algebra
can be determined by the \XXX/ Bethe ansatz method, see an example in
Section \ref{sec Okounk}.

\vsk.2>
The above identification brings another interesting relation. Namely,
Theorem \ref{conj} shows that the isomorphism \,$\nu^-_\bla$ identifies
the trigonometric dynamical connection of \cite{TV1} with the quantum
connection of \cite{BMO, MO}. It is known that the flat sections of
the trigonometric dynamical connection are given by multidimensional
hypergeometric integrals, see \cite{TV1,MV,SV}, cf.~\cite{TV2}.
These hypergeometric integrals provide flat sections of the quantum connection
of \cite{MO,BMO}, see \cite{TV3}. This presentation of flat sections
of the quantum connection as hypergeometric integrals is in the spirit of
the mirror symmetry, see Candelas et al.~\cite{COGP}, Givental \cite{G1,G2},
and \cite{BCK,BCKS,GKLO,I,JK}, see also an example in Section \ref{sec Okounk}.

\vsk.2>
The results of this paper are parallel to the results of the paper \cite{RSTV},
where we consider the Bethe subalgebra of $U(\gln[t])$ instead of the Bethe
algebra of $\ty$.

\vsk.2>
We thank D.\,Maulik and A.\,Okounkov for answering questions on \cite{MO}.
The first and fourth authors thank the Max Planck Institute for Mathematics
for hospitality. The fourth author also thanks the Hausdorff Research Institute
for Mathematics and Institut des Hautes Etudes Scientifiques for hospitality.

\newpage
\section{Spaces \,$\V^+$ and \,$\DV$}
\label{alg sec}

\subsection{Lie algebra $\gln$}
\label{sec gln}

Let $e_{\ij}$, $i,j=1\lc N$, be the standard generators of the Lie algebra
$\gln$ satisfying the relations
$[e_{\ij},e_{\kl}]=\dl_{\jk}e_{\il}-\dl_{\il}e_{\kj}$. We denote by
$\h\subset\gln$ the subalgebra generated by $e_{\ii},\,i=1\lc N$. For a Lie
algebra $\g\,$, we denote by $U(\g)$ the universal enveloping algebra of $\g$.

A vector $v$ of a $\gln$-module $M$ has weight
$\bla=(\la_1\lc\la_N)\in\C^N$ if $e_{\ii}\>v=\la_i\>v$ for $i=1\lc N$.
We denote by $M_\bla\subset M$ the weight subspace of weight $\bla$.

\vsk.2>
Let $\C^N$ be the standard vector representation of $\gln$ with basis
$v_1\lc v_N$ such that $e_{\ij}v_k=\dl_{\jk}v_i$ for all $i,j,k$.
A tensor power $V=(\C^N)^{\ox n}$ of the vector representation has a basis
given by the vectors $v_{i_1}\!\lox v_{i_n}$, where
\;$i_j\in\{1\lc N\}$. Every such sequence $(i_1\lc i_n)$ defines
a decomposition $I=(I_1\lc I_N)$ of $\{1\lc n\}$ into disjoint subsets
$I_1\lc I_N$, \,where \;$I_j=\{k\ |\ i_k=j\}$. We denote the basis vector
$v_{i_1}\!\lox v_{i_n}$ by $v_I$.

Let
\vvn-.6>
\be
V\,=\!\bigoplus_{\bla\in\Z^N_{\geq 0},\,|\bla|=n}\!V_\bla
\ee
be the weight decomposition. Denote $\Il$ the set of all indices $I$ with
$|I_j|=\la_j$, \;$j=1,\dots N$. The vectors $\{v_I, I\in\Il\}$ form a basis of
$V_\bla$. The dimension of $V_\bla$ equals the multinomial coefficient
$d_\bla:=\frac{n!}{\la_1!\dots\la_N!}$.

\smallskip
Let $\Sc$ be the bilinear form on $V$ such that the basis $\{v_I\}$ is
orthonormal. We call $\Sc$ the Shapovalov form.

\subsection{$S_n$-actions}
\label{sec S-actions}
Define an action of the symmetric group $S_n$ on $\Il$.
Let $I=(I_1\lc I_N)\in\Il$, where $I_j=\{i_1\lc i_{\la_j}\}\subset\{1\lc n\}$.
For $\si\in S_n$, define $\si(I) = (\si(I_1)\lc\si(I_N))$, where
$\si(I_j) =\{\si(i_1)\lc\si(i_{\la_j})\}$.

Let $P^{(\ij)}$ be the permutation of the $i$-th and $j$-th factors
of $V=(\C^N)^{\ox n}$.
Define two $S_n$-actions on \,$V\<$-valued functions of $\zzz,h$, called
$S^\pm_n$-actions. For the $S_n^+$-action, the $i$-th elementary transposition
$s_i\in S_n$ acts by the formula: \vvn.4>
\begin{align}
\label{Sn+}
s_i: f(\zzz,h)\,\mapsto\,
\frac{(z_i-z_{i+1})\,P^{(\ii+1)}-h}{z_i-z_{i+1}}\;
& f(\zzii,h)\,+{}
\\[3pt]
{}+\,\frac{h}{z_i-z_{i+1}}\;&f(\zzi,h)\,.
\notag
\\[-15pt]
\notag
\end{align}
For the $S_n^-$-action, the $i$-th elementary transposition $s_i\in S_n$
acts by the formula:
\vvn.4>
\begin{align}
\label{Sn-}
s_i: f(\zzz,h)\,\mapsto\,
\frac{(z_i-z_{i+1})\,P^{(\ii+1)}+h}{z_i-z_{i+1}}\;
& f(\zzii,h)\,-{}
\\[3pt]
{}-\,\frac{h}{z_i-z_{i+1}}\;&f(\zzi,h)\,.
\notag
\\[-15pt]
\notag
\end{align}

\begin{lem}
\label{Spm}
The \,$S_n^\pm$-actions \Ref{Sn+} and \Ref{Sn-} are well-defined.
\qed
\end{lem}

Define operators \,$\st_1^\pm\lc\st_{n-1}^\pm$ acting on \<$V\<$-valued
functions of $\zzz,h$ \,by \vvn.2> \beq
\label{stilde}
\st_i^\pm\>f(\zzz,h)=
\frac{(z_i-z_{i+1})\,P^{(\ii+1)}\mp h}{z_i-z_{i+1}\mp h}\;f(\zzii,h)\,.
\eeq

\begin{lem}
\label{st+}
The assignment \,$s_i\mapsto\st_i^+$, \,$i=1\lc n-1$,
\,defines an action of $S_n$.
\qed
\end{lem}

\begin{lem}
\label{st-}
The assignment \,$s_i\mapsto\st_i^-$, \,$i=1\lc n-1$,
\,defines an action of $S_n$.
\qed
\end{lem}

\begin{lem}
\label{S2}
The function $f(\zzz,h)$ is invariant with respect to the $S_n^+$-action
(\>resp.\ $S_n^-$-action) \,if and only if \,$f=\st_i^+f$
\,(\>resp.~\,$f=\st_i^-f$\,) \,for every \,$i=1\lc n-1$.
\qed
\end{lem}

Define operators \,$\sh_1^\pm\lc\sh_{n-1}^\pm$ acting on functions of
$\zzz,h$ \,by
\begin{align}
\label{sh}
\sh^\pm_if(\zzz,h)\,=\,\frac{z_i-z_{i+1}\pm h}{z_i-z_{i+1}}\; &
f(\zzii,h)\,\mp{}
\\[3pt]
{}\mp\,\frac{h}{z_i-z_{i+1}}\;&f(\zzi,h)
\notag
\end{align}

\begin{lem}
\label{sh+}
The assignment \,$s_i\mapsto\sh_i^+$, \,$i=1\lc n-1$,
\,defines an action of $S_n$.
\qed
\end{lem}

\begin{lem}
\label{sh-}
The assignment \,$s_i\mapsto\sh_i^-$, \,$i=1\lc n-1$,
\,defines an action of $S_n$.
\qed
\end{lem}

Let $f(\zzz,h)$ be a \,$V\<$-valued function with coordinates
$\{f_I(\zzz,h)\}$:
\vvn.2>
\be
f(\zzz,h)\,=\,\sum_I\,f_I(\zzz,h)\;v_I\,.
\vv-.1>
\ee

\begin{lem}
\label{S1}
The function $f(\zzz,h)$ is invariant with respect to the $S_n^+$-action
(\>resp.\ $S_n^-$-action) \,if and only if \,$f_{\si(I)}=\sih^+\<f_I$
\,(\>resp.~$f_{\si(I)}=\sih^-\<f_I$) \,for any \,$I\in\Il$ and
any \,$\si\in S_n$.
\qed
\end{lem}

\begin{lem}
\label{S3}
The function $f(\zzz,h)$ is skew-invariant with respect to the $S_n^+$-action
(\>resp.\ $S_n^-$-action) if and only if \,$f_{\si(I)}=(-1)^\si\sih^-f_I$
\,(\>resp.~$f_{\si(I)}=(-1)^\si\sih^+f_I$) \,for any \,$I\in\Il$ and
any \,$\si\in S_n$.
\qed
\end{lem}

Let
\vvn-.7>
\be
D\,=\!\prod_{1\leq i<j\leq n}(z_i-z_j+h)\,,\quad\qquad
\Dc\,=\!\prod_{1\leq i<j\leq n}(z_j-z_i+h)\,.
\ee

\begin{lem}
\label{S4}
The function $f$ is skew-invariant with respect to the $S_n^+$-action
(\>resp.\ $S_n^-$-action) if and only if the function \,$\Dci f$
\,(\>resp.~$\Di f$) \,is invariant with respect to the $S_n^+$-action
(\>resp.\ $S_n^-$-action).
\qed
\end{lem}

\subsection{Spaces \,$\V^+\<$, \,$\DVe\<$ and \,$\DV$}
\label{sec Vpm}
Let \,$\V^+=(\Vz)^{S_n^+}$ be the space of $S_n^+$-invariant
\,$V\<$-valued polynomials.
Let \,$\DVe$ be the space of $S_n^+$-invariant
\,$V\<$-valued functions of the form $\Dci f$, \,$f\<\in\Vz$.
Let \,$\DV$ be the space of $S_n^-$-invariant
\,$V\<$-valued functions of the form $\Di f$, \,$f\<\in\Vz$.
All \,$\V^+\<$, \,$\DVe\<$ and \,$\DV$ are $\CZH$-modules.

\vsk.2>
The $S^\pm$-actions on $\Vz$ commute with the $\gln$-action.
Hence \,$\V^+\<$, \,$\DVe\<$ and \,$\DV$ are $\gln$-modules.
Consider the $\gln$-weight decompositions
\vvn.3>
\be
\V^+=\!\tbigoplus_{\fratop{\bla\in\Z^N_{\geq 0}}{|\bla|=n}}\!\Vl,
\qquad\quad
\DVe=\!\tbigoplus_{\fratop{\bla\in\Z^N_{\geq 0}}{|\bla|=n}}\!\DLe,
\qquad\quad
\DV=\!\tbigoplus_{\fratop{\bla\in\Z^N_{\geq 0}}{|\bla|=n}}\!\DL.
\ee

\vsk.3>
Define a partial ordering on subsets of \,$\{1\lc n\}$ of the same cardinality.
Say \,$A\le B$ \,if \,$A=\{a_1\lsym<a_i\}$, \,$B=\{b_1\lsym<b_i\}$,
and \,$a_j\le b_j$ \,for all \,$j=1\lc i$.

\vsk.2>
Fix \,$\bla=(\la_1\lc\la_N)\in\Z_{\geq 0}$, \,$|\bla|=n$\,.
For \,$I,J\in\Il$, \,say \,$I<J$ \,if \,$I_k=J_k$ \,for \,$k=1\lc l-1$
\,with some \,$l$, \,and \,$I_l<J_l$. Define \,$\IMI,\IMA\in\Il$ \,by
\vvn.3>
\begin{gather}
\label{IMA}
\IMI\,=\,\bigl(\{1\lc\la_1\}\>,\{\la_1+1\lc\la_1+\la_2\}\>,\;\ldots\;,
\{n-\la_N+1\lc n\}\bigr)\,,
\\[4pt]
\notag
\IMA\,=\,\bigl(\{n-\la_1+1\lc n\}\>,\{n-\la_1-\la_2+1\lc n-\la_1\}\>,
\;\ldots\;,\{1\lc\la_N\}\bigr)\,.\!
\end{gather}
Clearly, \,$\IMI\le I\le\IMA$ \,for any \,$I\in\Il$.
\vsk.2>
Given \,$I\in\Il$, \,denote
\vvn-.2>
\beq
\label{Q}
Q(\zb_I)\,=\!\prod_{1\le a<b\le N}\,\prod_{i\in I_a}\,\prod_{j\in I_b}\,
(z_i-z_j+h)\,.
\vv.2>
\eeq
For any function \,$f(\zzz,h)$\>, \,set \,$\fc(\zzz,h)=f(\zzzn,h)$.
Let
\vvn.3>
\beq
\label{Ql}
Q_\bla(\zzz,h)\,=\,Q(\zb_\IMI)\,,
\vv.2>
\eeq
so that \,$\Qc_\bla(\zzz,h)=Q(\zb_\IMA)$\>.
\vsk.3>
Define
\vvn-.4>
\begin{gather}
\label{thi+}
\thi^+_\bla:\Czhl\,\to\,\Vl\,,
\\[4pt]
\notag
\thi^+_\bla(f)\,=\,\frac1{\la_1!\dots\la_N!}\;
\sum_{\si\in S_n\!}\sih^+\!\fc\;v_{\si(\IMA)}\,,
\end{gather}
\vv.1>
\begin{gather}
\label{thi=}
\thi^=_\bla:\Czhl\,\to\,\DLe\,,
\\[4pt]
\thi^=_\bla(f)\,=\,\frac1{\la_1!\dots\la_N!}\;
\sum_{\si\in S_n\!}\sih^+\bigl({\tfrac1{\Qc_\bla}\fc}\>\bigr)\;v_{\si(\IMA)}\,,
\notag
\\[-24pt]
\notag
\end{gather}
and
\begin{gather}
\label{thi-}
\thi^-_\bla:\Czhl\,\to\,\DL\,,
\\[4pt]
\thi^-_\bla(f)\,=\,\frac1{\la_1!\dots\la_N!}\;
\sum_{\si\in S_n\!}\sih^-\bigl({\tfrac1{Q_\bla}f}\>\bigr)\;v_{\si(\IMI)}\,.
\notag
\end{gather}

\begin{lem}
\label{thi+-}
The maps \;$\thi^+_\bla\<$, \,$\thi^=_\bla\<$, and \;$\thi^-_\bla$ \!are
isomorphisms of
\vvn.16>
the \;$\CZH$-module \;$\Czhl$ with the \;$\CZH$-modules \;$\Vl$, \,$\DLe$
\alb and \;$\DL$, respectively.
\end{lem}

\begin{proof}
The operators \,$\sih^+\!$, \,$\si\in S_n$, \,preserve \;$\Czh$\>.
\vvn.1>
Thus \,$\thi^+_\bla(f)=\sum_I\>f_I\>v_I$, \,and \,$f_I\in\Czh$ \,for any \,$I$.
Moreover, \,$f_{\si(I)}=\sih^+\<f_I$ \,for any \,$\si\in S_n$\>, and
\,$f_\IMA=\fc$. Hence, the map \,$\thi^+_\bla$ is well-defined by
Lemma \ref{S1}. If \,$g\in\Vl$, \;$g=\sum_I\>g_I\>v_I$, then
\,$g=\thi^+_\bla(\check g_\IMA)$ by Lemma \ref{S1},
\>so \,$\thi^+_\bla$ is an isomorphism.
\vsk.2>
The proof for the cases of \,$\thi^=_\bla$ and \,$\thi^-_\bla$ are similar
because the operators \,$\sih^+\!$, \,$\si\in S_n$, \,preserve the localized
algebra \,$\C[\zzz,h,\Dc^{-1}]$\>, and the operators \,$\sih^-\!$,
\,$\si\in S_n$, \,preserve the localized algebra \,$\C[\zzz,h,D^{-1}]$\>.
\vv-.3>
\end{proof}

It is known in Schubert calculus that \,$\C[\zzz]^{\>\Sla}$ is a free
\vvn.1>
\,$\C[\zzz]$-module of rank \,$d_\bla$. This yields that for any $\bla$,
the subspaces \;$\V_\bla^+$, \;$\DLe$ \,and \;$\DL$ are free \,$\CZH$-modules
of rank \,$d_\bla$.

\vsk.3>
The Shapovalov form $\Sc$ on $V$ induces a pairing
\vvn.3>
\beq
\label{Shap pair}
(\Vz)\ox (\Vz)\to\Czh
\vv.3>
\eeq
denoted by the same letter.

\begin{lem}
\label{lem Shap pm}
The pairing \Ref{Shap pair} induces a pairing
\vvn.2>
\beq
\label{s+-}
\V^+\!\ox\DV\to\,\CZH\,.
\vv.1>
\eeq
\end{lem}
\begin{proof}
For any $f\<\in\V^+$ and $g\in\DV$, the scalar function $\Sc(f,g)$ is
a symmetric function of $\zzz$ with possible poles only at the hyperplanes
$z_i-z_j+h=0$ for $1\leq i<j\leq n$. Since this arrangement of hyperplanes is
not invariant under the permutations of $\zzz$, the poles are absent.
\end{proof}

\begin{lem}
\label{surj}
For any \,$\bla$, the pairing \;$\Vl\<\ox\DL\<\to\>\CZH$\>,
induced by pairing \Ref{s+-}, is surjective.
\end{lem}

\noindent
Lemma \ref{surj} is proved in Section \ref{tyactions}

\vsk.5>
The pairing \Ref{Shap pair} also induces a pairing
\vvn.3>
\beq
\label{s=-}
\DVe\<\<\ox\DV\to\,Z^{-1}\>\CZH\,,
\vv.1>
\eeq
where
\vvn-.8>
\beq
\label{Z}
Z\,=\,D\>\Dc\,=\prod_{\satop{\ij=1}{i\ne j}}^n(z_i-z_j+h)\,.
\vv.1>
\eeq
The pairing \Ref{s+-} and \Ref{s=-} will also be denoted by $\Sc$.

\begin{lem}
\label{Sfg}
The pairings \,\Ref{s+-} and \,\Ref{s=-} are nondegenerate. That is,
if \;$\Sc(f,g)=0$ \,for a given \,$f\<\in\V^+$ and every \,$g\in\DV$,
then \,$f=0$. And similarly, for all other possible cases.
%if \;$\Sc(f,g)=0$ \,for a given \,$g\in\DV$ and every
%\,$f\<\in\V^+$, then \,$g=0$.
\qed
\end{lem}

\subsection{Vectors \,$\xi_I^\pm$}
\label{secxi}
Fix \,$\bla=(\la_1\lc\la_N)\in\Z_{\geq 0}$, \,$|\bla|=n$\,.
Given \,$I\in\Il$, \,denote
\vvn.4>
\beq
\label{R}
R(\zb_I)\,=\!\prod_{1\le a<b\le N}\,\prod_{i\in I_a}\,\prod_{j\in I_b}\,
(z_i-z_j)\,.
\vv.1>
\eeq

\begin{prop}
\label{xi+}
There exist unique elements
\,$\{\xi^+_I\in V\!\ox\C[\zzz,h,\Dc^{-1}]\ |\ I\in\Il\}$ such that
\;$\xi^+_\IMI=v_\IMI$ \,and
\vvn-.3>
\beq
\label{xi+si}
\xi^+_{s_i(I)}\>=\,\st_i^+\>\xi^+_I
\vv.2>
\eeq
for every \,$I\in\Il$ \,and \,$i=1\lc n-1$\,. Moreover,
\vvn.3>
\beq
\label{xi+v}
\xi^+_I\,=\,\sum_{J\le I}\,X^+_{\IJ}\,v_J\,,
\eeq
where \,$X^+_{\IJ}\in\C[\zzz,h,\Dc^{-1}]$\>, \,$X^+_{\II}\ne 0$, \,and
\beq
\label{X+}
X^+_{\IMA\<,\>\IMA}\,=\,\frac{R(\zb_\IMA)}{Q(\zb_\IMA)}\,.
\eeq
\end{prop}

For example, if \,$N=n=2$ \,and \,$\bla=(1,1)$, \,then
\,$\xi^+_{(1,2)}(z_1,z_2,h)=v_{(1,2)}$ \,and
\vvn.1>
\be
\xi^+_{(2,1)}(z_1,z_2,h)\,=\,
%%\frac{(z_1-z_2)\>P_{(1,2)}-h}{z_1-z_2-h}\;v_{(1,2)}\,=\,
\frac{z_2-z_1}{z_2-z_1+h}\;v_{(2,1)}+\frac h{z_2-z_1+h}\;v_{(1,2)}\,.
\vv.3>
\ee

\begin{prop}
\label{xi-}
There exist unique elements
\,$\{\xi^-_I\in V\!\ox\C[\zzz,h,D^{-1}]\ |\ I\in\Il\}$ such that
\;$\xi^-_\IMA=v_\IMA$ \,and
\vvn-.3>
\beq
\label{xi-si}
\xi^-_{s_i(I)}\>=\,\st_i^-\>\xi^-_I
\vv.2>
\eeq
for every \,$I\in\Il$ \,and \,$i=1\lc n-1$\,. Moreover,
\vvn.3>
\beq
\label{xi-v}
\xi^-_I\,=\,\sum_{J\ge I}\,X^-_{\IJ}\,v_J\,,
\eeq
where \,$X^-_{\IJ}\in\C[\zzz,h,D^{-1}]$\>, \,$X^-_{\II}\ne 0$, \,and
\vvn-.2>
\beq
\label{X-}
X^-_{\IMI\<,\>\IMI}\,=\,\frac{R(\zb_\IMI)}{Q(\zb_\IMI)}\,.
\vv.1>
\eeq
\end{prop}

\begin{proof}[Proof of Propositions~\ref{xi+},~\ref{xi-}]
Proposition~\ref{xi+} follows from formula~\Ref{stilde} and Lemma \ref{st+}.
Notice that \,$\st_i^+\>v_\IMI=v_\IMI$ \,if and only if
\,$s_i(\IMI)=\IMI$. Similarly, Proposition \ref{xi-} follows from
formula~\Ref{stilde} and Lemma~\ref{st-}.
\end{proof}

For example, if \,$N=n=2$ \,and \,$\bla=(1,1)$, \,then
\,$\xi^-_{(2,1)}(z_1,z_2,h)=v_{(2,1)}$ \,and
\vvn.1>
\be
\xi^-_{(1,2)}(z_1,z_2,h)\,=\,
%%\frac{(z_1-z_2)\>P_{(1,2)}+h}{z_1-z_2+h}\;v_{(1,2)}\,=\,
\frac{z_1-z_2}{z_1-z_2+h}\;v_{(1,2)}+\frac h{z_1-z_2+h}\;v_{(2,1)}\,.
\vv.3>
\ee

Let \,$f_I(\zzz,h)$, \,$I\in\Il$, be a collection of scalar functions.

\begin{lem}
\label{xinv}
The \,$V\!\<$-valued function \;$\sum_I f_I(\zzz,h)\,\xi^+_I$
\rlap{\,$\Bigl($resp.~$\sum_I f_I(\zzz,h)\,\xi^-_I\,\Bigr)$,}\\
is invariant with
respect to the $S_n^+$-action (\>resp.~$S_n^-$-action) \,if and only if
\vvn.2>
\be
f_{\si(I)}(\zzz,h)\,=\,f_I(z_{\si_1}\lc z_{\si_n},h)
\vv.2>
\ee
for any \,$I\in\Il$ and any \,$\si\in S_n$.
\qed
\end{lem}

\begin{prop}
\label{thixi}
For any \,$f\<\in\Czhl$\>, \,we have
\vvn.2>
\begin{gather}
\label{thi+xi}
\thi^+_\bla(f)\,=\>
\sum_{I\in\Il}\,\frac{f(\zb_I,h)\>Q(\zb_I)}{R(\zb_I)}\,\,\xi^+_I\,,
\\[3pt]
\label{thi=xi}
\thi^=_\bla(f)\,=\>\sum_{I\in\Il}\,\frac{f(\zb_I,h)}{R(\zb_I)}\,\,\xi^+_I\,,
\\[3pt]
\label{thi-xi}
\thi^-_\bla(f)\,=\>\sum_{I\in\Il}\,\frac{f(\zb_I,h)}{R(\zb_I)}\,\,\xi^-_I\,.
\end{gather}
\end{prop}
\begin{proof}
By formulae \Ref{xi+v}, \Ref{X+}, we have
\,$\thi^+_\bla(f)=\sum_{\>I}\>c_I\>\xi^+_I$
\vvn.16>
for some coefficients $c_I$, and
\,$c_\IMA=f(\zb_\IMA,h)\>Q(\zb_\IMA)/R(\zb_\IMA)$.
Now formulae \Ref{thi+xi}, \Ref{thi=xi} follow from Lemma \ref{xinv}.
\vsk.2>
Formula \Ref{thi-xi} follows similarly from formulae \Ref{xi-v}, \Ref{X-},
and Lemma \ref{xinv}.
\end{proof}

\begin{thm}
\label{thm S(+,-)}
For \,$I,J\in\Il$, we have
\vvn-.3>
\beq
\label{S+-}
\Sc(\xi^+_I,\xi^-_J)\,=\,\dl_{\IJ}\;\frac{R(\zb_I)}{Q(\zb_I)}\;.
\eeq
\end{thm}

\begin{proof}
If \,$I=\IMI$ or \,$J=\IMA$, the claim follows from
formulae~\Ref{xi+v}\,--\,\Ref{X-}.
\,Then formula \Ref{S+-} extends to arbitrary \,$I,J\in\Il$
\,by properties~\Ref{xi+si}, \Ref{xi-si}.
\end{proof}

\subsection{Distinguished elements \,$v^+_\bla\<$, $v^=_\bla\<$,
and \,$v^-_\bla$}
\label{v+-=}
\,Let \;$v^+_\bla\!=\sum_{I\in\Il}\!v_I$\>.
\vvn.08>
Clearly, \,$v^+_\bla\!\<\in\V^+\cap\DL$. Moreover,
\;$v^+_\bla\!=\thi^+_\bla(1)$\>, \;$v^+_\bla\!=\thi^=_\bla(Q_\bla)$ \,and
\vvn.2>
\;$v^+_\bla\!=\thi^-_\bla(Q_\bla)$, \,where \;$Q_\bla$ \,is given by \Ref{Ql}.
\vsk.3>
Denote \;$v^=_\bla\!=\thi^=_\bla(1)\in\DLe$ \,and
\;$v^-_\bla\!=\thi^-_\bla(1)\in\DL$\>. By Proposition \ref{thixi} we have
\vvn.4>
\begin{gather}
\label{v+xi}
v^+_\bla=\>\sum_{I\in\Il}\,\frac{Q(\zb_I)}{R(\zb_I)}\,\,\xi^+_I\,
=\>\sum_{I\in\Il}\,\frac{Q(\zb_I)}{R(\zb_I)}\,\,\xi^-_I\,,
\\[4pt]
v^=_\bla=\>\sum_{I\in\Il}\,\frac 1{R(\zb_I)}\,\,\xi^+_I\,,\qquad
v^-_\bla=\>\sum_{I\in\Il}\,\frac 1{R(\zb_I)}\,\,\xi^-_I\,.
\\[-14pt]
\notag
\end{gather}

\begin{lem}
\label{v+gl}
The vector \,$v^+_\bla$ \<belongs to the irreducible \,$\gln$-submodule
of \;$V\!$ of highest weight \,$(n,0\lc 0)$\>, generated by the vector
\,$\vone$.
\qed
\end{lem}

\begin{lem}[\cite{RTVZ}]
\label{v-gl}
For any \;$i\<\>,j$ \>such that \,$\la_i\ge\la_j$, we have
\;$e_{\ij}\>v^-_\bla\<=\<\>0$ \,and \;$e_{\ij}\>v^=_\bla\<=\<\>0$
with respect to the pointwise \,$\gln$-action on \;$V$.
\qed
\end{lem}

Notice that the functions \,$v^-_\bla$ and \,$v^=_\bla$ under certain
conditions become quantized conformal blocks and satisfy \qKZ/ equations
with respect to $\zzz$, see \cite{RTVZ}, cf.~\cite{V,RV,RSV}.

\vsk.2>
\section{Bethe subalgebras}
\label{Bethe}

\subsection{Yangian $\Yn$}
\label{sec yangian}

The Yangian $\Yn$ is the unital associative algebra with generators
\,$T_{\ij}\+s$ \,for \,$i,j=1\lc N$, \;$s\in\Z_{>0}$, \,subject to relations
\vvn.3>
\beq
\label{ijkl}
(u-v)\>\bigl[\<\>T_{\ij}(u)\>,T_{\kl}(v)\<\>\bigr]\>=\,
T_{\kj}(v)\>T_{\il}(u)-T_{\kj}(u)\>T_{\il}(v)\,,\qquad i,j,k,l=1\lc N\,,
\vv.3>
\eeq
where
\vvn-.7>
\be
T_{\ij}(u)=\dl_{\ij}+\sum_{s=1}^\infty\,T_{\ij}\+s\>u^{-s}\>.
\vv.2>
\ee
Let \,$T(u)=\sum_{\ij=1}^N E_{\ij}\ox T_{\ij}(u)$\>,
where $E_{\ij}\!\in\End(\C^N)$ is the image of \,$e_{\ij}\in\gln$.
Relations \Ref{ijkl} can be written as the equality of series with
coefficients in \,$\End(\C^N\!\ox\C^N)\ox\Yn$\,:
\vvn.3>
\beq
\label{RTT}
(u-v+P^{(1,2)})\,T^{(1)}(u)\>T^{(2)}(v)\,=\,
T^{(2)}(v)\>T^{(1)}(u)\,(u-v+P^{(1,2)})\,,
\vv.3>
\eeq
where \,$P^{(1,2)}$ is the permutation of the \,$\C^N\!$ factors,
\;$T^{(1)}(u)=\sum_{\ij=1}^N E_{\ij}\ox 1\ox T_{\ij}(u)$
\;and \;$T^{(2)}(u)=1\ox T(u)$\,.

\vsk.2>
The Yangian $\Yn$ is a Hopf algebra with the coproduct
\;$\Dl:\Yn\to\Yn\ox\Yn$ \,given by
\vvn.16>
\;$\Dl\bigl(T_{\ij}(u)\bigr)=\sum_{k=1}^N\,T_{\kj}(u)\ox T_{\ik}(u)$ \,for
\,$i,j=1\lc N$\>. The Yangian $\Yn$ contains \>$\Ugln$ \,as a Hopf subalgebra,
the embedding given by \,$e_{\ij}\mapsto T_{\ji}\+1$.

\vsk.2>
Notice that \,$\bigl[\<\>T_{\ij}\+1,T_{\kl}\+s\<\>\bigr]\>=\,
\dl_{\il}\>T_{\kj}\+s-\dl_{\jk}\>T_{\il}\+s$ \,for \,$i,j,k,l=1\lc N$,
\;$s\in\Z_{>0}$\,, \,which implies that the Yangian \>$\Yn$ is generated by
the elements \,$T_{\ii+1}\+1\>,\>T_{\ioi}\+1$, \;$i=1\lc N-1$, \,and
\,$T_{1,1}\+s$, \;$s>0$.

\vsk.2>
The assignment
\vvn-.2>
\beq
\label{varpi}
\varpi:T_{\ij}\+s\,\mapsto T_{\ji}\+s\,,\qquad i,j=1\lc N\,,\quad s>0\,,
\vv.1>
\eeq
defines a Hopf anti-automorphism of the Yangian \,$\Yn$.

\vsk.2>
More information on the Yangian $\Yn$ can be found in~\cite{M}. Notice that
the series \,$T_{\ij}(u)$ \>here corresponds to the series \,$T_{j,i}(u)$
\>in~\cite{M}.

\subsection{Bethe algebra}
\label{sec Bethe alg}
For \;$p=1\lc N$, \;$\ib=\{1\leq i_1<\dots<i_p\leq N\}$,
\;$\jb=\{1\leq j_1<\dots<j_p\leq N\}$, define
\vvn-.3>
\be
M_{\ijb}(u) =\sum_{\si\in S_p}(-1)^\si\,
T_{i_1,j_{\si(1)}}(u)\dots T_{i_p,j_{\si(p)}}(u-p+1)\,.
\ee
For \,$\ib=\{1\lc N\}$, \,the series \,$M_{\iib}(u)$ is called the quantum
determinant and denoted by $\qdet T(u)$. Its coefficients generate the center
of the Yangian \,$\Yn$.

\vsk.2>
We have
\vvn-.2>
\beq
\label{DlM}
\Dl\bigl(M_{\ijb}(u)\bigr)\,=\!
\sum_{\kb\>=\<\>\{1\le\>k_1<\<\dots<\>k_p\le N\}}\!\!
M_{\kjb}(u)\ox M_{\ikb}(u)\,,
\vv.2>
\eeq
see, for example, \cite[Proposition~1.11]{NT} or \cite[Lemma 4.3]{MTV1}.
Notice that the Yangian coproduct used here is opposite to that used in
\cite{NT}. We also have
\vvn.3>
\beq
\label{varpiM}
\varpi\bigl(M_{\ijb}(u)\bigr)\,=\,M_{\jib}(u)\,,
\vv.2>
\eeq
see, for example, \cite[Lemma 1.5]{NT}.

\goodbreak
\vsk.2>
For \,$\kk=(\kkk)\in(\C^*)^N$ and \,$p=1\lc N$, \,define
\beq
\label{Bp}
\Bk_p(u)\,=\!\sum_{\ib\>=\<\>\{1\leq i_1<\<\dots<\>i_p\leq N\}}\!\!
\kk_{i_1}\!\dots\kk_{i_p}\>M_{\iib}(u)\,=\,\si_p(\kkk)+
\sum_{s=1}^\infty\>\Bk_{p,s}\>u^{-s}\>,\kern-1em
\vv.2>
\eeq
where \,$\si_p$ is the \,$p$-th elementary symmetric function and
\,$\Bk_{p,s}\in\Yn$. Let \,$\Bck\!\subset\Yn$ \,be the unital subalgebra
generated by the elements \,$\Bk_{p,s}$\>, \,$p=1\lc N$, \,$s\ge0$.
It is easy to see that the subalgebra \,$\Bck\<$ does not change
if all $\kkk$ are multiplied simultaneously by the same number.
The algebra \,$\Bck\<$ is called a {\it Bethe subalgebra\/} of $\Yn$.

\begin{thm}[\cite{KS}]
\label{BY}
The subalgebra \,$\Bck\<$ is commutative and commutes with the subalgebra
$U(\h)\subset\Yn$.
\qed
\end{thm}

The series \,$\Bk_p(u)$ depends polynomially on \,$\kkk$. Let \,$\kk$
\,tend to infinity so that $\kk_{i+1}/\kk_i\to 0$ \,for all \,$i=1\lc N-1$,
\,and \,$\kk_N=1$. In this limit,
\vvn.3>
\beq
\label{Blim}
\Bk_p(u)\,=\,\kk_1\dots \kk_p\>\bigl(M_{\iib}(u)+o(1)\bigr)\,,
\qquad\ib=\{1\lc p\>\}\,.
\vv.3>
\eeq
Introduce the elements $\Bin_{p,s}\in\Yn$\>, \,$p=1\lc N$, \,$s\in\Z_{>0}$,
by the formula
\beq
\label{Bp8}
M_{\iib}(u)\,=\,1+\sum_{s=1}^\infty\, \Bin_{p,s}\,u^{-s}\,,
\qquad\ib=\{1\lc p\>\}\,.
\eeq
Let \,$\Bci\subset\Yn$ \,be the unital subalgebra generated by the elements
\,$B_{p,s}^\infty$\>, \,$p=1\lc N$, \,$s>0$. The algebra \,$\Bci\<$ is called
a {\it\GZ/ subalgebra\/} of $\Yn$. It has been studied in \cite{NT}.

\vsk.2>
The subalgebra \,$\Bci\<$ is commutative by Theorem \ref{BY}.
\vvn.06>
Since \,$\Bin_{p,1}=T_{1,1}\+1\<\lsym+T_{\pp}\+1$ \>for \,$p=1\lc N$,
the subalgebra \,$\Bci\<$ contains the subalgebra $U(\h)\subset\Yn$.

\vsk.3>
Assume that \>$\kkk$ are distinct numbers. Define the elements
\,$\Sk_{i,s}\in\Bck$ \>for \,$i=1\lc N$, \,$s\in\Z_{>0}$\>, \,by the rule
\vvn-.2>
\beq
\label{S}
\Sk_{i,s}\,=\,\sum_{p=1}^N\,(-1)^{p-1}\,\Bk_{p,s}\,\kk_i^{N-p-1}
\prod_{\satop{j=1}{j\ne i}}^N\,\frac1{\kk_i-\kk_j}\;,
\vv->
\eeq
so that
\be
\prod_{i=1}^N\,\frac1{x-\kk_i}\;\sum_{p=1}^N\,\sum_{s=1}^\infty\,
(-1)^p\>\Bk_p(u)\,x^{N-p}\>=\,1\>-\sum_{i=1}^N\,\sum_{s=1}^\infty\,
\frac{\kk_i}{x-\kk_i}\;\Sk_{i,s}\,u^{-s}\>.\kern-1em
\vv.4>
\ee
In particular,
\vvn-.3>
\beq
\label{S12}
\Sk_{i,1}\>=\>T_{\ii}\+1\,,\kern3em
\Sk_{i,2}\>=\>T_{\ii}\+2-\>
\sum_{\satop{j=1}{j\ne i}}^N\,\frac{\kk_j}{\kk_i-\kk_j}\,\bigl(\<\>
T_{\ii}\+1\>T_{\jj}\+1\<-T_{\ij}\+1\>T_{\ji}\+1\<+T_{\jj}\+1\bigr)\,.\kern-1em
\eeq

\begin{lem}
For distinct numbers \,$\kkk$, the subalgebra \,$\Bck\!$ contains
the subalgebra $U(\h)\subset\Yn$.
\qed
\end{lem}

\begin{lem}
\label{S8}
Let \;$\kk_{i+1}/\kk_i\to 0$ \,for all \,$i=1\lc N-1$.
\vvn.06>
Then \,$\Sk_{i,s}\to\Bin_{i,\<\>s}-\Bin_{i-1,\<\>s}$ \,for \,$i=1\lc N$,
\,where \,$\Bin_{0,\<\>s}\<=0$.
\end{lem}
\begin{proof}
Write formula \Ref{S} as
\vvn-.1>
\be
\Sk_{i,s}\,=\,\sum_{p=1}^N\,(-1)^{p-i}\,h^{s-1}\>\Bk_{p,s}\;
\frac{\kk_i^{\<\>i-p-1}}{\kk_1\dots\kk_{i-1}}\;
\prod_{j=1}^{i-1}\,\frac1{1-\kk_i/\kk_j}\;
\prod_{j=i+1}^N\,\frac1{1-\kk_j/\kk_i}\;.
\vv.2>
\ee
Now the claim follows from formulae \Ref{Blim}, \Ref{Bp8}.
\end{proof}

\begin{lem}
\label{varpiB}
For any \,$X\!\in\Bck$, we have \,$\varpi(X)=X$, and
for any \,$X\!\in\Bci$, we have \,$\varpi(X)=X$.
\end{lem}
\begin{proof}
By formulae \Ref{varpiM}\,--\,\Ref{Bp8}, see also
\cite[Proposition~4.11]{MTV1}, for any \,$p=1\lc N$ and \,$s>0$, we have
\;$\varpi(\Bk_{p,s})=\Bk_{p,s}$ \>and \;$\varpi(\Bin_{p,s})=\Bin_{p,s}$. Since
the algebras \,$\Bck$ and \,$\Bci$ are commutative, the statement follows.
\end{proof}

As a subalgebra of $\Yn$, the Bethe algebra \,$\Bck\<$ acts on any $\Yn$-module
$M$. Since \,$\Bck\<$ commutes with $U(\h)$, it preserves the weight subspaces
$M_\bla$. If $L\subset M$ is a \,$\Bck$-invariant subspace, then the image of
\,$\Bck\<$ in $\End(L)$ will be called the Bethe algebra of $L$ and denoted by
$\Bck(L)$.

\subsection{Algebras \,$\ty\>,\;\tbk\!,\;\tbi\<$.}
Let \,$\ty$ be the subalgebra of \,$\Yn\ox\C[h]$ generated over \,$\C\>$
by \,$\C[h]$ and the elements \,$h^{s-1}\>T_{\ij}\+s$ \,for \,$i,j=1\lc N$,
\;$s>0$. Equivalently, the subalgebra \,$\ty$ \,is generated over \,$\C\>$ by
\,$\C[h]$ and the elements \,$T_{\ii+1}\+1\>,\>T_{\ioi}\+1$,
\;$i=1\lc N-1$, \,and \,$h^{s-1}\>T_{1,1}\+s$, \;$s>0$.

\vsk.1>
Let \,$\tbk$ (resp.~\,$\tbi$) be the subalgebra of \,$\Yn\ox\C[h]$
generated over \,$\C\>$ by \,$\C[h]$ and the elements \,$h^{s-1}\>\Bk_{p,s}$
(resp.~\,\,$h^{s-1}\>\Bin_{p,s}$) \,for \,$p=1\lc N$, \;$s>0$. It is easy
to see that \,$\tbk$ and \,$\tbi$ are subalgebras of \,$\ty$.

\vsk.2>
Introduce the series \,$A_1(u)\lc A_N(u)$, \,$E_1(u)\lc E_{N-1}(u)$,
\,$F_1(u)\lc F_{N-1}(u)$:
\vvn.1>
\begin{gather}
\label{A}
A_p(u)\,=\,M_{\iib}(u/h)\,=\,
1+\sum_{s=1}^\infty\,h^s\>\Bin_{p,s}\,u^{-s}\,,
\\[3pt]
E_p(u)\,=\,h^{-1}M_{\jib}(u/h)\>\bigl(M_{\iib}(u/h)\bigr)^{-1}\,=\,
\sum_{s=1}^\infty\,h^{s-1}\>E_{p,s}\,u^{-s}\,,
\label{EF}
\\[3pt]
F_p(u)\,=\,h^{-1}\bigl(M_{\iib}(u/h)\bigr)^{-1}M_{\ijb}(u/h)\,=\,
\sum_{s=1}^\infty\,h^{s-1}\>F_{p,s}\,u^{-s}\,,
\notag
\\[-20pt]
\notag
\end{gather}
where \,$\ib=\{1\lc p\>\}$\>, \,$\jb=\{1\lc p-1,p+1\>\}$\>.
Observe that \,$E_{p,1}=T_{\pop}\+1$\>, \,$F_{p,1}=T_{\ppo}\+1$
\,and \,$\Bin_{1,s}=T_{1,1}\+s$, so the coefficients of the series
\,$E_p(u)$, \,$F_p(u)$ and \,$h^{-1}(A_p(u)-1\bigr)$ together with \,$\C[h]$
generate \,$\ty$.

\vsk.2>
In what follows we will describe actions of the algebra \,$\ty$ by using series
\Ref{A}, \Ref{EF}.

\vsk.2>
The quotient algebra \,$\ty/(h-1)\>\ty$ is canonically isomorphic to \,$\Yn$.
\vvn.14>
Also, the quotient algebra \,$\ty\big/h\>\ty$ is naturally isomorphic
to the universal enveloping algebra \,$U(\gln[\<\>t\<\>])$, the element
\,$h^{s-1}T_{\ij}\+s$ \,projecting to \,$e_{\ji}\ox t^{s-1}$.

\subsection{Dynamical Hamiltonians and dynamical connection}
\label{sec dyn hams}
Assume that \,$\kkk$ \>are distinct numbers. Define the elements
\,$\Xk_1\lc\Xk_N\in\tbk$ by the rule
\vvn.3>
\begin{align}
\label{X}
\Xk_i\, &{}=\,h\>\Sk_{i,2}-\frac h2\,T_{\ii}\+1\>\bigl(\<\>T_{\ii}\+1-1\bigr)
+\>h\,\sum_{\satop{j=1}{j\ne i}}^N\,\frac{\kk_j}{\kk_i-\kk_j}\,
T_{\ii}\+1\>T_{\jj}\+1
\\[4pt]
&{}=\,h\>T_{\ii}\+2-\>\frac h2\,T_{\ii}\+1\>\bigl(\<\>T_{\ii}\+1-1\bigr)+\>
h\>\sum_{\satop{j=1}{j\ne i}}^N\,\frac{\kk_j}{\kk_i-\kk_j}\,G_{\ij}\,,
\notag
\\[-17pt]
\notag
\end{align}
where \,$\Sk_{i,2}$ \,is given by \Ref{S12} and
\vvn-.06>
$G_{\ij}=\,T_{\ij}\+1\>T_{\ji}\+1\<-T_{\jj}\+1=\,
T_{\ji}\+1\>T_{\ij}\+1\<-T_{\ii}\+1$.
Here the second formula is obtained from the first one by using commutation
\vvn-.06>
relations of the elements \,$T_{\kl}\+1$. Notice that \,$T_{\ii}\+1$ and
\,$G_{\ij}$ belong to the subalgebra \,$\Ugln$ embedded in \,$\ty$,
\vvn.1>
\be
T_{\ii}\+1=\>e_{\ii}\,,\qquad
G_{\ij}\>=\,e_{\ij}\>e_{\ji}\<-e_{\ii}\>=\,e_{\ji}\>e_{\ij}\<-e_{\jj}\>.
\ee

\vsk.2>
Taking the limit $\kk_{i+1}/\kk_i\to0$ \,for all \,$i=1\lc N-1$,
we define the elements \,$\Xin_1\lc\Xin_N\in\tbi$,
\vvn-.1>
\begin{align}
\label{X8}
\Xin_i &{}=\,h\>\Bin_{i,2}-h\>\Bin_{i-1,2}
-\>\frac h2\,e_{\ii}\>\bigl(\<\>e_{\ii}-1\bigr)-\>
h\>e_{\ii}\>\bigl(\<\>e_{1,1}\<\lsym+e_{i-1,\<\>i-1}\bigr)
%% -\>\frac h2\,T_{\ii}\+1\>\bigl(\<\>T_{\ii}\+1-1\bigr)-\>
%% h\>T_{\ii}\+1\>\bigl(\<\>T_{1,1}\+1\<\lsym+T_{i-1,\<\>i-1}\+1\bigr)
\\[4pt]
\notag
&{}=\,h\>T_{\ii}\+2-\>\frac h2\,e_{\ii}\>\bigl(\<\>e_{\ii}-1\bigr)
-h\>(G_{i,1}\<\lsym+\<\>G_{\ii-1})\,.
%% &{}=\,h\>T_{\ii}\+2-\>\frac h2\,T_{\ii}\+1\>\bigl(\<\>T_{\ii}\+1-1\bigr)
%% -h\>(G_{i,1}\<\lsym+\<\>G_{\ii-1})\,.
\\[-14pt]
\notag
\end{align}
cf.~Lemma \ref{S8}. We will call the elements
\vvn.1>
\,$\Xk_i\<,\> \Xin_i$, \,$i=1\lc N$, the {\it dynamical Hamiltonians\/}.
Observe that
\vvn-.3>
\beq
\label{XX}
\Xk_i\,=\,\Xin_i+\>h\>\sum_{j=1}^{i-1}\,\frac{\kk_i}{\kk_i-\kk_j}\,G_{\ij}
\>+\>h\!\sum_{j=i+1}^n\frac{\kk_j}{\kk_i-\kk_j}\,G_{\ij}\,.
\eeq

\vsk.2>
Let \;$G^{\>+}_{\ij}=\>G_{\ij}-e_{\ii}\>e_{\jj}$\>.
Given \,$\bla=(\la_1\lc\la_N)$, \,set
\vvn.1>
\;$G^{\>-}_{\bla\<\>,\<\>\ij}\>=\,e_{\ji}\>e_{\ij}$ \,for \,$\la_i\ge\la_j$\>.
\,and
\;$G^{\>-}_{\bla\<\>,\<\>\ij}=\>e_{\ij}\>e_{\ji}$ \,for \,$\la_i<\la_j$\>.
Define the elements
\,$\Xkp_1\,\lc\Xkp_N\<,\;\Xkm_{\bla\<\>,1}\,\lc\Xkm_{\bla\<\>,\<\>N}\<\in\tbk$,
\vvn.2>
\beq
\label{Xkp}
\Xkp_i\!=\,\Xin_i+
\>h\>\sum_{j=1}^{i-1}\,\frac{\kk_i}{\kk_i-\kk_j}\,G^{\>+}_{\ij}\>+
\>h\!\sum_{j=i+1}^n\frac{\kk_j}{\kk_i-\kk_j}\,G^{\>+}_{\ij}\,,
% \Xk_i-\>
% h\>\sum_{j=1}^{i-1}\,\frac{\kk_i}{\kk_i-\kk_j}\,e_{\ii}\>e_{\jj}\>-\>
% h\!\sum_{j=i+1}^n\frac{\kk_j}{\kk_i-\kk_j}\,e_{\ii}\>e_{\jj}\,,
\eeq
\beq
\label{Xkm}
\Xkm_{\bla\<\>,\<\>i}\>=\,\Xin_i+
\>h\>\sum_{j=1}^{i-1}\,\frac{\kk_i}{\kk_i-\kk_j}\,G^{\>-}_{\bla\<\>,\<\>\ij}\>+
\>h\!\sum_{j=i+1}^n\frac{\kk_j}{\kk_i-\kk_j}\,G^{\>-}_{\bla\<\>,\<\>\ij}\,.
% \Xk_i+\>
% h\>\sum_{j=1}^{i-1}\,\frac{\kk_i}{\kk_i-\kk_j}\,e_{\ii}\>+\>
% h\!\sum_{j=i+1}^n\frac{\kk_j}{\kk_i-\kk_j}\,e_{\jj}
% \\[4pt]
% &{}=\,\Xin_i+\>
% h\>\sum_{j=1}^{i-1}\,\frac{\kk_i}{\kk_i-\kk_j}\,e_{\ij}\>e_{\ji}\>+\>
% h\!\sum_{j=i+1}^n\frac{\kk_j}{\kk_i-\kk_j}\,e_{\ji}\>e_{\ij}\,.
\eeq

\vsk.4>
It is straightforward to verify that for any nonzero complex number \,$\kp$\>,
the formal differential operators
\vvn-.4>
\beq
\label{dyneq}
\nabla_i\>=\,\kp\,\ddk_i\>-\>\Xk_i,\qquad i=1\lc N,
\vv.2>
\eeq
pairwise commute and, hence, define a flat connection for any \,$\ty$-module.

\begin{lem}
\label{flat+-}
The connections \;$\nabla^+\<$ and \;\;$\nabla^-_{\<\bla}\<$ defined by
\vvn.2>
\beq
\label{nablapm}
\nabla^+_i=\,\kp\,\ddk_i\>-\>\Xkp_i,\qquad
\nabla^-_{\<\bla\<\>,\<\>i}=\,\kp\,\ddk_i\>-\>\Xkm_{\bla\<\>,\<\>i}\,,
\vv-.1>
\eeq
$i=1\lc N$, are flat for any \,$\kp$.
\end{lem}
\begin{proof}
The connections \;$\nabla^+\<$ and \;\;$\nabla^-_{\<\bla}\<$ are gauge
equivalent to connection \Ref{dyneq},
\vvn.4>
\be
\nabla^+_i=\,(\Tht^+)^{-1}\;\nabla_i\;\Tht^+,\qquad
\Tht^+=\prod_{1\le i<j\le N}\!(1-\kk_j/\kk_i)
%% ^{-h\>e_{i,i}\<\>e_{j,j}/\<\kp}\,,
^{\<\>h\>e_{i,i}\<\>e_{j,j}/\<\kp}\,,
\vv.2>
\ee
\beq
\label{gauge}
\nabla^-_{\<\bla\<\>,\<\>i}=\,(\Tht^-_\bla)^{-1}\;\nabla_i\;\Tht^-_\bla\,,
\qquad \Tht^-_\bla\>=\prod_{1\le i<j\le N}\!(1-\kk_j/\kk_i)
%% ^{\<\>h\>\eps_{\bla\<\>,\ij}/\<\kp}\,,
^{-h\>\eps_{\bla\<\>,\ij}/\<\kp}\,,
\vv.2>
\eeq
where \;$\eps_{\bla\<\>,\ij}=\>e_{\jj}$ \,for \,$\la_i\ge\la_j$\>,
\,and \;$\eps_{\bla\<\>,\ij}=\>e_{\ii}$ \,for \,$\la_i<\la_j$\>.
\end{proof}

\vsk.3>
Connection \Ref{dyneq} was introduced in \cite{TV1}, see also \cite{MTV1}, and
is called the {\it trigonometric dynamical connection\/}. The connection is
defined over \;$\C^N$ with coordinates \,$\kkk$, and has singularities
at the union of the diagonals \,$\kk_i=\kk_j$. In the case of a tensor product
of evaluation \,$\ty$-modules, the trigonometric dynamical connection commutes
with the associated \qKZ/ difference connection, see \cite{TV1}. Under
the \,$(\gln,\>\gl_{\>n})$ \,duality, the trigonometric dynamical connection
and the associated \qKZ/ difference connection are respectively identified
with the trigonometric \KZ/ connection and the dynamical difference connection,
see \cite{TV1}. Notice that the trigonometric dynamical connection here differs
from that in \cite{TV1} by a gauge transformation.

\section{Yangian actions}
\label{Yaction}
\subsection{$\ty\<\>$-actions}
\label{tyactions}
Let $\Czh$ act on $\Vz$
\vvn.2>
by multiplication on the second factor. Set
\begin{align}
\label{Lpm}
L^+(u)\,&{}=\,(u-z_n+h\<\>P^{(0,n)})\dots(u-z_1+h\<\>P^{(0,1)})\,,
\\[4pt]
L^-(u)\,&{}=\,(u-z_1+h\<\>P^{(0,1)})\dots(u-z_n+h\<\>P^{(0,n)})\,,
\notag
\\[-15pt]
\notag
\end{align}
where the factors of \,$\C^N\!\ox V$ are labeled by \,$0,1\lc n$.
\vvn.1>
Both \,$L^+(u)$ and \,$L^-(u)$ are polynomials in \,$u,\zzz,h$ \,with values
in \,$\End(\C^N\!\ox V)$. We consider $L^\pm(u)$ as $N\!\times\!N$ matrices
with \,$\End(V)\ox\C[u,\zzz, h]\>$-valued entries \,$L^\pm_{\ij}(u)$.

\begin{prop}
The assignments
\vvn-.2>
\beq
\label{pho}
\pho^\pm\bigl(T_{\ij}(u/h)\bigr)\,=\,
L^\pm_{\ij}(u)\,\prod_{a=1}^n\,(u-z_a)^{-1}
\vv-.1>
\eeq
define the actions \,$\pho^+\!$ \,and \,$\pho^-\!$ of the algebra \,$\ty$
\vvn.2>
on \,$\Vz$\,. Here the right-hand side of \Ref{pho} is a series
in \,$u^{-1}$ with coefficients in \,$\End(V)\ox\Czh$\>.
\end{prop}
\begin{proof}
The claim follows from relations \Ref{RTT} and the Yang-Baxter equation
\vvn.3>
\begin{align}
\label{YB}
(u-v+h\<\>P^{(1,2)})\> &(u+h\<\>P^{(1,3)})\>(v+h\<\>P^{(2,3)})\,=\,
\\[3pt]
{}=\,{}&(v+h\<\>P^{(2,3)})\>(u+h\<\>P^{(1,3)})\>(u-v+h\<\>P^{(1,2)})\,.
\notag
\\[-36pt]
\notag
\end{align}
\vv.2>
\end{proof}

For both actions \,$\pho^\pm$, the subalgebra $U(\gln)\subset\ty$ acts on $\Vz$
in the standard way: any element \,$x\in\gln$ \,acts as $x^{(1)}\lsym+x^{(n)}$.
The actions \,$\pho^\pm$ clearly commute with the $\Czh$-action.

\begin{lem}
\label{SY}
The Yangian actions \,$\pho^\pm\!$ are contravariantly related through
the Shapovalov pairing \Ref{Shap pair}:
\vvn-.2>
\be
\Sc\bigl(\pho^+(X)\>f,g)\,=\,\Sc\bigl(f,\pho^-(\varpi(X))\>g\<\>\bigr)\,,
\vv.2>
\ee
for any \,$X\in\ty$ \,and any \,$V\!$-valued functions \,$f,g$
\,of \,$\zzz,h$.
\end{lem}
\begin{proof}
The claim follows from formulae \Ref{varpi}, \Ref{Lpm} and \Ref{pho}.
\end{proof}

\begin{cor}
\label{SB}
For any \,$X\!\in\tbk$ or \>$X\!\in\tbi$, \,and any \,$V\!$-valued
functions \,$f,g$ \,of \,$\zzz,h$\>, we have
\vvn-.1>
\be
\Sc\bigl(\pho^+(X)\>f,g)\,=\,
\Sc\bigl(f,\pho^-(X)\>g\<\>\bigr)\,
%\qquad\Sc\bigl(\pho^+(Y)\>f,g)\,=\,\Sc\bigl(f,\pho^-(Y)\>g\<\>\bigr)\,.
\vv.2>
\ee
\end{cor}
\begin{proof}
The claim follows from Lemma \ref{varpiB}.
\end{proof}

\begin{lem}
\label{pm & Yan}
The Yangian action \,$\pho^+\!$ (\>resp.~$\pho^-$) on \,$V\!$-valued functions
of \,$\zzz,h$ commutes with the $S_n^+$-action \Ref{Sn+}
\,(\>resp.~the $S_n^-$-action \Ref{Sn-}\>).
\end{lem}
\begin{proof}
The claim follows from the Yang-Baxter equation \Ref{YB} and the fact that
the actions \,$\pho^\pm$ commute with multiplication by functions of
\,$\zzz\<\>,\>h$.
\end{proof}

By Lemma \ref{pm & Yan}, the action \,$\pho^+$ makes the spaces \,$\V^+$ and
\,$\DVe$ into $\ty$-modules, and the action \,$\pho^-$ on $\Vz$ makes the space
\,$\DV$ into a $\ty$-modules.

\vsk.3>
Let \,$\Pi\<\in\End(V)$ \,be the following linear map,
\,${\Pi\>(\<\>g_1\lox g_n)=\<\>g_n\lox g_1}$ \,for any \,$g_1\lc g_n\in\C^N$.
For any \,$V\!$-valued function \,$f(\zzz,h)$ set
\vvn.3>
\beq
\label{Pit}
\Pit\>f(\zzz,h)\,=\,\Pi\>\bigl(f(\zzzn,h)\bigr)\,.
\vv.1>
\eeq

\begin{lem}
\label{=-}
The map \>\Ref{Pit} induces an isomorphism \;$\DVe\!\to\DV\!$ of
\;$\ty$-modules. Moreover, \;$\Pit\,v^=_\bla\<=v^-_\bla$ \,for any $\bla$.
\end{lem}
\begin{proof}
We have \;$\Pit\,\st_i^+\<\<=\>\st_{n-i}^-\,\Pit$ \,for every \,$i=1\lc n-1$,
where the operators \;$\st_i^\pm$ are given by \Ref{stilde}.
Thus by Lemma \ref{S2}, \,$\Pit$ \,induces a vector space isomorphism
\;$\DVe\!\to\DV\!$. Since
\vvn.2>
\beq
\label{pho+-}
\Pit\;\pho^+(X)\,=\,\pho^-(X)\;\Pit
\vv.2>
\eeq
for every \,$X\!\in\ty$, see \Ref{Lpm}, \Ref{pho}, \,$\Pit$ \,induces an
\vvn.06>
isomorphism of \>$\ty$-modules. The equality \;$\Pit\,v^=_\bla\<=v^-_\bla$
follows from formulae \Ref{thi=}, \Ref{thi-}.
\end{proof}

\begin{prop}
\label{V+v1}
The \,$\ty$-module \;$\V^+\!$ is generated by the vector \,$\vone$.
\end{prop}
\begin{proof}
For every \,$\bla=(\la_1\lc\la_N)$, the subspace \,$\Vl$ is generated by
\vvn.1>
the \;$\tbi\<$-action on \;$v^+_\bla$, see Theorem \ref{thm main+}
in Section \ref{secHla}. Since
\vvn.4>
\be
v^+_\bla\,=\,\pho^+\Bigl(\<\bigl(T_{2,1}\+1\bigr)^{\la_1-\la_2}\!\dots\,
\bigl(T_{N-1,1}\+1\bigr)^{\la_{N-1}-\la_N}\bigl(T_{N,1}\+1\bigr)^{\la_N}\Bigr)
\;\frac{\vone}
{(\la_1-\la_2)\<\>!\dots(\la_{N-1}-\la_N)\<\>!\,\la_N!}\;,
\vv-.1>
\ee
the \,$\ty$-module \;$\V^+$ is generated by \,$\vone$.
\end{proof}

Notice that the \,$\ty$-modules \;$\V^+$ is not isomorphic to
\;$\DVe\<\simeq\DV$ because for any homomorphism \;$\V^+\!\to\DVe$ of
\vvn.1>
\,$\ty$-modules, the image of the vector \,$\vone$ belongs to
\,$(\vone)\ox\CZH$ by the weight reasoning and generates a proper
\,$\ty$-submodule of \;$\DVe$.

\vsk.5>
Write \,$n=kN\<+l$ \,for \,$k,l\in\Z_{\ge0}$, \;$l<N$.
Set \,$\mub=(k+1\lc k+1,k\lc k)$ with \,$l$ \,parts equal to \,$k+1$
\,and $N-l$ \,parts equal to $k$.

\begin{prop}
\label{V=v1}
The \,$\ty$-module \;$\DV$ (\>resp.~\>$\DVe$) \,is generated by
the function \,$v^-_\mub$ \,(resp.~\>$v^=_\mub$\>).
\end{prop}
\begin{proof}
For every \,$\bla=(\la_1\lc\la_N)$, the subspace \,$\DL$ is generated by
\vvn.06>
the \;$\tbi\<$-action on \;$v^-_\bla$, see Theorem \ref{thm main-} in
Section \ref{secHla}. By formula \Ref{Exi} below, we have
\vvn.3>
\beq
\label{Ev-}
\pho^-(h^{\la_p-\la_{p+1}+1}E_{p,\<\>\la_p-\la_{p+1}+2})\,v^-_\bla\>=\,
(-1)^{\la_p}\,v^-_{\bla\<\>-\aal_p}\,,
\vv.3>
\eeq
if \,$\la_p-\la_{p+1}\ge-\>1$. Here \,$\aal_p=(0\lc0,1,\<-\>1,0\lc0)$, with
\,$p-1$ first zeros. Similarly, by formula \Ref{Fxi} below, we have
\vvn.1>
\be
\pho^-(h^{\la_{p+1}-\la_p+1}F_{p,\<\>\la_{p+1}-\la_p+2})\,v^-_\bla\>=\,
(-1)^{\la_p-1}\,v^-_{\bla\<\>+\aal_p}\,,
\ee
if \,$\la_{p+1}-\la_p\ge-\>1$. Thus every \,$v^-_\bla$ \,can be obtained from
\,$v^-_\mub$ by the \>$\ty$-action. This proves the statement for \;$\DV$.
The proof for \;$\DVe$ is similar.
\end{proof}

\begin{proof}[Proof of Lemma \ref{surj}]
Clearly, if the statement holds for \,$\bla=(\la_1\lc\la_N)$, then it holds
for \,$\bla^\si\!=(\la_{\si(1)}\lc\la_{\si(N)})$ \,for any \,$\si\in S_N$.
So it suffices to prove the statement assuming that \,$\la_1\lsym\ge\la_N$.
Then by \Ref{Ev-}, there is \,$X\!\in\ty$ such that
\,$\pho^-(X)\>v^-_\bla\<=\alb\>\vone$. Thus by Lemma \ref{SY},
\vvn.3>
\be
\Sc\bigl(\pho^+(X)\>\vone\>,v^-_\bla)\,=\,
\Sc\bigl(\vone\>,\pho^-(X)\>v^-_\bla\bigr)\,=\,1\,.
\vvn-1.3>
\ee
\vv>
\end{proof}

\begin{prop}
\label{VBa}
The Bethe algebras \;$\tbk(\V^+)$, \;$\tbk(\DVe)$, and \;$\tbk(\DV)$ are
isomorphic. Similarly, the Bethe algebras \;$\tbi(\V^+)$, \;$\tbi(\DVe)$,
and \;$\tbi(\DV)$ are isomorphic.
\end{prop}
\begin{proof}
The statement follows from Corollary \ref{SB} and Lemma \ref{Sfg}.
For the pairs \;$\tbk(\DVe)$, \;$\tbk(\DV)$ and \;$\tbi(\DVe)$,
\;$\tbi(\DV)$, the algebras are isomorphic by Lemma \ref{=-} as well.
\end{proof}

\subsection{$\Yn\<\>$-actions}
\label{Ynactions}
Let \,$\phoo^\pm:\Yn\to\End(V)$ be the homomorphisms obtained from \,$\pho^\pm$
by taking \,$z_1\lsym=z_n=0$ \>and \>$h=1$\>. They make the space \>$V$ into
$\Yn$-modules, respectively denoted by \>$\Vy\pm$. The following statement is
well known.

\begin{prop}
\label{V0}
The $\Yn$-modules \,$\Vy\pm\<$ are irreducible and isomorphic.
The isomorphism \,$\Vy+\!\to\Vy-\<$ is given by the map
\,$x_1\<\lox x_n\mapsto\>x_n\<\lox x_1$ \,for any $\xxx\in\C^N$.
The modules \,$\Vy\pm\<$ are contravariantly dual,
\vvn.2>
\be
\Sc\bigl(\phoo^+(X)\>f,g)\,=\,\Sc\bigl(f,\phoo^-(\varpi(X))\>g\<\>\bigr)\,,
\vv.2>
\ee
for any \,$X\in\Yn$ \,and \,$f,g\in V$\>.
\qed
\end{prop}

Let $\Jc\subset\CZH$ be the ideal generated by the relations \>$h=1$ \>and
\;$\si_i(\zzz)=0$ \,for all \,$i=1\lc n$, where \;$\si_i$ is the \>$i$-th
elementary symmetric function. The quotient
\,$\Vh=\Vz\>/\Jc(\Vz)\simeq V\<\ox(\Czh\>/\Jc)$ is a complex vector space
of dimension \,$n!\,N^n$. The actions \,$\pho^\pm$ make it into $\Yn$-modules
respectively denoted by \,$\Vyh\pm$.

\begin{lem}
\label{lem01}
The evaluation assignment
\vvn.3>
\beq
\label{01}
f(\zzz,h)\,\mapsto\,f(0\lc 0,1)
\vv.3>
\eeq
defines homomorphisms \;$\Vyh+\!\to\Vy+$ and \;$\Vyh-\!\to\Vy-$ of
\,$\Yn$-modules.
\qed
\end{lem}

Notice that the \>$\Yn$-modules \,$\Vyh\pm$ are respectively isomorphic to
\vvn.16>
the \>$\Yn$-modules \,$\Vy\pm\<\ox(\Czh\>/\Jc)$ where the second factor
\,$\Czh\>/\Jc$ is the trivial \>$\Yn$-module with all the generators
\,$T_{\ij}\+s$ acting by zero. However, the isomorphisms are not given
by the identity operator on \,$V\<\ox(\Czh\>/\Jc)$.

\vsk.3>
The quotients \,$\V^+\</\Jc\V^+$, \,$\DVe\</\Jc\DVe$ and \,$\DV\</\Jc\DV$
are complex vector spaces of dimension $N^n$. The \,$\ty$-module structure
on the respective spaces \,$\V^+$, \,$\DVe$ $\DV$ make \,$\V^+\</\Jc\V^+$,
\,$\DVe\</\Jc\DVe$ and \,$\DV\</\Jc\DV$ into \>$\Yn$-modules.

\begin{thm}
The assignment \,\Ref{01} defines isomorphisms
\vvn.2>
\be
\eta^+\!:\V^+\</\Jc\V^+\to\,\Vy+,\qquad \eta^=\!:\DVe\</\Jc\DVe\to\,\Vy+,
\qquad \eta^-\!:\DV\</\Jc\DV\to\,\Vy-
\vv.1>
\ee
of \,$\Yn$-modules.
\end{thm}
\begin{proof}
Clearly, the maps \,$\eta^+$, \,$\eta^=$ and \,$\eta^-$ are nonzero
\vv.1>
homomorphisms of \>$\Yn$-modules, see Lemma~\ref{lem01}. Since
the \>$\Yn$-modules \,$\Vy\pm$ are irreducible and all the spaces
\vv.1>
\,$\V^+\</\Jc\V^+$, \,$\DVe\</\Jc\DVe$, \,$\DV\</\Jc\DV$ and \,$\Vy\pm$ have
\vv.1>
the same dimension, the maps \,$\eta^+$, \,$\eta^=$ and \,$\eta^-$ are
isomorphisms.
\end{proof}

\begin{cor}
\label{nondegS}
The Shapovalov pairing
\vvn.3>
\beq
\label{SJ}
\Sc:\V^+\</\Jc\V^+\!\ox\DV\</\Jc\DV\to\,(\CZH)/\Jc\simeq\C
\vv.3>
\eeq
induced by pairing \Ref{s+-} is nondegenerate. The \>$\Yn$-modules
\,$\V^+\</\Jc\V^+$ and \,$\DV\</\Jc\DV$ are contravariantly dual:
\vvn-.2>
\be
\Sc\bigl(\pho^+(X)\>f,g)\,=\,\Sc\bigl(f,\pho^-(\varpi(X))\>g\<\>\bigr)\,,
\vv.3>
\ee
for any \,$X\in\Yn$ \,and \,$f\<\in\V^+\</\Jc\V^+$, $g\in\DV\</\Jc\DV$.
\qed
\end{cor}

Alternatively, the nondegeneracy of pairing \Ref{SJ} follows from
Lemma \ref{lem Shap pm}, Proposition \ref{SP} in Section~\ref{sec EQ}
and the nondegeneracy of the Poincare pairing on $\Fla$.

\begin{cor}
\label{nondegS=}
The Shapovalov pairing
\vvn.3>
\be
\Sc:\DVe\</\Jc\DVe\<\<\ox\DV\</\Jc\DV\to\,(Z^{-1}\>\CZH)/\Jc\simeq\C
\vv.3>
\ee
induced by pairing \Ref{s=-} is nondegenerate. The \>$\Yn$-modules
\,$\DVe\</\Jc\DVe$ and\\ \,$\DV\</\Jc\DV$ are contravariantly dual:
\vvn.2>
\be
\Sc\bigl(\pho^+(X)\>f,g)\,=\,\Sc\bigl(f,\pho^-(\varpi(X))\>g\<\>\bigr)\,,
\vv.3>
\ee
for any \,$X\in\Yn$ \,and \,$f\<\in\DVe\</\Jc\DVe$, $g\in\DV\</\Jc\DV$.
\qed
\end{cor}

\subsection{Yangian actions on \;$\xi_I^\pm$.}
\label{sec eigenvectors}
Denote
\vvn.3>
\beq
\label{AEF}
A_p^\pm(u)\>=\>\pho^\pm\bigl(A_p(u)\bigr)\,,\qquad
E_p^\pm(u)\>=\>\pho^\pm\bigl(E_p(u)\bigr)\,,\qquad
F_p^\pm(u)\>=\>\pho^\pm\bigl(F_p(u)\bigr)\,.
\vv.2>
\eeq

\begin{lem}
\label{lem det}
The series \,$A_N^\pm(u)$ act on $\Vz$ as the operator of multiplication by
\;$\prod_{j=1}^n\,\bigl(1+h\>(u-z_j)^{-1}\bigr)$\,.
\end{lem}
\begin{proof}
For \,$n=1$, \,the proof is straightforward. For \,$n>1$, the claim follows
from the coproduct formula~\Ref{DlM} and the \,$n=1$ \,case.
% see also Lemma~2.1 in \cite{NT}.
\end{proof}

\begin{cor}
\label{cor sym z}
Each of the images \,$\pho^\pm(\tbi)$ contains the subalgebra of
multiplication operators by symmetric polynomials in variables \,$\zzz$.
\end{cor}
\begin{proof}
The elements \,$\pho^\pm(h^{s-1}\Bin_{N,s})\in\End(V)\ox\Czh$
\vv.1>
have the form \,$\sum_{a=1}^nz_a^{s-1}+\alb h\>g_{p,s}^\pm$\>, \,where
\,$g_{p,s}^\pm\in\End(V)\ox\Czh$ are symmetric in \,$\zzz$ and have
homo\-ge\-neous degree \,$s-2$. The claim follows.
\end{proof}

\begin{thm}
\label{AEFxi}
For each choice of the sign, $+$ or $-$, we have
\vvn.1>
\begin{gather}
\label{Axi}
A_p^\pm(u)\,\xi_I^\pm\,=\,\xi_I^\pm\;
\prod_{a=1}^p\,\prod_{i\in I_a\!}\;\Bigl(1+\frac h{u-z_i}\Bigr)\,,
\\[5pt]
\label{Exi}
E_p^\pm(u)\,\xi_I^\pm\,=
\sum_{i\in I_{p+1}\!}\,\frac{\xi_{I^{i\prime}}^\pm}{u-z_i}\,
\prod_{\satop{k\in I_{p+1}\!}{k\ne i}}\!\frac{z_i-z_k+h}{z_i-z_k}\;,
\\[2pt]
\label{Fxi}
F_p^\pm(u)\,\xi_I^\pm\,=
\sum_{j\in I_p\!}\;\frac{\xi_{I^{\jp}}^\pm}{u-z_j}\,
\prod_{\satop{k\in I_p\!}{k\ne j}}\,\frac{z_j-z_k-h}{z_j-z_k}\;,
\end{gather}
where the sequences \,$I^{\ip}$ and \,$I^{\jp}$ are defined as
follows: \,$I^{\ip}_a\!=I^{\jp}_a\!=I_a$ \,for \,$a\ne p,p+1$,
\,and \,$I^{\ip}_p\!=I_p\cup\iset$\>, \,$I^{\ip}_{p+1}\!=I_{p+1}-\iset$\>,
\,$I^{\jp}_p\!=I_p-\jset$\>, \,$I^{\jp}_{p+1}\!=I_{p+1}\cup\jset$\>.
\end{thm}
\begin{proof}
Formulae~\Ref{Axi}\,--\,\Ref{Fxi} can be obtained from the results of \cite{NT}
as a particular case. Here we outline a partially alternative proof. We
consider the case of the plus sign. The other case can be done similarly.

\vsk.2>
First observe that by formula \Ref{xi+si} and Lemma \ref{pm & Yan}, it suffices
to prove formulae~\Ref{Axi}\,--\,\Ref{Fxi} only for \,$I=\IMI$. In this case,
formula \Ref{Axi} for \,$n>1$ follows from the coproduct formula~\Ref{DlM}
and the \,$n=1$ \,case of \Ref{Axi}. The proof of \Ref{Axi} for \,$n=1$
\,is straightforward.

\vsk.2>
To get formula \Ref{Exi} for \,$I=\IMI$, observe that by formulae \Ref{xi+v},
\Ref{DlM}, \Ref{Axi}, we have
\,$E_p^+(u)\,\xi_\IMI^+=\sum_{\>i\in I_{p+1}\!}c_i\,\xi_{\IMIp}^+$\>.
The largest element of \,$I^{\<\>\min}_{p+1}$ \>equals
\,$i_{\max}=\la_1\lsym+\la_{p+1}$\>, the coefficient \,$c_{i_{\max}}$
can be calculated due to the triangularity property \Ref{xi+v},
and \,$c_{i_{\max}}$ has the required form. The coefficient \,$c_i$ for other
\,$i\in I_{p+1}$ can be obtained from \,$c_{i_{\max}}$ by permuting \,$z_i$ and
\,$z_{i_{\max}}$ because \,$\IMI$ is invariant under the transposition of
\,$i$ and \,$i_{\max}$\>. Thus all the coefficients \,$c_i$ are as required,
which proves formula \Ref{Exi}.

\vsk.2>
The proof of formula \Ref{Fxi} is similar.
\end{proof}

\subsection{Actions of the dynamical Hamiltonians}
\label{dynact}
\vsk.2>
\begin{lem}
\label{lphoX8}
For actions \,$\pho^\pm$, see \Ref{pho}, of the dynamical Hamiltonians
\,$\Xin_1\lc\Xin_N\in\tbi\!$ on \,$V\!$-valued functions of \,$\zzz,h$, we have
\begin{align}
\label{phoX8}
\pho^+(\Xin_i)\, &{}=\,\sum_{a=1}^n z_a\>e^{(a)}_{\ii}+\>
\frac h2\>(e_{\ii}-e_{\ii}^2)+\>
h\>\sum_{j=1}^N\,\sum_{1\leq a<b\leq n}\!e^{(a)}_{\ij}\>e^{(b)}_{\ji}-\>
h\>\sum_{j=1}^{i-1}\,G_{\ij}\,,
\\[2pt]
\pho^-(\Xin_i)\, &{}=\,\sum_{a=1}^n z_a\>e^{(a)}_{\ii}+\>
\frac h2\>(e_{\ii}-e_{\ii}^2)+\>
h\>\sum_{j=1}^N\,\sum_{1\leq b<a\leq n}\!e^{(a)}_{\ij}\>e^{(b)}_{\ji}-\>
h\>\sum_{j=1}^{i-1}\,G_{\ij}\,,
\notag
\\[-16pt]
\notag
\end{align}
where \;$G_{\ij}=\>e_{\ij}\>e_{\ji}\<-e_{\ii}=\>e_{\ji}\>e_{\ij}\<-e_{\jj}$
\,and \;$e_{\kl}=\>e_{\kl}^{(1)}\lsym+e_{\kl}^{(n)}$ \,for every \;$k,l$\>.
\end{lem}
\begin{proof}
The statement follow from formulae \Ref{Lpm} and \Ref{X}.
\end{proof}

The operators \,$\pho^\pm(\Xk_i)$, \,$\pho^\pm(\Xkp_i)$,
\,$\pho^\pm(\Xkm_{\bla\<\>,\<\>i})$, \,$i=1\lc N$,
can be found from formulae \Ref{XX}\,--\,\Ref{Xkm}.

\begin{lem}
\label{Xv}
For any \,$\bla$ \,and any \,$i=1\lc N$, we have
\;$\pho^+(\Xkp_i)\>v^+_\bla\<\>=\,\pho^+(\Xin_i)\>v^+_\bla$\>,
\vvn.3>
\be
\pho^+(\Xkm_{\bla\<\>,\<\>i})\>v^=_\bla\>=\,\pho^+(\Xin_i)\>v^=_\bla\>,\qquad
\pho^-(\Xkm_{\bla\<\>,\<\>i})\>v^-_\bla\<\>=\,\pho^-(\Xin_i)\>v^-_\bla\>.
\vv.3>
\ee
\end{lem}
\begin{proof}
By Lemma \ref{v+gl}, we have
\;$(e_{\ii}\>e_{\jj}\<-e_{\ij}\>e_{\ji}\<+e_{\ii})\>v^+_\bla\<\<=\<\>0$
\vvn.1>
\,for any \,$i\ne j$\>, \,which yields the equality for \,$v^+_\bla$,
\vvn.08>
\>see \Ref{Xkp}. By Lemma \ref{v-gl}, we have \;$e_{\ij}\>v^-_\bla\<=\<\>0$
\,and \;$e_{\ij}\>v^=_\bla\<=\<\>0$ \,for any \;$i\<\>,j$ \>such that
\,$\la_i\ge\la_j$\>, \,which yields the equalities for \,$v^=_\bla$ \,and
\,$v^-_\bla$\>, \>see \Ref{Xkm}.
\end{proof}

\subsection{\qKZ/ difference connection}
\label{sec qkz}
\,Let
\vvn.1>
\be
R^{(\ij)}(u)\,=\,\frac{u+h\<\>P^{(\ij)}}{u+h}\;,\qquad
i,j=1\lc n\,,\quad i\ne j\,.\kern-3em
\vv.3>
\ee
Define operators \,$\Ko^\pm_1\lc\Ko^\pm_n\in\End(V)\ox\Czh$\>,
\vvn.4>
\begin{align*}
\Ko^+_i\>&{}=\,
R^{(\ii-1)}(z_i\<-z_{i-1})\,\dots\,R^{(i,1)}(z_i\<-z_1)\,\times{}
\\[2pt]
& {}\>\times\<\;\kk_1^{e_{1,1}^{(i)}}\!\dots\,\kk_N^{e_{N,N}^{(i)}}\,
R^{(i,n)}(z_i\<-z_n\<-\kp\<\>)\,\dots\,R^{(\ii+1)}(z_i\<-z_{i+1}\<-\kp\<\>)\,,
\\[9pt]
\Ko^-_i\>&{}=\,
R^{(\ii+1)}(z_i\<-z_{i+1})\,\dots\,R^{(i,n)}(z_i\<-z_n)\,\times{}
\\[2pt]
& {}\>\times\<\;\kk_1^{e_{1,1}^{(i)}}\!\dots\,\kk_N^{e_{N,N}^{(i)}}\,
R^{(i,1)}(z_i\<-z_1\<-\kp\<\>)\,\dots\,R^{(\ii-1)}(z_i\<-z_{i-1}\<-\kp\<\>)\,.
\\[-14pt]
\end{align*}
Consider the difference operators \,$\Kh^\pm_1\lc\Kh^\pm_n$
acting on \,$V\<$-valued
\vvn.4>
functions of \>$\zzz,h$\>, %% $\kkk$
\be
\Kh^\pm_i\>f(\zzz,h)\,=\,K^\pm_i(\zzz,h)\,f(\zzip)\,.
\vv.3>
\ee

\begin{thm}[\cite{FR}]
\label{thmfr}
The operators \;$\Kh^+_1\lc\Kh^+_n$ pairwise commute.
Similarly, the operators \;$\Kh^-_1\lc\Kh^-_n$ pairwise commute.
\qed
\end{thm}

\begin{thm}[\cite{TV1}]
\label{thm qkz}
The operators \;$\Kh^+_1\lc\Kh^+_n$,
\,$\pho^+(\nabla^+_1)\lc\pho^+(\nabla^+_N)$ \,pairwise commute.
Similarly, the operators \;$\Kh^-_1\lc\Kh^-_n$,
\,$\pho^-(\nabla^-_{\<\bla\<\>,1})\lc\pho^-(\nabla^-_{\<\bla\<\>,\<\>N})$
\,pairwise commute.
\qed
\end{thm}

\begin{cor}
\label{cor qkz}
The operators \;$\Ko^+_1|_{\kp=0}\>\lc\Ko^+_n|_{\kp=0}$\>,
\vvn.06>
\,$\pho^+(\Xkp_1)\lc\pho^+(\Xkp_N)$ \,pairwise commute.
Similarly, the operators \;$\Ko^-_1|_{\kp=0}\>\lc\Ko^-_n|_{\kp=0}$\>,
\,$\pho^-(\Xkm_{\<\bla\<\>,1})\lc\pho^-(\Xkm_{\<\bla\<\>,\<\>N})$
\,pairwise commute.
\end{cor}

The commuting difference operators $\Kh^+_1\lc\Kh^+_n$ define the {\it \qKZ/
difference connection}. Similarly, the commuting difference operators
$\Kh^-_1\lc\Kh^-_n$ define another \qKZ/ difference connection.
Theorem \ref{thm qkz} says that the \qKZ/ difference connections
commute with the corresponding dynamical connections

\begin{prop}
\label{BK}
For every \;$X\!\in\tbk$, \,the operators
\;$\Ko^+_1|_{\kp=0}\>\lc\Ko^+_n|_{\kp=0}$ \,commute with \;$\pho^+(X)$,
\,and the operators \;$\Ko^-_1|_{\kp=0}\>\lc\Ko^-_n|_{\kp=0}$
\,commute with \;$\pho^-(X)$.
\end{prop}
\begin{proof}
By formulae \Ref{Lpm}, \Ref{pho}, \Ref{Bp}, we have
\vvn.1>
\be
\Ko^\pm_i|_{\kp=0}\,=\,\Bigl(\pho^\pm\bigl(\Bk_1(u)\bigr)\,
\prod_{j=1}^n\,\frac{u-z_j}{u-z_j\<+h}\>\Bigr)\Big|_{\>u=z_i}\>,
\qquad i=1\lc n\,,
\vv.1>
\ee
so the statement follows from Theorem \ref{BY}.
Alternatively, the proposition follows from
the Yang-Baxter equation \Ref{YB}.
\end{proof}

\section{Equivariant cohomology of the cotangent bundles of\\
partial flag varieties}
\label{sec EQ}

\subsection{Partial flag varieties}
\label{sec Partial flag varieties}
For \,$\bla\in\Z^N_{\geq 0}$, \,$|\bla|=n$, consider the partial flag variety
\;$\Fla$ parametrizing chains of subspaces
\be
0\,=\,F_0\subset F_1\lsym\subset F_N =\,\C^n
\ee
with \;$\dim F_i/F_{i-1}=\la_i$, \;$i=1\lc N$.
Denote by \,$\tfl$ the cotangent bundle of \;$\Fla$.

\vsk.2>
Let \,$T^n\!\subset GL_n=GL_n(\C)$ \,be the torus of diagonal matrices.
The groups \,$T^n\!\subset GL_n$ act on \;$\C^n\<$ and hence on \,$\tfl$.
Let the group \;$\C^*\<$ act on \,$\tfl$ by multiplication in each fiber.

\vsk.2>
The set of fixed points \,$(\tfl)^{T^n\times\C^*}\!$ of the torus action lies
in the zero section of the cotangent bundle and consists of the coordinate
flags \,$F_I=(F_0\lsym\subset F_N)$, \,$I=(I_1\lc I_N)\in\Il$, where \,$F_i$
\,is the span of the standard basis vectors \;$u_j\in\C^n$ with
\,$j\in I_1\lsym\cup I_i$. The fixed points are in a one-to-one correspondence
with the elements of \;$\Il$ and hence with the basis vectors \,$v_I$ of
\,$V_\bla$, see Section \ref{sec gln}.

\vsk.3>
We consider the $\GLC$-equivariant cohomology algebra
\vvn.3>
\be
H_\bla\,=\,H^*_{GL_n\times\C^*}(\tfl;\C)\,.
\vv.3>
\ee
Denote by $\Ga_i=\{\ga_{i,1}\lc\ga_{i,\la_i}\}$ the set of
the Chern roots of the bundle over $\Fla$ with fiber $F_i/F_{i-1}$.
Let \;$\Gmm=(\Ga_1\<\>\lsym;\Ga_N)$. Denote by $\zb=\{\zzz\}$ the Chern roots
corresponding to the factors of the torus $T^n$. Denote by $h$ the Chern root
corresponding to the factor $\C^*$ acting on the fibers of $\tfl$
by multiplication. Then
\vvn-.2>
\beq
\label{Hrel}
H_\bla\,=\,\Czghl\>\Big/\Bigl\bra\,
\prod_{i=1}^N\prod_{j=1}^{\la_i}\,(u-\ga_{\ij})\,=\,\prod_{a=1}^n\,(u-z_a)
\Bigr\ket\,.
\eeq
The cohomology $H_\bla$ is a module over $H^*_{GL_n\times\C^*}({pt};\C)=\CZH$.

\vsk.2>
For $A\subset\{1\lc n\}$ denote $\zb_A=\{z_a,\,a\in A\}$, and
for $I=(I_1\lc I_N)\in\Il$ denote \,$\zb_I=(\zb_{I_1}\lsym;\zb_{I_N})$\>.
Set
\vvn.3>
\beq
\label{zeta}
\zla:H_\bla\to\,\Czhl\,,\qquad
[f(\zb;\Gmm;h)]\,\mapsto\,f(\zb;\zb_\IMI;h)\,.\kern-1em
\vv.3>
\eeq
Clearly, \,$\zla$ \,is an isomorphism of \;$\CZH$-algebras.

\vsk.2>
Let $\Jc\subset H_\bla$ be the ideal generated by the relations $h=1$ and
$\si_i(\zzz)=0$, \,$i=1\lc n$. Then \,$H_\bla/\Jc=H^*(\tfl;\C)$\,.

\subsection{Equivariant integration}
We will need the {\it integration on \;$\Fla$ map}
\vvn.2>
\be
\int\!:\>H_\bla\to\,H^*_{GL_n\times\C^*}({pt},\C)\,,
\vv-.3>
\ee
that is, the composition
\vvn-.1>
\be
H_\bla\>=\>H^*_{GL_n}(\tfl;\C)\ox\C[h]\,\xto{\;\int_{\Fla}\!\!\ox\id\;}\,
H^*_{GL_n}({pt};\C)\ox\C[h]\>=\>H^*_{GL_n\times\C^*}({pt};\C)\,.
\vv.3>
\ee

\vsk.2>
The Atiyah-Bott equivariant localization theorem \cite{AB} gives
the integration on \,$\Fla$ map in terms of the fixed point set
$(\tfl)^{T^n\times\C^*}$: for any $[f(\zb;\Gmm;h)]\in H_\bla$\>,
\vvn.3>
\beq
\label{int}
\int [f]\,=\,(-1)^{\sum_{i<j}\la_i\la_j}
\sum_{I\in\Il}\frac{f(\zb;\zb_I;h)}{R(\zb_I)}\;,
\eeq
where \,$R(\zb_I)$ \,is given by \Ref{R}. Clearly, the right-hand side
in \Ref{int} lies in $\CZH$. The integration on \,$\Fla$ map induces
the {\it Poincare pairing on\/} \,$\Fla$,
\beq
\label{poinc}
H_\bla\ox H_\bla\to\CZH\,,\qquad [f]\ox[\<\>g\<\>]\mapsto\int\>[fg]\;.
\eeq
After factorization by the ideal $\Jc$ we obtain the nondegenerate Poincare
pairing
\vvn.3>
\beq
\label{PP over C}
H^*(\tfl;\C)\ox H^*(\tfl;\C)\,\to\,\C\,.
\eeq

\vsk.3>
We will also use the {\it integration on \,$\tfl$ map},
\vvn.3>
\begin{gather}
\intt\!:\>H^*(\tfl;\C)\,\to\,Z^{-1}\>\CZH\,,
\notag
\\[4pt]
\label{intt}
\intt\>[f]\,=\,(-1)^{\sum_{i<j}\la_i\la_j}
\sum_{I\in\Il}\frac{f(\zb;\zb_I;h)}{Q(\zb_I)R(\zb_I)}\;.
\\[-14pt]
\notag
\end{gather}
Here \;$Z=\prod_{i\ne j}(z_i-z_j+h)$\>, cf.~\Ref{Z}.
\vvn.08>
Notice that \;$Q(\zb_I)\>R(\zb_I)$ is the Euler class of the tangent space
at the point $F_I\in\tfl$. This integration map was used in \cite{BMO},
see also \cite{HP}. The integration on \,$\tfl$ map induces
the {\it Poincare paring on\/} \,$\tfl$,
\vvn.2>
\beq
\label{intnew}
H_\bla\ox H_\bla\to Z^{-1}\>\CZH\,,\qquad
[f]\ox[\<\>g\<\>]\mapsto\intt\>[fg]\;,
\vv.1>
\eeq
which will be called the {\it Poincare paring on\/} \,$\tfl$.

\begin{example}
Let \,$N=n=2$ \,and \,$\bla=(1,1)$. Then
\vvn.2>
\begin{align*}
\int\>[\>1\>]\,=\,0\,, &\kern3em \int\>[\>\ga_{1,1}\>]\,=\,-\>1\,,
%\qquad\int\>[\>\ga_{2,1}\>]\,=\,1\,,
\\[4pt]
\intt\>[\>1\>]\,=\,\frac2{(z_1-z_2+h)\>(z_2-z_1+h)}\;, &\kern3em
\intt\>[\>\ga_{1,1}\<\>]\,=\,\frac{z_1+z_2-h}{(z_1-z_2+h)\>(z_2-z_1+h)}\;.
%\intt\>[\>\ga_{2,1}\<\>]\,=\,\frac{z_1+z_2+h}{(z_1-z_2+h)\>(z_2-z_1+h)}
\end{align*}
\end{example}

\subsection{$H_\bla$ and \,$\Vl$, \,$\DLe$, \,$\DL$}
\label{secHla}
Consider the maps \;$\nu^+_\bla\<:H_\bla\to\Vl$,
\;$\nu^=_\bla\<:H_\bla\to\DLe$, \,and \;$\nu^-_\bla\<:H_\bla\to\DL$,
\vvn-.4>
\begin{gather}
\label{nu+}
\nu^+_\bla:\,[f]\,\mapsto\, %% [f(\zb;\Gmm;h)]
\sum_{I\in\Il}\,\frac{f(\zb;\zb_I;h)\>Q(\zb_I)}{R(\zb_I)}\,\,\xi_I^+\,,
\\[3pt]
\label{nu=}
\nu^=_\bla:\,[f]\,\mapsto\,
\sum_{I\in\Il}\,\frac{f(\zb;\zb_I;h)}{R(\zb_I)}\,\,\xi_I^+\,.
\\[3pt]
\label{nu-}
\nu^-_\bla:\,[f]\,\mapsto\,
\sum_{I\in\Il}\,\frac{f(\zb;\zb_I;h)}{R(\zb_I)}\,\,\xi_I^-\,.
\\[-14pt]
\notag
\end{gather}
In particular, \;$\nu^+_\bla:[\>1\>]\mapsto v^+_\bla$\>,
\,\;$\nu^=_\bla:[\>1\>]\mapsto v^=_\bla$\>,
\,\;$\nu^-_\bla:[\>1\>]\mapsto v^-_\bla$\>, \,see \Ref{v+xi}.
\vvn.2>
We have \;$\nu^+_\bla\<=\<\>\thi^+_\bla\,\zla$\>,
\;$\nu^=_\bla\<=\<\>\thi^=_\bla\,\zla$\>,
\;$\nu^-_\bla\<=\<\>\thi^-_\bla\,\zla$\>,
\vvn.16>
where the maps \;$\thi^+_\bla$, \;$\thi^=_\bla$, \;$\thi^-_\bla$ \vvn.06> and
\,$\zla$ are given by formulae \Ref{thi+}\,--\,\Ref{thi-} and \Ref{zeta}.
Observe that
\vvn.2>
\beq
\label{nu=-}
\nu^-_\bla=\,\Pit\;\nu^=_\bla\,,
\vv-.1>
\eeq
where \;$\Pit$ \,is given by formula \Ref{Pit}.

\begin{lem}
\label{nu+-}
The maps \;$\nu^+_\bla$, \;$\nu^=_\bla$, \;$\nu^-_\bla$ are isomorphisms of
\;$\CZH$-modules.
\end{lem}
\begin{proof}
The claim follows from Lemma \ref{thi+-}.
\end{proof}

\begin{prop}
\label{SP}
The Shapovalov form and Poincare pairing on \,$\Fla$ are related
by the formula
\vvn-.2>
\be
\Sc\bigl(\nu^+_\bla[f],\nu^-_\bla[\<\>g\<\>]\bigr)\,=\,
(-1)^{\sum_{i<j}\la_i\la_j}\int [f][\<\>g\<\>]\;.
\vv.4>
\ee
\end{prop}
\begin{proof}
The statement follows from formulae \Ref{S+-}, \Ref{int}, \Ref{nu+}, \Ref{nu-},
and Lemma \ref{nu+-}.
\end{proof}

\begin{prop}
\label{SP=}
The Shapovalov form and Poincare pairing on $\tfl$ are related
by the formula
\vvn-.2>
\be
\Sc\bigl(\nu^=_\bla[f],\nu^-_\bla[\<\>g\<\>]\bigr)\,=\,
(-1)^{\sum_{i<j}\la_i\la_j}\intt\>[f][\<\>g\<\>]\;.
\vv.4>
\ee
\end{prop}
\begin{proof}
The statement follows from formulae \Ref{S+-}, \Ref{intt}, \Ref{nu=},
\Ref{nu-}, and Lemma \ref{nu+-}.
\end{proof}

Introduce the elements \,$f_{p,s}\in\Czghl$, \,$p=1\lc N$, \,$s\in\Z_{>0}$\>,
by
\vvn.2>
\beq
\label{fps}
\prod_{i=1}^{\la_p}\,\Bigl(1+\frac h{u-\ga_{\pci}}\Bigr)\,=\,
1+h\>\sum_{s=1}^\infty\,f_{p,s}\>u^{-s}\,.
\vv-.1>
\eeq
Since \,$f_{p,s}=\sum_{i=1}^{\la_p}\ga_{\pci}^{s-1}+h\>g_{p,s}$\>, \,where
\,$g_{p,s}\in\Czghl$ are of homogeneous degree \,$s-2$, the elements
\,$[f_{p,s}]$ \,$p=1\lc N$, \,$s>0$, \,generate \,$H_\bla$ over \,$\C[h]$.

\vsk.2>
Define the elements \,$C_{p,s}^\pm\in\End(V)\ox\Czh$\>, \,$p=1\lc N$,
\,$s\in\Z_{>0}$, by
\beq
\label{Cps}
\bigl(A_{p-1}^\pm(u)\bigr)^{-1}A_p^\pm(u)\,=\,
1+h\>\sum_{s=1}^\infty\,C_{p,s}^\pm\>u^{-s}\,,
\eeq
where \,$A_0^\pm(u)=1$ \,and \,$A_1^\pm(u)\lc A_N^\pm(u)$ are given
by \Ref{AEF}.

\vsk.3>
Let \,$\Ac$ \,be a commutative algebra. The algebra \,$\Ac$ \,considered as
a module over itself is called the regular representation of \,$\Ac$\>.

\begin{thm}
\label{thm main+}
The assignment \;$\mu^+_\bla:[f_{p,s}]\>\mapsto C_{p,s}^+$\>,
\,$p=1\lc N$, \,$s>0$, \,extends uniquely to an isomorphism
\,$H_\bla\to\tbi(\Vl)$ \,of \;$\CZH$-algebras.
The maps \;$\mu^+_\bla\<$ and \;$\nu^+_\bla\<$ identify
\vvn.1>
the \,$\tbi(\Vl)$-module \,$\Vl$
with the regular representation of the algebra \,$H_\bla$.
\end{thm}

\begin{thm}
\label{thm main=}
The assignment \;$\mu^=_\bla:[f_{p,s}]\>\mapsto C_{p,s}^+$\>,
\,$p=1\lc N$, \,$s>0$, \,extends uniquely to an isomorphism
\,$H_\bla\to\tbi(\DLe)$ \,of \;$\CZH$-algebras.
The maps \;$\mu^=_\bla\<$ and \;$\nu^=_\bla\<$ identify
\vvn.1>
the \,$\tbi(\DLe)$-module \,$\DLe$
with the regular representation of the algebra \,$H_\bla$.
\end{thm}

\begin{thm}
\label{thm main-}
The assignment \;$\mu^-_\bla:[f_{p,s}]\>\mapsto C_{p,s}^-$\>,
\,$p=1\lc N$, \,$s>0$, \,extends uniquely to an isomorphism
\,$H_\bla\to\tbi(\DL)$ \,of \;$\CZH$-algebras.
The maps \;$\mu^-_\bla\<$ and \;$\nu^-_\bla\<$ identify
\vvn.1>
the \,$\tbi(\DL)$-module \,$\DL$
with the regular representation of the algebra \,$H_\bla$.
\end{thm}

\begin{proof}[Proofs of Theorems \ref{thm main+}\,--\,\ref{thm main-}.]
The operator \,$\nu^+_\bla\,C^+_{p,s}\>(\nu^+_\bla)^{-1}$ acts on \,$H_\bla$
\vvn.1>
as the multiplication by \,$[f_{p,s}]$\>, see Lemma \ref{nu+-} and
formula \Ref{Axi}. This yields Theorem \ref{thm main+}.
The proofs of Theorems~\ref{thm main=} and \ref{thm main-} are similar.
\end{proof}

Theorems \ref{thm main=} and \ref{thm main-} are equivalent by Lemma \ref{=-}
and relation \Ref{nu=-}. In particular,
\vvn.2>
\beq
\label{mu=-}
\Pit\;\mu^=_\bla(f)\,=\,\mu^-_\bla(f)\;\Pit
\vv-.1>
\eeq
for any $f\<\in H_\bla$\>, see \Ref{pho+-}.

\begin{cor}
\label{VB}
The \,$\tbi\<$-modules \;$\V^+\!$, \,$\DVe\!$, and \,$\DV\!$ are isomorphic.
\end{cor}
\begin{proof}
The isomorphisms restricted to weight subspaces are
\;$\nu^=_\bla\,(\nu^+_\bla)^{-1}\!:\Vl\to\DLe$ \,and
\;$\nu^-_\bla\,(\nu^+_\bla)^{-1}\!:\Vl\to\DL$\>.
\end{proof}

\begin{cor}
\label{Bimax}
The subalgebras \;$\tbi(\Vl)\subset\End(\Vl)$, \;$\tbi(\DLe)\subset\End(\DLe)$
\,and\\ \;$\tbi(\DL)\subset\End(\DL)$ \,are maximal commutative subalgebras.
\qed
\end{cor}

\begin{lem}
\label{mufg}
For any \;$f\<\in H_\bla$, and any \,$g_1\in\Vl$, \,$g_2\in\DLe$,
\,$g_3\in\DLe$, we have
\vvn.2>
\be
\Sc\bigl(\mu^+_\bla(f)\>g_1,g_3\bigr)\,=\,
\Sc\bigl(g_1,\mu^-_\bla(f)\>g_3\bigr)\,,\qquad
\Sc\bigl(\mu^=_\bla(f)\>g_2,g_3\bigr)\,=\,
\Sc\bigl(g_2,\mu^-_\bla(f)\>g_3\bigr)\,.
\vv.2>
\ee
\end{lem}
\begin{proof}
The statement follows from the definition of the maps \,$\mu^+_\bla$,
\,$\mu^=_\bla$, \,$\mu^-_\bla$ and Corollary \ref{SB}.
\end{proof}

\begin{example}
Let \,$N=n=2$ \,and \,$\bla=(1,1)$. \,Then the weight subspace \,$V_\bla$
is two-dimen\-sional, \,$V_\bla=\>\C\>v_{(1,2)}\oplus\C\>v_{(2,1)}$\>,
\;$v_{(1,2)}\<=v_1\<\ox v_2$\>, \,\,$v_{(2,1)}\<=v_2\ox v_1$\>.
\vvn.3>
The vectors \;$\xi^\pm_I$ are
\begin{alignat*}2
\xi^+_{(1,2)} &{}=\>v_{(1,2)}\,, &\kern 3em
\xi^+_{(2,1)} &{}=\,
\frac{z_2-z_1}{z_2-z_1+h}\;v_{(2,1)}+\frac h{z_2-z_1+h}\;v_{(1,2)}\,,
\\[4pt]
\xi^-_{(2,1)} &{}=\>v_{(2,1)}\,, &
\xi^-_{(1,2)} &{}=\,
\frac{z_1-z_2}{z_1-z_2+h}\;v_{(1,2)}+\frac h{z_1-z_2+h}\;v_{(2,1)}\,.
\\[-14pt]
\end{alignat*}
Elements of \,$\Vl$ have the form
\be
f(z_1,z_2,h)\,\frac{z_1-z_2+h}{z_1-z_2}\;\xi^+_{(1,2)}+\>
f(z_2,z_1,h)\,\frac{z_2-z_1+h}{z_2-z_1}\;\xi^+_{(2,1)}\,,
\vv.2>
\ee
elements of \,$\DLe$ have the form
\be
\frac{f(z_1,z_2,h)}{z_1-z_2}\;\xi^+_{(1,2)}+\>
\frac{f(z_2,z_1,h)}{z_2-z_1}\;\xi^+_{(2,1)}\,,
\ee
and elements of \,$\DL$ have the form
\be
\frac{f(z_1,z_2,h)}{z_1-z_2}\;\xi^-_{(1,2)}+\>
\frac{f(z_2,z_1,h)}{z_2-z_1}\;\xi^-_{(2,1)}\,.
\ee
The series \Ref{AEF}, \Ref{Cps} for generators of the Bethe algebras
\,$\tbi(\Vl)$ \,and \,$\tbi(\DL)$ \,are
\vvn.2>
\be
A_2^\pm(u)\,=\,\Bigl(1+\frac h{u-z_1}\Bigr)\,\Bigl(1+\frac h{u-z_2}\Bigr)\,,
\ee
\be
\leftline{$\dsize A_1^+(u)\,=\, \begin{pmatrix}\dsize 1+\frac h{u-z_1} & h^2\\
0 & \dsize 1+\frac h{u-z_2}\end{pmatrix},$\hfill $\dsize
\bigl(A_1^+(u))^{-1}A_2^+(u)\,=\,\begin{pmatrix}\dsize1+\frac h{u-z_2}&-\>h^2\\
0 & \dsize 1+\frac h{u-z_1}\end{pmatrix},$}
\ee
\be
\leftline{$\dsize A_1^-(u)\,=\, \begin{pmatrix}\dsize 1+\frac h{u-z_1} & 0\\
h^2 & \dsize 1+\frac h{u-z_2}\end{pmatrix},$\hfill $\dsize
\bigl(A_1^-(u))^{-1}A_2^-(u)\,=\,\begin{pmatrix}\dsize1+\frac h{u-z_2}& 0 \\
-\>h^2 & \dsize 1+\frac h{u-z_1}\end{pmatrix},$}
\vv.4>
\ee
where we are using the basis \;$v_{(1,2)}\>,\>v_{(2,1)}$ \,of \,$V_\bla$\>.
For the maps \,$\nu^+_\bla$, \,$\nu^=_\bla$, \,$\nu^-_\bla$, and
\,$\nu^+_\bla$, \,$\nu^=_\bla$, \,$\nu^-_\bla$, we have
\vvn-.2>
\be
\mu^+_\bla:[\>1\>]\,\mapsto\,\begin{pmatrix}\,1 &0\,\\\,0&1\,\end{pmatrix},
\qquad \mu^+_\bla:[\>\ga_{1,1}\<\>]\,\mapsto\,
\begin{pmatrix} z_1 & h \\ 0 & z_2 \end{pmatrix},
%\qquad \mu^+_\bla:[\>\ga_{2,1}\<\>]\,\mapsto\,
%\begin{pmatrix} z_2 & -\>h \\ 0 & z_1\!\! \end{pmatrix},
\vv-.1>
\ee
\begin{gather*}
\nu^+_\bla:[\>1\>]\,\mapsto\,
\frac{z_1-z_2+h}{z_1-z_2}\;\xi^+_{(1,2)}+\>
\frac{z_2-z_1+h}{z_2-z_1}\;\xi^+_{(2,1)}\,=\,v_{(1,2)}+\>v_{(2,1)}=\>v^+_\bla,
\\[6pt]
\nu^+_\bla:[\>\ga_{1,1}\<\>]\,\mapsto\,
z_1\,\frac{z_1-z_2+h}{z_1-z_2}\;\xi^+_{(1,2)}+\>
z_2\,\frac{z_2-z_1+h}{z_2-z_1}\;\xi^+_{(2,1)}\,=\,
(z_1\<+h)\,v_{(1,2)}+\>z_2\>v_{(2,1)}\,,
%\\[6pt]
%\nu^+_\bla:[\>\ga_{2,1}\<\>]\,\mapsto\,
%z_2\,\frac{z_1-z_2+h}{z_1-z_2}\;\xi^+_{(1,2)}+\>
%z_1\,\frac{z_2-z_1+h}{z_2-z_1}\;\xi^+_{(2,1)}\,=\,
%(z_2\<-h)\,v_{(1,2)}+\>z_1\>v_{(2,1)}\,.
\end{gather*}
\vvn.2>
\be
\mu^=_\bla:[\>1\>]\,\mapsto\,\begin{pmatrix}\,1 &0\,\\\,0&1\,\end{pmatrix},
\qquad \mu^=_\bla:[\>\ga_{1,1}\<\>]\,\mapsto\,
\begin{pmatrix} z_1 & h \\ 0 & z_2 \end{pmatrix},
%\qquad \mu^=_\bla:[\>\ga_{2,1}\<\>]\,\mapsto\,
%\begin{pmatrix} z_2 & -\>h \\ 0 & z_1\!\! \end{pmatrix},
\vv-.2>
\ee
\begin{gather*}
\nu^=_\bla:[\>1\>]\,\mapsto\,
\frac1{z_1-z_2}\;\xi^+_{(1,2)}+\>\frac1{z_2-z_1}\;\xi^+_{(2,1)}\,=\,
\frac{v_{(1,2)}-\>v_{(2,1)}}{z_1-z_2-h}\,=\,v^=_\bla,
\\[6pt]
\nu^=_\bla:[\>\ga_{1,1}\<\>]\,\mapsto\,
\frac{z_1}{z_1-z_2}\;\xi^+_{(1,2)}+\>\frac{z_2}{z_2-z_1}\;\xi^+_{(2,1)}\,=\,
\frac{(z_1-h)\>v_{(1,2)}-\>z_2\>v_{(2,1)}}{z_1-z_2-h}\,,
%\\[6pt]
%\nu^=_\bla:[\>\ga_{2,1}\<\>]\,\mapsto\,
%\frac{z_2}{z_1-z_2}\;\xi^+_{(1,2)}+\>\frac{z_1}{z_2-z_1}\;\xi^+_{(2,1)}\,=\,
%\frac{(z_2+h)\>v_{(1,2)}-\>z_1\>v_{(2,1)}}{z_1-z_2-h}\,.
\end{gather*}
\vvn.2>
\be
\mu^-_\bla:[\>1\>]\,\mapsto\,\begin{pmatrix}\,1 &0\,\\\,0&1\,\end{pmatrix},
\qquad \mu^-_\bla:[\>\ga_{1,1}\<\>]\,\mapsto\,
\begin{pmatrix} z_1 & 0 \\ h & z_2 \end{pmatrix}\,.
%\qquad \mu^-_\bla:[\>\ga_{2,1}\<\>]\,\mapsto\,
%\begin{pmatrix} z_2 & 0 \\ -\>h & z_1 \end{pmatrix}\>.
\vvn-.2>
\ee
\begin{gather*}
\nu^-_\bla:[\>1\>]\,\mapsto\,
\frac1{z_1-z_2}\;\xi^-_{(1,2)}+\>\frac1{z_2-z_1}\;\xi^-_{(2,1)}\,=\,
\frac{v_{(1,2)}-\>v_{(2,1)}}{z_1-z_2+h}\,=\,v^-_\bla,
\\[6pt]
\nu^-_\bla:[\>\ga_{1,1}\<\>]\,\mapsto\,
\frac{z_1}{z_1-z_2}\;\xi^-_{(1,2)}+\>\frac{z_2}{z_2-z_1}\;\xi^-_{(2,1)}\,=\,
\frac{z_1\>v_{(1,2)}-\>(z_2-h)\>v_{(2,1)}}{z_1-z_2+h}\,,
%\\[6pt]
%\nu^-_\bla:[\>\ga_{2,1}\<\>]\,\mapsto\,
%\frac{z_2}{z_1-z_2}\;\xi^-_{(1,2)}+\>\frac{z_1}{z_2-z_1}\;\xi^-_{(2,1)}\,=\,
%\frac{z_2\>v_{(1,2)}-\>(z_1+h)\>v_{(2,1)}}{z_1-z_2+h}\,.
\\[-12pt]
\end{gather*}
\end{example}

\subsection{Cohomology as $\ty$-modules}
\label{sec cohom and Yang}

The isomorphisms
\vvn.3>
\be
\nu^+=\,\tbigoplus_\bla\,\nu^+_\bla\,:\,\tbigoplus_\bla\,H_\bla\,\to\;\V^+,
\qquad
\nu^-=\,\tbigoplus_\bla\,\nu^-_\bla\,:\,\tbigoplus_\bla\,H_\bla\,\to\;\DV
\ee
induce two \>$\ty$-module structures on \;$\Hplus$ \>denoted respectively
by $\rho^+$ and $\rho^-$,
\vvn.3>
\be
\rho^\pm(X)\,=\,(\nu^\pm)^{-1}\>\pho^\pm(X)\;\nu^\pm
\vv.2>
\ee
for any \,$X\!\in\ty$\>.
\vvn.4>
The \>$\ty$-module structure on \;$\Hplus$ induced by the isomorphism
\be
\nu^==\,\tbigoplus_\bla\,\nu^=_\bla\,:\,\tbigoplus_\bla\,H_\bla\,\to\;\DVe
\vv-.2>
\ee
coincides with the $\rho^-$-structure
\vvn.2>
\beq
\label{rho=-}
(\nu^=)^{-1}\>\pho^+(X)\;\nu^=\>=\,\rho^-(X)\,,
\vv.3>
\eeq
since \;$\nu^-_\bla\<=\>\Pit\;\nu^=_\bla$
\vvn.3>
\,and \;$\Pit\,\pho^+(X)=\pho^-(X)\,\Pit$\>, see \Ref{nu=-}, \Ref{pho+-}.

\vsk.3>
By formulae \Ref{nu+}, \Ref{nu-} and \Ref{Axi}, we have
\;$\rho^\pm\bigl(A_p(u)\bigr):H_\bla\to\>H_\bla$,
\vvn.2>
\beq
\label{Arho}
\rho^\pm\bigl(A_p(u)\bigr)\>:\>[f]\;\mapsto\,
\Bigl[\>f(\zb;\Gmm;h)\;\prod_{a=1}^p\,\prod_{i=1}^{\la_p}
\;\Bigl(1+\frac h{u-\ga_{\pci}}\Bigr)\Bigr]\,,
\vv.3>
\eeq
for \,$p=1\lc N$. In particular, by \Ref{A}, \Ref{X8},
\vvn.4>
\beq
\label{rhoX8}
\rho^\pm(\Xin_i)\>:\>[f]\;\mapsto\,
[\>(\ga_{i,1}\lsym+\<\ga_{i,\la_i})\>f(\zb;\Gmm;h)]\,,\qquad i=1\lc N\,.
\kern-4em
\eeq

\vsk.3>
Let \;$\aal_1\lc\aal_{N-1}$ \,be simple roots,
\,$\aal_p=(0\lc0,1,\<-\>1,0\lc0)$, with \,$p-1$ first zeros.

\begin{thm}
\label{rho+-}
We have
\vvn-.8>
\begin{gather*}
\rho^+\bigl(E_p(u)\bigr)\<\>:\>H_{\bla\<\>-\aal_p}\,\mapsto\,H_\bla\,,
\\[6pt]
\rho^+\bigl(E_p(u)\bigr)\<\>:\>[f]\;\mapsto\,\biggl[\;
\sum_{i=1}^{\la_p}\;\frac{f(\zb;\Ga^{\ipi};h)}{u-\ga_{\pci}}\,\;
\prod_{\satop{j=1}{j\ne i}}^{\la_p}\,
\frac{\ga_{\pci}-\ga_{\pcj}-h}{\ga_{\pci}-\ga_{\pcj}}\;\biggr]\,,
\end{gather*}
\vv-.3>
\begin{gather*}
\rho^+\bigl(F_p(u)\bigr)\<\>:\>H_{\bla\<\>+\aal_p}\,\mapsto\,H_\bla\,,
\\[6pt]
\rho^+\bigl(F_p(u)\bigr)\<\>:[f]\;\mapsto\,\biggl[\;
\sum_{i=1}^{\la_{p+1}\!}\;\frac{f(\zb;\Ga^{\ip};h)}{u-\ga_{\poi}}\,\;
\prod_{\satop{j=1}{j\ne i}}^{\la_{p+1}\!}\,
\frac{\ga_{\poi}-\ga_{\poj}+h}{\ga_{\poi}-\ga_{\poj}}\;\biggr]\,,
\\[-20pt]
\end{gather*}
where
\vvn-.5>
\begin{gather*}
\Ga^{\ipi}=\,(\Ga_1\<\>\lsym;\Ga_{p-1}\<\>;\Ga_p-\{\ga_{\pci}\};
\Ga_{p+1}\cup\{\ga_{\pci}\};\Ga_{p+2}\<\>\lsym;\Ga_N)\,,
\\[8pt]
\Ga^{\ip}=\,(\Ga_1\<\>\lsym;\Ga_{p-1}\<\>;\Ga_p\cup\{\ga_{\poi}\};
\Ga_{p+1}-\{\ga_{\poi}\};\Ga_{p+2}\<\>\lsym;\Ga_N)\,.
\\[-14pt]
\end{gather*}
Similarly,
\vvn-.6>
\begin{gather*}
\rho^-\bigl(E_p(u)\bigr)\<\>:\>H_{\bla\<\>-\aal_p}\,\mapsto\,H_\bla\,,
\\[6pt]
\rho^-\bigl(E_p(u)\bigr)\<\>:\>[f]\;\mapsto\,\biggl[\;
\sum_{i=1}^{\la_p}\;\frac{f(\zb;\Ga^{\ipi};h)}{u-\ga_{\pci}}\,\;
\prod_{\satop{j=1}{j\ne i}}^{\la_p}\,\frac1{\ga_{\pcj}-\ga_{\pci}}\;
\prod_{k=1}^{\la_{p+1}\!}\,(\ga_{\pci}-\ga_{p+1,k}+h)\,\biggr]\,,
\end{gather*}
\vv-.3>
\begin{gather*}
\rho^-\bigl(F_p(u)\bigr)\<\>:\>H_{\bla\<\>+\aal_p}\,\mapsto\,H_\bla\,,
\\[6pt]
\rho^-\bigl(F_p(u)\bigr)\<\>:\>[f]\;\mapsto\,\biggl[\;
\sum_{i=1}^{\la_{p+1}\!}\;\frac{f(\zb;\Ga^{\ip};h)}{u-\ga_{\poi}}\,\;
\prod_{\satop{j=1}{j\ne i}}^{\la_{p+1}\!}\,\frac1{\ga_{\poi}-\ga_{\poj}}\;
\prod_{k=1}^{\la_p}\,(\ga_{p,k}-\ga_{\poi}+h)\,\biggr]\,.
\\[-14pt]
\end{gather*}
\end{thm}
\begin{proof}
The statement follows from formulae \Ref{nu+}\,--\,\Ref{nu-} and \Ref{Exi},
\Ref{Fxi}.
\end{proof}

The \,$\ty$-module structures \,$\rho^\pm$ on \;$\Hplus$ are
the Yangian versions of representations of the quantum affine algebra
$U_q(\Hat{\gln})$ considered in \cite{Vas1, Vas2}.
The $U_q(\Hat{\gln})$ analog of the $\rho^+$ structure on \;$\Hplus$
was considered in \cite{GV}. The $\rho^+$ structure on \;$\Hplus$
was considered in \cite{Va}.

\vsk.3>
The $\rho^-$-structure appears to be more preferable, it was used in
\cite{RTVZ} to construct quantized conformal blocks in the tensor power
$V^{\ox n}$, see also Sections \ref{elsewhere}\,--\,\ref{sec Special}.

\begin{cor}
\label{rho+-dual}
The \,$\ty$-module structures \,$\rho^\pm\!$ on \;$\Hplus$ are contravariantly
related through the Poincare pairing on \,$\Fla$,
\vvn.2>
\be
\int\>[f]\;\rho^+(X)\>[\<\>g\<\>]\;=
\int\>[\<\>g\<\>]\;\rho^-\bigl(\varpi(X)\bigr)\>[f]
\vv.1>
\ee
for any \,$X\in\ty$ and any \,$f,g\in\Hplus$.
\end{cor}
\begin{proof}
The statement follows from Proposition \ref{SP} and Lemma \ref{SY}.
\end{proof}

\begin{cor}
\label{rho-dual}
The \,$\ty$-module structure \,$\rho^-\!$ on \;$\Hplus$ is contravariantly
related to itself through the Poincare pairing on \,$\tfl$,
\vvn.2>
\be
\intt\>[f]\;\rho^-(X)\>[\<\>g\<\>]\;=
\intt\>[\<\>g\<\>]\;\rho^-\bigl(\varpi(X)\bigr)\>[f]
\vv.1>
\ee
for any \,$X\in\ty$ and any \,$f,g\in\Hplus$.
\end{cor}
\begin{proof}
The statement follows from relation \Ref{rho=-}, Proposition \ref{SP=} and
Lemma \ref{SY}.
\end{proof}

\begin{cor}
\label{Brho}
For any \,$X\!\in\tbk$ or \>$X\!\in\tbi$, \,and any \,$f,g\in\Hplus$,
we have
\vvn.3>
\be
\int\>[f]\;\rho^+(X)\>[\<\>g\<\>]\;=
\int\>[g]\;\rho^-(X)\>[\<\>f\<\>]\,,\qquad
\intt\>[f]\;\rho^-(X)\>[\<\>g\<\>]\;=
\intt\>[g]\;\rho^-(X)\>[\<\>f\<\>]\,.
\vv.1>
\ee
\end{cor}
\begin{proof}
The statement follows from Corollaries \ref{rho+-dual}, \ref{rho-dual}
and \ref{SB}.
\end{proof}

\vsk.3>
The \>$\ty$-module structures on \;$\Hplus$ descend to
two $\Yn$-module structures on the cohomology with constant coefficients
\vvn.2>
\be
H(\C)\,:=\!\bigoplus_{\bla\in\Z^N_{\geq 0},\,|\bla|=n}\!H^*(\tfl;\C)\,,
\ee
denoted by the same letters \,$\rho^+\!$ and \,$\rho^-$. The obtained
\,$\Yn$-modules are isomorphic to the irreducible \,$\Yn$-modules \,$\Vy\pm$,
and hence isomorphic among themselves, see Proposition~\ref{V0}.

\begin{cor}
\label{rho+-dual0}
The $\Yn$-module $H(\C)$ with the $\rho^+\!$-structure is contravariantly dual
to the $\Yn$-module $H(\C)$ with the $\rho^-\!$-structure with respect to the
Poincare pairing on \,$\Fla$,
\vvn.3>
\be
\int\>[f]\;\rho^+(X)\>[\<\>g\<\>]\;=
\int\>[\<\>g\<\>]\;\rho^-\bigl(\varpi(X)\bigr)\>[f]
\vv.1>
\ee
for any \,$X\in\ty$ and any \,$f,g\in\Hplus$.
\end{cor}
\begin{proof}
The statement follows from Corollary \ref{rho+-dual}.
\end{proof}

\begin{cor}
\label{rho-dual0}
The $\Yn$-module $H(\C)$ with the $\rho^-\!$-structure is self-dual
with respect to the Poincare pairing on \,$\tfl$,
\vvn-.1>
\be
\intt\>[f]\;\rho^-(X)\>[\<\>g\<\>]\;=
\intt\>[\<\>g\<\>]\;\rho^-\bigl(\varpi(X)\bigr)\>[f]
\vv.1>
\ee
for any \,$X\in\ty$ and any \,$f,g\in\Hplus$.
\end{cor}
\begin{proof}
The statement follows from Corollary \ref{rho-dual}.
\end{proof}

\subsection{Topological interpretation of the Yangian actions on cohomology}
The \>$\ty$-actions \,$\rho^\pm\<$ on \,$\Hplus$
have topological interpretations. Consider the partitions
\,$\mub'\!=(\la_1\lc\la_p-1,\,1,\>\la_{p+1}\lc,\la_N)$ \,and
\,$\mub''\!=(\la_1\lc\la_p,\,1,\>\la_{p+1}-1\lc\la_N)$\>.
There are natural forgetful maps
\beq
\label{eqn:forgetfuls}
\F_{\bla\<\>-\aal_p}\,\xlto{\pi'_1}\>\F_{\mub'}\xto{\pi'_2}\,\F_{\bla}\,,
\qquad\qquad\qquad
\F_{\bla\<\>+\aal_p}\,\xlto{\pi''_1}\>\F_{\mub''}\xto{\pi''_2}\,\F_{\bla}\,.
\vv.4>
\eeq
The rank $\la_p-1$, $1$, $\la_{p+1}$ bundles over $\F_{\mub'}$ with fibers
$F_p/F_{p-1}$, $F_{p+1}/F_p$, $F_{p+2}/F_{p+1}$ will be respectively denoted by
$A',\>B',\>C'$. The rank $\la_p$, $1$, $\la_{p+1}-1$ bundles over $\F_{\mub''}$
with fibers $F_p/F_{p-1}$, $F_{p+1}/F_p$, $F_{p+2}/F_{p+1}$ will be
respectively denoted by $A'',\>B'',\>C''$.

\vsk.3>
For a $T^n\<$-\>equivariant bundle $\xi$, let $e(\xi)$ be its equivariant
Euler class. We can make the extra \,$\C^*$ (whose Chern root is $h$) act
on any bundle by fiberwise action with weight $k\>{\cdot}\>h$ ($k\in \Z$).
The equivariant Euler class with this extra action will be denoted by
$e_{k\cdot h}(\xi)$. Note that $B'$ and $B''$ are line bundles, hence
their Euler class is their first Chern class.

\vsk.3>
Recall that a proper map $f:X\to Y$ induces the pullback $f^*:H^*(Y)\to H^*(X)$
and push-forward (also known as Gysin) $f_*:H^*(X)\to H^*(Y)$ maps on
cohomology.

\begin{thm}
\label{Varag}
The operators \,$\rho^\pm\bigl(E_p(u)\bigr)$, \,$\rho^\pm\bigl(F_p(u)\bigr)$,
are equal to the following topological operations
\vvn-.6>
\begin{gather*}
\rho^+\bigl(E_p(u)\bigr)\>:\>x\,\mapsto\pi'_{2*}
\biggl(\<\pi_1^{'*}(x)\cdot\frac{e_{-h}\bigl(\Hom(A',B')\bigr)}{u-e(B')}\biggr)\,,
\\[6pt]
\rho^+\bigl(F_p(u)\bigr)\>:\>x\,\mapsto\pi''_{2*}
\biggl(\<\pi_1^{''*}(x)\cdot\frac{e_{-h}(\Hom\bigl(B'',C'')\bigr)}{u-e(B'')}\biggr)\,,
\end{gather*}
\vv-.3>
\begin{gather*}
\rho^-\bigl(E_p(u)\bigr)\>:\>x\,\mapsto\,(-1)^{\la_p-\la_{p+1}+1}\,\pi'_{2*}
\biggl(\<\pi_1^{'*}(x)\cdot\frac{e_{-h}\bigl(\Hom(B',C')\bigr)}{u-e(B')}
\biggr)\,,
\\[6pt]
\rho^-\bigl(F_p(u)\bigr)\>:\>x\,\mapsto\,(-1)^{\la_p-\la_{p+1}+1}\,\pi''_{2*}
\biggl(\<\pi_1^{''*}(x)\cdot\frac{e_{-h}\bigl(\Hom(A'',B'')\bigr)}
{u-e(B'')}\biggr)\,.
\\[-12pt]
\end{gather*}
\end{thm}

This theorem for the $\rho^+$-structure was conjectured in \cite{GV} and proved
by M.\,Varagnolo in \cite{Va}. Our proof below is different from the proof in
\cite{Va}, where the author follows \cite{Na} and checks that the right hand
sides of formulae for $\rho^+$ satisfy the Yangian relations.

\begin{proof}
By Theorem \ref{rho+-}, we already know that the $\rho^\pm$-structures on
\,$\Hplus$ are Yangian module structures. Hence we only need to check that
formulae of Theorem \ref{rho+-} are given by the right hand sides of formulae
in Theorem \ref{Varag}.
\vsk.1>
The proof is a straightforward application of the equivariant localization
formulation of push-forward maps. We omit the details, because they are
completely analogous to the arguments in the Appendix of \cite{RSTV}.
Notice that the sign \,$(-1)^{\la_p-\la_{p+1}+1}$ in the formulae for
\,$\rho^-$ comes from the fact that the functions \,$R(\zb_I)$ here and
\,$R(\zb_{I_1}\lsym|\<\>\zb_{I_N})$ in \cite{RSTV} differ by the sign
\,$(-1)^{\sum_{i<j}\la_i\la_j}$.
\end{proof}

\begin{rem}
Observe that the Euler classes in the four expressions above are Euler classes
of fiberwise tangent or cotangent bundles of some of the forgetful maps
(\ref{eqn:forgetfuls}), following the convention of
Section \ref{sec Partial flag varieties} that \,$\C^*$ acts
on these bundles by fiberwise multiplication.
\end{rem}

\section{Bethe algebras \,$\tbk(\Vl)$\>, \,$\tbk(\DL)$, \,and discrete
Wronskian}
\label{elsewhere}
The algebra \,$H_\bla$, see \Ref{Hrel}, is a free polynomial algebra over
\,$\C$ with generators \,$h$ and elementary symmetric functions
\,$\si_i(\ga_{p,1}\lc\ga_{p,\la_p})$ \,for \,$p=1\lc N$, \,$i=1\lc\la_p$.
Hence, \,$H_\bla$ is a free polynomial algebra over \,$\C$ with generators
\,$h$ \,and the elements \,$f_{p,s}$, see \Ref{fps}, for \,$p=1\lc N$,
\,$s=1\lc\la_p$. Then Theorems \ref{thm main+}, \ref{thm main-} show that
the Bethe algebras \,$\tbi(\Vl)$ \,and \,$\tbi(\DL)$ \,are
\vvn.1>
free polynomial algebras over \,$\C$ with generators \,$h$ \,and the respective
elements \,$C^\pm_{p,s}$, see \Ref{Cps}, for \,$p=1\lc N$, \,$s=1\lc\la_p$.

\vsk.2>
In this section we give similar statements for the Bethe algebras
\,$\tbk(\Vl)$ \,and \,$\tbk(\DL)$. We also formulate counterparts
of Theorems \ref{thm main+}, \ref{thm main-}, see Theorems \ref{mainK+},
\ref{mainK-} below.

\vsk.3>
In the rest of the paper we assume that \,$\kkk$ are distinct
numbers. Set
\beq
\label{W}
\Wk(u)\,=\,\det\>\biggl(\kk_i^{N-j}\,\prod_{k=1}^{\la_i}\,
\bigl(u-\gak_{\ik}+h\>(i-j)\bigr)\<\biggr)_{i,j=1}^N\,.
\vv.1>
\eeq
The function $\Wk(u)$ is essentially a Casorati determinant (discrete
Wronskian) of functions \;$g_i(u)=\>\kk_i^{u/h}\>
\vvn-.4>
\prod_{j=1}^{\la_i}\,\bigl(u-\gak_{\ij}+h\>(i-1)\bigr)$\>,
\be
\Wk(u)\,=\,\det\>\Bigl(\<\>g_i\bigl(u-h\>(j-1)\bigr)\<\Bigr)_{i,j=1}^N
\;\prod_{i=1}^N\,\kk_i^{N-1-u/h}\,.
\ee

\vsk-.1>
Define the algebra
\beq
\label{HK}
\Hck_\bla\>=\,\Czghl\>\Big/
\Bigl\bra \Wk(u)\>=\<\!\prod_{1\le i<j\le N}\!(\kk_i-\kk_j)\,
\prod_{a=1}^n\,(u-z_a)\Bigr\ket\,.
\eeq
For example, if \,$N=n=2$ \,and \,$\bla=(1,1)$, \,then the relations are
\vvn.1>
\be
\gak_{1,1}+\gak_{2,1}\>=\,z_1+z_2\,,\kern3em
\gak_{1,1}\>\gak_{2,1}+\>
\frac{\kk_2}{\kk_1-\kk_2}\,h\>(\gak_{1,1}-\gak_{2,1}+h)\,=\,z_1\<\>z_2\,.
\ee
It is easy to see that the subalgebra \,$\Hck_\bla$ does not change
if all $\kkk$ are multiplied simultaneously by the same number.
Notice that in the limit $\kk_{i+1}/\kk_i\to 0$ \,for all \,$i=1\lc N-1$,
the relations in \,$\Hck_\bla$ turn into the relations in \,$H_\bla$,
see \Ref{Hrel}.

\vsk.2>
Below we describe isomorphisms of the regular representation of \,$\Hck_\bla$
with the \,$\tbk(\Vl)$-module $\Vl$, \,as well as with the \,$\tbk(\DL)$-module
\,$\DL$.

\vsk.3>
Let \,$x$ be a complex variable. Set
\beq
\label{Wh}
\Whk(u,x)\,=\,\det\>\biggl(\kk_i^{N-j}\,\prod_{k=1}^{\la_i}\,
\bigl(u-\gak_{\ik}+h\>(i-j)\bigr)\!\biggr)_{i,j=0}^N\,,
\vv.3>
\eeq
where \,$\kk_0\<=x$ \,and \,$\la_0=0$. Clearly,
\vvn.1>
\;$\Whk(u,x)=x^N\>\Wk(u)\lsym+(-1)^N\>\Wk(u+h)$\>. Define the elements
\,$\Wk_{p,s}\in\Hck_\bla$ \,for \,$p=1\lc N$, \,$s\in\Z_{>0}$, \,by the rule
\vvn.4>
\beq
\label{Wps}
\frac{\Whk(u,x)}{\Wk(u)}\,=\,\prod_{i=1}^N\,(x-\kk_i)\>
+\>h\>\sum_{p=1}^N\,\sum_{s=1}^\infty\,(-1)^p\,\Wk_{p,s}\;x^{N-p}\>u^{-s}\,.
\vv.3>
\eeq
Define also the elements \,$\Uk_{i,s}\in\Hck_\bla$
\,for \,$i=1\lc N$, \,$s\in\Z_{>0}$ \,by the rule
\vvn.3>
\beq
\label{U}
\Uk_{i,s}\,=\,\sum_{p=1}^N\,(-1)^{p-1}\,\Wk_{p,s}\,\kk_i^{N-p-1}
\prod_{\satop{j=1}{j\ne i}}^N\,\frac1{\kk_i-\kk_j}\;,
\vv-.5>
\eeq
so that
\vvn-.1>
\be
\frac{\Whk(u,x)}{\Wk(u)}\,\prod_{i=1}^N\,\frac1{x-\kk_i}\,=\,
1\>-\>h\>\sum_{i=1}^N\,\sum_{s=1}^\infty\,
\frac{\kk_i}{x-\kk_i}\;\Uk_{i,s}\,u^{-s}\>.
\ee

\begin{prop}
\label{Ugen}
The algebra \,$\Hck_\bla$ is a free polynomial algebra over \,$\C$ with
generators \,$h$ \,and \;$\Uk_{i,s}$ \,for \,$i=1\lc N$, \,$s=2\lc\la_i+1$.
\end{prop}
\begin{proof}
Formulae \Ref{W} and \Ref{Wh}\,--\,\Ref{U} imply that \,$\Uk_{i,1}=\la_i$ \,for
all \,$i=1\lc N$, \,and
\vvn.3>
\beq
%% no label
\Uk_{i,s}\,=\,(s-1)\>\si_{s-1}(\gak_{i,1}\lc\gak_{i,\la_i})\>+\>
\Upk_{i,s}\,, \qquad s\ge 1\,,
\vv.3>
\eeq
where \,$\si_{s-1}$ is the $(s-1)$-th elementary
symmetric function and \,$\Upk_{i,s}$ is a polynomial in \,$h$ \,and
\,$\si_r(\gak_{j,1}\lc\gak_{j,\la_j})$ \,for \,$r<s-1$ \,and \,$j=1\lc N$.
For instance,
\vvn.2>
\beq
\label{U2}
\Uk_{i,2}\,=\,\si_1(\gak_{i,1}\lc\gak_{i,\la_i})
\>-\>h\>i\>\la_i+\>\frac h2\,\la_i\>(\la_i+1)-\>
h\>\la_i\sum_{\satop{j=1}{j\ne i}}^N\,\frac{\kk_j\>(\la_j+1)}{\kk_i-\kk_j}\,.
\vv-.1>
\eeq
Since \,$\Hck_\bla$ is clearly a free polynomial algebra over \,$\C$
with generators \,$h$ \,and \,$\si_s(\gak_{i,1}\lc\gak_{i,\la_i})$
\,for \,$i=1\lc N$, \,$s=1\lc\la_i$, the statement follows.
\end{proof}

\begin{cor}
\label{Wgen}
The elements \,$\Wk_{p,s}$ \,for \,$p=1\lc N$, \,$s\ge 2$, together with
\,$\C[h]$ generate the algebra \,$\Hck_\bla$.
\qed
\end{cor}

\begin{thm}[\cite{MTV5}]
\label{mainK+}
The assignment \;$\mukp_\bla\!:\Wk_{p,s}\mapsto\,\pho^+(h^{s-1}\Bk_{p,s})$,
\,$p=1\lc N$, \,$s>0$, extends uniquely to an isomorphism of \;$\CZH$-algebras
\;$\mukp_\bla\!:\Hck_\bla\<\to\tbk(\Vl)$\>. The map
\vvn-.4>
\beq
\label{nukp}
\nukp_\bla\!:\Hck_\bla\to\>\Vl\,,\qquad \quad %% \nukp_\bla\!:
[f]\,\mapsto\,\mukp_\bla([f])\,v^+_\bla\>,
\vv.3>
\eeq
is an isomorphism of vector spaces identifying the \;$\tbk(\Vl)$-module
\;$\Vl$ and the regular rep\-resentation of \,$\Hck_\bla$.
\end{thm}

\begin{thm}[\cite{MTV5}]
\label{mainK=}
The assignment \;$\muke_\bla\!:\Wk_{p,s}\mapsto\,\pho^+(h^{s-1}\Bk_{p,s})$,
\,$p=1\lc N$, \,$s>0$, extends uniquely to an isomorphism of \;$\CZH$-algebras
\;$\muke_\bla\!:\Hck_\bla\<\to\tbk(\DLe)$\>. The map
\vvn-.4>
\beq
\label{nuke}
\nuke_\bla\!:\Hck_\bla\to\>\DLe\,,\qquad \quad %% \nuke_\bla\!:
[f]\,\mapsto\,\muke_\bla([f])\,v^=_\bla\>,
\vv.2>
\eeq
is an isomorphism of vector spaces identifying the \;$\tbk(\DLe)$-module
\;$\DLe$ and the regular representation of \,$\Hck_\bla$.
\end{thm}

\begin{thm}[\cite{MTV5}]
\label{mainK-}
The assignment \;$\mukm_\bla\!:\Wk_{p,s}\mapsto\,\pho^-(h^{s-1}\Bk_{p,s})$,
\,$p=1\lc N$, \,$s>0$, extends uniquely to an isomorphism of \;$\CZH$-algebras
\;$\mukm_\bla\!:\Hck_\bla\<\to\tbk(\DL)$\>. The map
\vvn-.4>
\beq
\label{nukm}
\nukm_\bla\!:\Hck_\bla\to\>\DL\,,\qquad \quad %% \nukm_\bla\!:
[f]\,\mapsto\,\mukm_\bla([f])\,v^-_\bla\>,
\vv.2>
\eeq
is an isomorphism of vector spaces identifying the \;$\tbk(\DL)$-module
\;$\DL$ and the regular representation of \,$\Hck_\bla$.
\end{thm}

%% \noindent
%% The proofs of Theorems \ref{mainK+}\,--\,\ref{mainK-} are similar
%% to the proof of Theorem~6.7 in \cite{MTV3} and will be published elsewhere.

\vsk.2>
Theorems \ref{mainK+}\,--\,\ref{mainK-} are instances of the discrete geometric
Langlands correspondence, see the corresponding versions of the geometric
Langlands correspondence in \cite{MTV2, MTV3}, cf.~also \cite{RSTV}.

Theorems \ref{mainK=} and \ref{mainK-} are equivalent by Lemma \ref{=-}.
In particular,
\vvn.2>
\beq
\label{muk=-}
\Pit\;\muke_\bla(f)\,=\,\mukm_\bla(f)\;\Pit\;,\qquad
\Pit\;\nuke_\bla\<=\,\nukm_\bla
\vv.1>
\eeq
where \;$\Pit$ \,is given by \Ref{Pit} and \,$f\<\in\Hck_\bla$\>.

\begin{cor}
\label{VBk}
The \,$\tbk\<$-modules \,$\V^+\!$, \,$\DVe\!$, and \,$\DV\!$ are isomorphic.
\end{cor}
\begin{proof}
The isomorphisms restricted to weight subspaces are
\;$\nuke_\bla\<\>(\nukp_\bla)^{-1}\!:\Vl\to\DLe$ \,and
\;$\nukm_\bla\<\>(\nukp_\bla)^{-1}\!:\Vl\to\DL$\>.
\end{proof}

\begin{cor}
\label{Bkmax}
The subalgebras \;$\tbk(\Vl)\subset\End(\Vl)$, \;$\tbk(\DLe)\subset\End(\DLe)$
\,and\\ \;$\tbk(\DL)\subset\End(\DL)$ \,are maximal commutative subalgebras.
\qed
\end{cor}

\begin{cor}
\label{generate}
The Bethe algebra \,$\tbk(\Vl)$ is a free polynomial algebra over
\,$\C$ with generators \,$h$ \,and \,$\pho^+(\Sk_{p,s})$ \,for \,$p=1\lc N$,
\,$s=1\lc\la_p$, where the elements \,$\Sk_{p,s}\!\in\tbk\<$ are defined
\vvn.1>
by \Ref{S}. The Bethe algebra \,$\tbk(\DLe)$ is a free polynomial algebra over
\,$\C$ with generators \,$h$ \,and \,$\pho^+(\Sk_{p,s})$ \,for \,$p=1\lc N$,
\,$s=1\lc\la_p$. The Bethe algebra \,$\tbk(\DL)$ is a free polynomial algebra
over \,$\C$ with generators \,$h$ \,and \,$\pho^-(\Sk_{p,s})$ \,for
\,$p=1\lc N$, \,$s=1\lc\la_p$.
\end{cor}
\begin{proof}
Since \,$\mukp_\bla(U_{i,s})=\muke_\bla(U_{i,s})=\pho^+(\Sk_{i,s})$
\,and \,$\mukm_\bla(U_{i,s})=\pho^-(\Sk_{i,s})$, see \Ref{U}, \Ref{S},
\vvn.06>
the statement follows from Theorems \ref{mainK+}\,--\,\ref{mainK-}, and
Proposition \ref{Ugen}.
\end{proof}

\begin{lem}
\label{mukfg}
For any \;$f\<\in\Hck_\bla$, and any \,$g_1\in\Vl$, \,$g_2\in\DLe$,
\,$g_3\in\DL$, we have
\vvn.2>
\be
\Sc\bigl(\mukp_\bla(f)\>g_1,g_3\bigr)\,=\,
\Sc\bigl(g_1,\mukm_\bla(f)\>g_3\bigr)\,,\qquad
\Sc\bigl(\muke_\bla(f)\>g_2,g_3\bigr)\,=\,
\Sc\bigl(g_2,\mukm_\bla(f)\>g_3\bigr)\,.
\vv.2>
\ee
\end{lem}
\begin{proof}
The statement follows from the definition of the maps \,$\mu^\pm_\bla$ and
Corollary \ref{SB}.
\end{proof}

\vsk.3>
Recall that \,$\ty$ contains \,$U(\gln)$ as a subalgebra. Denote by \,$\ty^\h$
the subalgebra of \,$\ty$ commuting with \,$U(\h)$, where \,$\h\subset\gln$
in the Lie subalgebra generated by \,$e_{\ii}$, \,$i=1\lc N$. The isomorphisms
\;$\nukpm_\bla$ induce homomorphisms
\vvn.2>
\beq
\label{rhoK}
\rhkpm_\bla:\ty^\h\to\>\End(\Hck_\bla)\,,\kern3em
\rhkpm_\bla:X\,\mapsto\,(\nukpm_\bla)^{-1}\>\pho^\pm(X)\;\nukpm_\bla\>.
\kern-3em
\vv.3>
\eeq
By relations \Ref{muk=-} and \Ref{pho+-}, we also have
\;$\rhkm(X)\<\>=\>(\nuke_\bla)^{-1}\>\pho^+(X)\;\nuke_\bla$\>.

\begin{lem}
\label{lrhoXK}
For \,$i=1\lc N$, we have
\vvn.2>
\begin{align*}
\rhkpm_\bla(\Xkm_{\bla\<\>,\<\>i})\>:\>[f]\;\mapsto\,
\biggl[f(\zb;\Gmm;h)\>\biggl(\,\sum_{k=1}^{\la_i}\<\>\gak_{\ik} &{}
+\>h\>\sum_{j=1}^{i-1}\,\frac{\kk_i}{\kk_i-\kk_j}\;\min\<\>(0\<\>,\la_j-\la_i)
\>+{}
\\[4pt]
&{}+\>h\!\sum_{j=i+1}^N\frac{\kk_j}{\kk_i-\kk_j}\;\min\<\>(0\<\>,\la_j-\la_i)
% -\>h\!\sum_{j=i+1}^N\frac{\kk_j}{\kk_i-\kk_j}\;(\la_i-\la_j)
\>\biggr)\biggr]\>.
% \kern-2em
\end{align*}
\end{lem}
\begin{proof}
The statement follow from formulae \Ref{U2}, \Ref{X}, \Ref{Xkm}
because \,$\rhkpm_\bla(\Sk_{i,2})$ \,acts as multiplication by the element
\,$\Uk_{i,2}$\>.
\end{proof}

\vsk.3>
Define pairings
\vvn.1>
\beq
\label{SHk}
(\,{,}\,):\>\Hck_\bla\<\ox\Hck_\bla\>\to\,\CZH\,,\qquad
(f,\<\>g\<\>)\,=\,
\Sc\bigl(\<\>\nukp_\bla\<(f)\<\>,\>\nukm_\bla\<(g)\bigr)\,,
\eeq
and
\beq
\label{<Hk>}
\bra\,{,}\,\ket:\>\Hck_\bla\<\ox\Hck_\bla\>\to\,Z^{-1}\>\CZH\,,\qquad
\bra f,\<\>g\<\>\ket\,=\,
\Sc\bigl(\<\>\nuke_\bla\<(f)\<\>,\>\nukm_\bla\<(g)\bigr)\,.
\vv.4>
\eeq
Here \;$Z=\prod_{\>i\ne j}\>(z_i-z_j+h)$\>, \;cf.~\Ref{Z}.

\begin{lem}
\label{<inv>}
Pairings \Ref{SHk} and \Ref{<Hk>} are symmetric and invariant,
\vvn.4>
\begin{alignat*}2
(f_1\>,f_2\<\>)\, &{}=\,(f_2\>,f_1\<\>)\,, &
(f_1\>f_2\>,f_3\<\>)\, &{}=\,(f_2\>,f_1\>f_3\<\>)\,,
\\[6pt]
\bra f_1\>,f_2\<\>\ket\, &{}=\,\bra f_2\>,f_1\<\>\ket\,,\qquad &
\bra f_1\>f_2\>,f_3\<\>\ket\, &{}=\,\bra f_2\>,f_1\>f_3\<\>\ket\,,
\\[-13pt]
\end{alignat*}
for any \,$f_1\>,f_2\>,f_3\<\in\Hck_\bla$.
\end{lem}
\begin{proof}
Theorems \ref{mainK+}\,--\,\ref{mainK-} and Corollary \ref{SB} yield
\;$(f_1\>,f_2\<\>)\>=\>(\<\>1\<\>,f_1\>f_2\<\>)$ \;and
\;$\bra f_1\>,f_2\<\>\ket\>=\>\bra\<\>1\<\>,f_1\>f_2\<\>\ket$\>,
\,which proves the statement.
\end{proof}

\begin{example}
Let \,$N=n=2$, \,$\bla=(1,1)$, so that
\,$V_\bla=\>\C\>v_{(1,2)}\oplus\C\>v_{(2,1)}$\>. Then
\vvn.3>
\be
\mukp_\bla:[\>1\>]\,\mapsto\,\begin{pmatrix}\,1 &0\,\\\,0&1\,\end{pmatrix},
\qquad \mukp_\bla\!:[\>\gak_{1,1}\<\>]\,\mapsto\,
\begin{pmatrix} z_1 & h \\ 0 & z_2 \end{pmatrix}+\>
\frac{\kk_2}{\kk_1-\kk_2}\begin{pmatrix}\>h & h\>\\\>h & h\>\end{pmatrix},
%\\[6pt]
%\mukp_\bla\!:[\>\gak_{2,1}\<\>]\,\mapsto\,
%\begin{pmatrix} z_2 & -\>h \\ 0 & z_1\!\!\end{pmatrix}-\>
%\frac{\kk_2}{\kk_1-\kk_2}\begin{pmatrix}\>h & h\>\\\>h & h\>\end{pmatrix},
\ee
\begin{gather*}
\nukp_\bla:[\>1\>]\,\mapsto\,v_{(1,2)}+\>v_{(2,1)}\,,
\\[8pt]
\nukp_\bla\!:[\>\gak_{1,1}\<\>]\,\mapsto\,
\Bigl(z_1+\>h\,\frac{\kk_1+\kk_2}{\kk_1-\kk_2}\>\Bigr)\,v_{(1,2)}+\>
\Bigl(z_2+\>\frac{2\>h\>\kk_2}{\kk_1-\kk_2}\>\Bigr)\,v_{(2,1)}\,,
%\\[8pt]
%\nukp_\bla\!:[\>\gak_{2,1}\<\>]\,\mapsto\,
%\Bigl(z_2-\>h\,\frac{\kk_1+\kk_2}{\kk_1-\kk_2}\>\Bigr)\,v_{(1,2)}+\>
%\Bigl(z_1-\>\frac{2\>h\>\kk_2}{\kk_1-\kk_2}\>\Bigr)\,v_{(2,1)}\,,
%\\[-14pt]
\end{gather*}
\vv.2>
\be
\muke_\bla:[\>1\>]\,\mapsto\,\begin{pmatrix}\,1 &0\,\\\,0&1\,\end{pmatrix},
\qquad \muke_\bla\!:[\>\gak_{1,1}\<\>]\,\mapsto\,
\begin{pmatrix} z_1 & h \\ 0 & z_2 \end{pmatrix}+\>
\frac{\kk_2}{\kk_1-\kk_2}\begin{pmatrix}\>h & h\>\\\>h & h\>\end{pmatrix},
%\\[6pt]
%\muke_\bla\!:[\>\gak_{2,1}\<\>]\,\mapsto\,
%\begin{pmatrix} z_2 & -\>h \\ 0 & z_1\!\!\end{pmatrix}-\>
%\frac{\kk_2}{\kk_1-\kk_2}\begin{pmatrix}\>h & h\>\\\>h & h\>\end{pmatrix},
\ee
\vv.4>
\be
\nuke_\bla:[\>1\>]\,\mapsto\,\frac{v_{(1,2)}-\>v_{(2,1)}}{z_1-z_2-h}\,,\qquad
%\\[8pt]
\nuke_\bla\!:[\>\gak_{1,1}\<\>]\,\mapsto\,
\frac{(z_1-h)\>v_{(1,2)}-\>z_2\>v_{(2,1)}}{z_1-z_2-h}\,,\qquad
%\\[8pt]
%\nuke_\bla\!:[\>\gak_{2,1}\<\>]\,\mapsto\,
%\frac{(z_2+h)\>v_{(1,2)}-\>z_1\>v_{(2,1)}}{z_1-z_2-h}\,,
\ee
\vv.4>
\be
\mukm_\bla:[\>1\>]\,\mapsto\,\begin{pmatrix}\,1 &0\,\\\,0&1\,\end{pmatrix},
\qquad \mukm_\bla\!:[\>\gak_{1,1}\<\>]\,\mapsto\,
\begin{pmatrix} z_1 & 0 \\ h & z_2 \end{pmatrix}+\>
\frac{\kk_2}{\kk_1-\kk_2}\begin{pmatrix}\>h & h\>\\\>h & h\>\end{pmatrix},
%\\[6pt]
%\mukm_\bla\!:[\>\gak_{2,1}\<\>]\,\mapsto\,
%\begin{pmatrix} z_2 & 0 \\ -\>h & z_1 \end{pmatrix}-\>
%\frac{\kk_2}{\kk_1-\kk_2}\begin{pmatrix}\>h & h\>\\\>h & h\>\end{pmatrix},
\ee
\vv.4>
\be
\nukm_\bla:[\>1\>]\,\mapsto\,\frac{v_{(1,2)}-\>v_{(2,1)}}{z_1-z_2+h}\,,\qquad
%\\[8pt]
\nukm_\bla\!:[\>\gak_{1,1}\<\>]\,\mapsto\,
\frac{z_1\>v_{(1,2)}-\>(z_2-h)\>v_{(2,1)}}{z_1-z_2+h}\,,\qquad
%\\[8pt]
%\nukm_\bla\!:[\>\gak_{2,1}\<\>]\,\mapsto\,
%\frac{z_2\>v_{(1,2)}-\>(z_1+h)\>v_{(2,1)}}{z_1-z_2+h}\,.
\ee
\vv.2>
\be
(\<\>1\<\>,1)\>=\>0\,,\qquad (\<\>1\<\>,\gak_{1,1}\<\>)\>=\>1\,,\qquad
(\gak_{1,1}\<,\gak_{1,1}\<\>)\>=\>z_1+z_2+\frac{2\>h\>\kk_2}{\kk_1-\kk_2}\,.
\vv.2>
\ee
\begin{gather*}
\bra\<\>1\<\>,1\ket\,=\,\frac2{(z_1-z_2+h)\>(z_1-z_2-h)}\,,\qquad
\bra\<\>1\<\>,\gak_{1,1}\<\>\ket\,=\,
\frac{z_1+z_2-h}{(z_1-z_2+h)\>(z_1-z_2-h)}\,,
\\[8pt]
\bra\gak_{1,1}\>,\gak_{1,1}\<\>\ket\,=\,
\frac{z_1^2+z_2^2-h\>(z_1+z_2)}{(z_1-z_2+h)\>(z_1-z_2-h)}\,.
\\[-14pt]
\end{gather*}
Notice that up to a simple change of basis, the form \,$\bra\,{,}\,\ket$
\,in this example coincides with the bilinear form \,$\bra\,{,}\,\ket$ \,in
\cite[Section 3.2]{BGr}, namely, with the form associated with the equivariant
quantum cohomology of the cotangent bundle \,$T^*\Pone$ of the projective line
\,$\Pone$. That form was used in \cite{BMO}, see also \cite{BGh}.
\end{example}

\section{Quantum multiplication on $H_\bla$}
\label{multsect}

Recall that \,$\kkk$ are distinct numbers.
Let \;$\beta^\pm_\bla\<:H_\bla\to\<\>\Hck_\bla$ \,be the following isomorphisms
of vector spaces
\vvn-.3>
\be
\beta^+_\bla\<=\,(\nukp_\bla)^{-1}\>\nu^+_\bla\,,\qquad
\beta^-_\bla\<=\,(\nukm_\bla)^{-1}\>\nu^-_\bla\,.
\vv.3>
\ee
We also have \;$\beta^-_\bla\<=\,(\nuke_\bla)^{-1}\>\nu^=_\bla$\>,
\,because \;$\nu^-_\bla\<=\>\Pit\;\nu^=_\bla$ \,and
\vvn.2>
\;$\nukm_\bla\!=\>\Pit\;\nuke_\bla$\>, where \;$\Pit$ \,is given by \Ref{Pit}.
Notice that \;$\beta^\pm_\bla(1)=1\in\Hck_\bla$\>.
\vsk.2>
Define new multiplications $\,\star\,$ and $\,\bul\,$ on \,$H_\bla$
\vvn.4>
by the rule
\beq
\label{mult}
\beta^+_\bla(f\star g)\,=\,\beta^+_\bla(f)\,\beta^+_\bla(g)\,,\qquad
\beta^-_\bla(f\bul g)\,=\,\beta^-_\bla(f)\,\beta^-_\bla(g)
\vv.4>
\eeq
for any \,$f,g\in H_\bla$. The defined multiplications are associative and
commutative.
%Notice that the products \,$f\star g$ \,and \,$f\bul g$ \,are
%rational functions of \,$\kkk$ with poles located only at
%the hyperplanes \,$q_i=q_j$\>.
The products \,$f\star g$ \,and \,$f\bul g$ \,tend to the ordinary product
\,$fg$ \,as \,$q_{i+1}/q_i\to0$ \,for every \,$i=1\lc N-1$,

\vsk.3>
Denote \,$f\star$ \,and \,$f\bul$ \,the corresponding operators on
\,$H_\bla$\>, \;$f\star:g\>\mapsto f\star g$\>, \;$f\bul:g\>\mapsto f\bul g$\>.

\vsk.4>
Recall the \>$\ty$-actions \,$\rho^\pm\<$ on \;$\Hplus$\>, defined
\vvn.1>
in Section \ref{sec cohom and Yang}. Given \,$X\!\in\ty^\h$, denote by
\,$\rho^\pm_\bla(X)\in\End(H_\bla)$ the restriction of \,$\rho^\pm(X)$
to \,$H_\bla$.

\begin{lem}
\label{btrho}
Let \,$f\in H_\bla$ \,and \,$X\!\in\tbk$.
If \;$\nu^+_\bla(f)=\pho^+(X)\,v^+_\bla$,
\vvn.1>
then \,$f\star=\rho^+_\bla(X)$.
Similarly, if \;$\nu^=_\bla(f)=\pho^+(X)\,v^=_\bla$ \,or
\;$\nu^-_\bla(f)=\pho^-(X)\,v^-_\bla$, then \,$f\bul=\rho^-_\bla(X)$.
\end{lem}
\begin{proof}
We will prove the statement for the map \,$\nu^-_\bla$. Other statements are
proved similarly. Since \;$\nu^-_\bla(g)=\nukm_\bla\bigl(\beta^-_\bla(g)\bigr)$
\,for any \,$g\in H_\bla$, we have
\vvn.3>
\begin{align*}
\nu^-_\bla(f\bul g)\>=\>\nu^-_\bla(g\bul f)\>=\>
\mukm_\bla\bigl(\beta^-_\bla(g)\bigr)\,\nu^-_\bla(f)\> &{}=\>
\mukm_\bla\bigl(\beta^-_\bla(g)\bigr)\,\pho^-(X)\,v^-_\bla\<\>=\>
\\[6pt]
{}=\>\pho^-(X)\,\mukm_\bla\bigl(\beta^-_\bla(g)\bigr)\,v^-_\bla\<\> &{}=\>
\pho^-(X)\,\nu^-_\bla(g)\>=\>\nu^-_\bla\bigl(\rho^-(X)\,g\<\>\bigr)\,,
\end{align*}
which proves the claim.
\end{proof}

Denote by \;$\tbkpm_\bla\!$ the images in \,$\End(H_\bla)$ of the Bethe
subalgebra \;$\tbk\!\subset\ty$ under the homomorphisms \,$\rho^\pm$,
respectively. The algebra \;$\tbkp_\bla\!$ is isomorphic to \;$\tbk(\Vl)$,
and the algebra \;$\tbkm_\bla\!$ is isomorphic to \;$\tbk(\DLe)$ and
to \;$\tbk(\DL)$.

\begin{cor}
\label{f1}
For any \,$f\in H_\bla$, the operator \,$f\star$ is the unique element of
\;$\tbkp_\bla\!$ which sends \;$1\<\in H_\bla$ to \,$f$, and the operator
\,$f\bul$ is the unique element of \;$\tbkm_\bla\!$ which sends
\;$1\<\in H_\bla$ to \,$f$.
\qed
\end{cor}

\begin{cor}
\label{betax}
For every \,$i=1\lc N$, we have
\vvn.3>
\beq
\label{gax}
(\ga_{i,1}\lsym+\<\ga_{i,\la_i})\,\star\,=\>\rho^+_\bla(\Xkp_i)\,,\qquad
(\ga_{i,1}\lsym+\<\ga_{i,\la_i})\,\bul\,=\rho^-_\bla(\Xkm_{\bla\<\>,\<\>i})
\,,
\vv.3>
\eeq
where \,$\ga_{\ij}$ are the Chern roots, see \>\Ref{Hrel}, and
\;$\Xkp_i\<\<,\>\Xkm_{\bla\<\>,\<\>i}\<\in\tbk$ are given by \Ref{Xkp},
\Ref{Xkm}.
\end{cor}
\begin{proof}
The statement follows from Lemmas \ref{btrho} and \ref{Xv}.
\end{proof}

\begin{prop}
\label{beta+}
The map \;$\chi^+_\bla:f\mapsto f\star$ \,is an isomorphism
\vvn.1>
\,$H_\bla\to\tbkp_\bla\!$ of vector spaces.
The map \;$\beta^+_\bla\>(\chi^+_\bla)^{-1}\<:\tbkp_\bla\!\<\to\Hck_\bla$ is
an isomorphism of \;$\CZH$-algebras.
The maps \;$\beta^+_\bla\>(\chi^+_\bla)^{-1}$ and \;$\beta^+_\bla$
identify the \,$\tbkp_\bla\!$-module \,$H_\bla$
with the regular representation of the algebra \,$\Hck_\bla$.
\qed
\end{prop}

\begin{prop}
\label{beta-}
The map \;$\chi^-_\bla:f\mapsto f\bul$ \,is an isomorphism
\vvn.1>
\,$H_\bla\to\tbkm_\bla\!$ of vector spaces.
The map \;$\beta^-_\bla\>(\chi^-_\bla)^{-1}\<:\tbkm_\bla\!\<\to\Hck_\bla$ is
an isomorphism of \;$\CZH$-algebras.
The maps \;$\beta^-_\bla\>(\chi^-_\bla)^{-1}$ and \;$\beta^-_\bla$
identify the \,$\tbkm_\bla\!$-module \,$H_\bla$
with the regular representation of the algebra \,$\Hck_\bla$.
\qed
\end{prop}

\vsk.2>
Consider $\kkn$ as variables. For a nonzero complex number \,$\kp$\>,
define the connections \,$\nablas$ and \,$\nablab_{\<\bla}$
\vvn.2>
\beq
\label{nablasb}
\nablas_i=\,\kp\,\ddk_i-(\ga_{i,1}\lsym+\ga_{i,\la_i})\,\star\;,\qquad
\nablab_{\<\bla\<\>,\<\>i}=\,\kp\,\ddk_i-(\ga_{i,1}\lsym+\ga_{i,\la_i})\,\bul
\;,\kern-1em
\eeq
$i=1\lc n$.

\begin{prop}
\label{flat}
The connections \;$\nablas\!$ and \;$\nablab_{\<\bla}$ are flat for
any \,$\kp$.
\end{prop}
\begin{proof}
The statement follows from Corollary \ref{betax} and Lemma \ref{flat+-}.
\end{proof}

\begin{lem}
\label{<2>}
For any \,$f,g\in H_\bla$, we have
\vvn.4>
\be
\intt f\>g\,=\intt\>f\bul g\,=\,
\bigl\bra\<\>\beta^-_\bla(f)\>,\>\beta^-_\bla(g)\bigr\ket\,,
\vv.3>
\ee
where the pairing \,$\bra\,{,}\,\ket$ \,on \;$\Hck_\bla$ is defined
by \;\Ref{<Hk>}.
\end{lem}
\begin{proof}
Propositions \ref{SP}, \ref{SP=} and Lemma \ref{<inv>} imply
\vvn.3>
\begin{align*}
& \kern-1em\!\!
\intt f\>g\,=\,\Sc\bigl(\nu^=_\bla(f)\>,\>\nu^-_\bla(g)\bigr)\,=\,
\,\Sc\bigl(\nuke_\bla\bigl(\beta^-(f)\bigr)\>,
\>\nukm_\bla\bigl(\beta^-(g)\bigr)\bigr)\,=\,
\bigl\bra\<\>\beta^-_\bla(f)\>,\>\beta^-_\bla(g)\bigr\ket\,={}
\\
& \kern-1em
{}=\,\bigl\bra\<\>1\>,\>\beta^-_\bla(f)\>\beta^-_\bla(g)\bigr\ket\,=\,
\Sc\bigl(\<\>v^=_\bla\>,\nukm_\bla\bigl(\beta^-_\bla(f)\>\beta^-_\bla(g)\bigr)
\bigr)\,=\Sc\bigl(\<\>v^=_\bla\>,\>\nu^-_\bla(f\bul g)\bigr)\,=
\intt\>f\bul g\,.
\\[-40pt]
\end{align*}
\end{proof}

\vsk.8>
\begin{cor}
\label{<3>}
For any \,$f_1\>,f_2\>,f_3\<\in H_\bla$, we have
\vvn.4>
\be
\intt(f_1\bul f_2)\>f_3\,=\intt f_1\bul f_2\bul f_3\,=
\intt f_2\,(f_1\bul f_3)\,.
\vv-1.6>
\ee
\qed
\end{cor}

\vsk.4>
\begin{cor}
The operators \,$f\bul$ \,are symmetric with respect to pairing \,\Ref{intnew}.
\end{cor}

In \cite{MO}, the quantum multiplication on $H^*_{T^n\times\C^*}(\tfl)$
is described. The quantum multiplication depends on parameters
$q^{\al_1}\lc q^{\al_{N-1}}$ associated with simple roots
$\al_1\lc\al_{N-1}$ of the Lie algebra \,$\sln$.

\vsk.2>
We identify the notation of \cite{MO} with the notation of this paper
as follows:
\vvn.3>
\beq
\label{id par}
h\,\mapsto\>-\>h\,,\quad\qquad
q^{\al_i}\<\>\mapsto\,\kk_{i+1}/\kk_i\,,\ \quad i=1\lc N-1\,.
\vv.2>
\eeq
and \,$\zzz$ \>are the same as in \cite{MO}.

\begin{thm}
\label{conj}
Under the identification of notation \Ref{id par}, the space $H_\bla$ with
the multiplication \;$\bul$ \;and pairing \,\Ref{intnew} is the equivariant
quantum cohomology algebra $QH_{GL_n\times\C^*}(\tfl;\C)$ of the cotangent
bundle $\tfl$ of a partial flag variety $\Fla$, described in \cite{MO}.
\end{thm}
\begin{proof}
Assume throughout the proof the identification of notation \Ref{id par}.
By \cite[Corollary 6.4]{RTV}, the Yangian structure on the equivariant
cohomology defined in \cite{MO} coincides with the Yangian structure \,$\rho^-$
defined in Section~\ref{sec cohom and Yang}. By \cite[Corollary 7.6]{RTV},
the action of the dynamical Hamiltonians \,$\Xkm_1\<,\dots,\Xkm_N$ on the
equivariant cohomology coincides with the quantum multiplication by the first
Chern classes defined in \cite{MO}. Assume that \,$\zzz$ are fixed, that is,
we consider the corresponding quotients of the Bethe algebra and the algebra
of quantum multiplication. Suppose that $\zzz$ are generic. Then by
formula~\Ref{rhoX8}, the dynamical Hamiltonians $\Xin_1,\dots,\Xin_N$ have
simple spectrum. Thus for generic \,$\kkk$, the dynamical Hamiltonians
\,$\Xkm_1\<,\dots,\Xkm_N$ have simple spectrum too, and hence generate the
whole Bethe algebra (as well as the whole algebra of quantum multiplication).
Both the Bethe algebra and the algebra of quantum multiplication continuously
depend on parameters \,$\kkk$ \>and \,$\zzz$ \>as long as \,$\kkk$ \>are
distinct. Hence, both algebras coincide for all distinct \,$\kkk$ \>and
all fixed \,$\zzz$\>.
\end{proof}

Below in Section \ref{sec Okounk}, we give a direct proof of the special case
\,$N=n$ \,and \,$\bla=(1\lc1)$ of Theorem \ref{conj}. The proof is based on
formulae from \cite{BMO}.

\vsk.3>
Observe that under Theorem \ref{conj}, the connection \,$\nablab$
\>for \,${\ka\<\>=-1}$ \,is the quantum connection on \,$H_\bla$.

\vsk.3>
\begin{cor}
\label{conjcor}
Under the identification of notation \Ref{id par},
the equivariant quantum cohomology algebra \;$QH_{GL_n\times\C^*}(\tfl;\C)$
\,of the cotangent bundle \;$\tfl$ of a partial flag variety \;$\Fla$ is
isomorphic to the algebra \;$\Hck_\bla$.
\end{cor}

This corollary gives a description of the algebra
\;$QH_{GL_n\times\C^*}(\tfl;\C)$ by generators and relations,
see formula \Ref{HK}. Theorem \ref{lim h to inf} below implies
that in the limit \,$h\to\infty$ \,such that all
\vvn-.1>
\be
r_i\>=\,h^{\la_i+\la_{i+1}}\,\kk_{i+1}\<\>/\kk_i\,,\qquad i=1\lc N-1\,,
\vv.3>
\ee
are fixed, this description becomes the known description of the equivariant
quantum cohomology algebra \;$QH_{GL_n}(\Fla;\C)$ of the partial flag variety
\>$\Fla\>$, see for instance \cite[formula (2.22)]{AS}\>.

\vsk.2>
Introduce the \,$N\>{\times}\>N$ \>matrix \,$M(u)$ with entries
\,$M_{\ij}\<=0$ \,for \,$|\<\>i-j\<\>|>1$\>,
\be
M_{\ii}(u)\,=\,\prod_{k=1}^{\la_i}\,(u-\gak_{\ik})\,,\qquad i=1\lc N\>,
\vv.1>
\ee
\be
M_{\ii+1}(u)\,=\,(-1)^{\la_i}\>,
\ \quad M_{\ioi}(u)\,=\,r_i\,,\qquad i=1\lc N-1\,.
\vv.4>
\ee

\begin{lem}
Let \,$A$ \,be the matrix \,used in \,\cite[formula (2.22)]{AS}\>.
Match the notation of \;\cite{AS} and this paper as follows\/{\rm:}
\vvn-.2>
\be
k\,\mapsto N\>,\qquad n_i\>\mapsto\,\la_i\,,
\qquad x^{(i)}_j\,\mapsto\si_j(\gak_{i,1}\lc\gak_{i,\la_i})\,,
\qquad q_i\>\mapsto\,r_i\,,
\ee
where \;$\si_j$ \>is the \,$j$-th \,elementary symmetric function.
Then \,\,$\det\>(u-A)\>=\>\det\>M(u)$.
\end{lem}
\begin{proof}
The claim follows by induction on \>$N$ \>by expanding
% both determinants
\;$\det\>(u-A)$ \,and \;$\det\>M(u)$ with respect to the last column.
\end{proof}

\begin{thm}
\label{lim h to inf}
Assume that \>$\la_i>0$ \>for all \,$i=1\lc N$, \,and \,$h\to\infty$ \,such
that \;$r_1\lc r_{N-1}$ are fixed. Then
\vvn.1>
\be
\Wk(u)\prod_{1\le i<j\le N}\!(\kk_i-\kk_j)^{-1}\,\to\;\det\<\>M(u)\,.
\vv.1>
\ee
\end{thm}
\begin{proof}
Let \;$p_i=\>\prod_{\>k=1}^{\>i-1}\>h^{\la_k}
(\kk_i\<\>/\kk_{i+1})^{N\<-\<\>i\>+1}$, \;$i=1\lc N$.
\vvn.1>
\,Then under the assumption of the theorem,
\vvn-.1>
\be
\prod_{1\le i<j\le N}\!(\kk_i-\kk_j)\,=\,
\kk_1^{\<\>N\<-1}\>\kk_2^{\<\>N\<-2}\<\ldots\>\kk_{N\<-1}\,
\bigl(1+o(1)\bigr)\,,
\vv-.2>
\ee
and
\vvn-.3>
\be
\kk_i^{N-j}\,\prod_{k=1}^{\la_i}\,\bigl(u-\gak_{\ik}+h\>(i-j)\bigr)\,=\,
\kk_j^{\<\>N\<-\<\>j}\;\frac{p_j}{p_i}\;\bigl(M_{\ij}(u)+o(1)\bigr)\,,
\vv.3>
\ee
which implies the claim.
\end{proof}

\section{Special case \,$N=n$, \,$\bla=(1\lc1)$, and Hecke algebra}
\label{sec Special}
In this section we consider the special case when $\Fla$ is the variety
of full flags in $\C^n$.

\subsection{Hecke algebra}
\label{sHecke}
Consider the algebra \;$\Hh$ \,generated by the central element \,$\cc$,
\,pairwise commuting elements \,$\yyy$, \,and elements of the symmetric group
\,$S_n$, subject to relations
\begin{align}
\label{Hecke}
\sd_i\>y_i-y_{i+1}\>\sd_i\>=\,\cc\,, & \qquad i=1\lc n-1\,,
\\[3pt]
\notag
\sd_i\>y_j-y_j\>\sd_i\>=\,0\,, & \qquad j\ne i,i+1\,,
\\[-14pt]
\notag
\end{align}
where \,$\sd_i$ is the transposition of \,$i$ and \,$i+1$. For every nonzero
complex number \,$t$\>, the quotient algebra \;$\Hh/(c=t)$ \,is the degenerate
affine Hecke algebra corresponding to the parameter \,$t$\>.

\vsk.2>
For \,$i=1\lc n$, set
\vvn-.4>
\beq
\label{Y}
\Yk_i\,=\,y_i+\>\cc\>\sum_{j=1}^{i-1}\,\frac{\kk_i}{\kk_i-\kk_j}\,\sd_{\ij}\>+
\>\cc\!\sum_{j=i+1}^n\frac{\kk_j}{\kk_i-\kk_j}\,\sd_{\ij}\,.
\vv.1>
\eeq
where \,$\sd_{\ij}\<\in\Hh$ is the transposition of \,$i$ and \,$j$.
It is known that for any nonzero complex number \,$p$\>,
the formal differential operators
\vvn.3>
\beq
\label{AKZ}
\nabla_i^\Hg\>=\,\kp\,\ddk_i\>-\>\Yk_i,\qquad i=1\lc N,
\vv.3>
\eeq
pairwise commute and, hence, define a flat connection for any \,$\Hh$-module.
This connection is called the {\it affine \KZ/ connection},
see \cite[formula (1.1.41)]{C}.

\subsection{The case \,$N=n$ \,and \,$\bla=(1\lc1)$}
Recall that \,$\kkk$ are distinct numbers.
\begin{thm}
\label{thm:generate}
Let \>$N=n$ \,and \,$\bla=(1\lc1)$. Then the operators \;$\pho^+(\Xk_i)$,
\,$i=1\lc N$, \,and \;$\CZH$ \,generate the Bethe algebra $\tbk(\Vl)$.
Similarly, the operators \;$\pho^+(\Xk_i)$, \,$i=1\lc N$, and \;$\CZH$
\,generate the Bethe algebra $\tbk(\DLe)$, and the operators \;$\pho^-(\Xk_i)$,
\,$i=1\lc N$, and \;$\CZH$ \,generate the Bethe algebra $\tbk(\DL)$.
\end{thm}
\begin{proof}
The statement follows from formula \Ref{X} and Corollary \ref{generate}.
\end{proof}

For any $\si\in S_n$ define $w_\si\in\End(V)$ by the rule
\vvn.3>
\be
w_\si\>(v_{i_1}\!\lox v_{i_n})\>=\,v_{\si(i_1)}\!\lox v_{\si(i_n)}
\vv.3>
\ee
for any sequence $i_1\lc i_n$. The operators \,$w_\si$ define the action
of \,$S_n$ on \,$V$.
%as the Weyl group of the Lie algebra \,$\gln$ on the \,$\gln$-module \,$V$.
We will write \,$w_{(\ij)}=w_\si$ for \,$\si$ \,being the transposition of
\,$i$ \,and \,$j$\>.

\vsk.4>
Recall the elements \;$G_{\ij}\>=\,e_{\ij}\>e_{\ji}\<-e_{\ii}$.

\begin{lem}
\label{Gw}
Let $N=n$ \,and \,$\bla=(1\lc1)$. For any \,$i\ne j$, the element \;$G_{\ij}$
acts on \,$V_\bla$ as \,$w_{(\ij)}$. \qed
\end{lem}

\begin{prop}
\label{HhV}
Let \>$N=n$ \,and \,$\bla=(1\lc1)$.
The assignments \;$y_i\mapsto\pho^+(\Xin_i)$, \;$c\mapsto\<h$,
\;$\sd_j\mapsto G_{(\jj+1)}$, \;$i=1\lc n$, \,$j=1\lc n-1$, define
\vvn.1>
a representation of \;$\Hh$ on \,$\Vl$. Similarly, the assignments
\;$y_i\mapsto\pho^-(\Xin_i)$, \;$c\mapsto-\>h$, \;$\sd_j\mapsto-\>G_{(\jj+1)}$,
\;$i=1\lc n$, \,$j=1\lc n-1$, define a representation of \;$\Hh$ on \,$\DL$.
\end{prop}
\begin{proof}
The verification is straightforward using formulae \Ref{phoX8}, \Ref{Hecke},
and Lemma \ref{Gw}.
\end{proof}

Denote the obtained representations of \,$\Hh$ on \,$\Vl$ and
\,$\DL$ respectively by \,$\psi^+$ and \,$\psi^-$. By Proposition \ref{HhV}
and formulae \Ref{XX}, \Ref{Y}, we have
\vvn.3>
\beq
\label{XY}
\pho^\pm(\Xk_i)\,=\,\psi^\pm(\Yk_i)\,,\qquad i=1\lc n\,.
\vv.3>
\eeq
So for \>$N=n$ \,and \,$\bla=(1\lc1)$, the trigonometric dynamical
connection \Ref{dyneq} defined on \,$\Vl$ through \,$\pho^+\<$ coincides
with the affine \KZ/ connection \Ref{AKZ} defined on \,$\Vl$ through
\,$\psi^+\<$, and similarly, the trigonometric dynamical connection \Ref{dyneq}
defined on $\DL$ through \,$\pho^-\<$ coincides with the affine \KZ/ connection
\Ref{AKZ} defined on $\DL$ through \,$\psi^-\<$.

\begin{lem}
\label{+-}
Let \>$N=n$ \,and \,$\bla=(1\lc1)$.
For the vectors \,$v^+_\bla\in\Vl$ and \,$v^-_\bla\in\DL$, \,we have
\vvn-.3>
\beq
v^+_\bla=\sum_{\si\in S_n\!}\>v_{\si(1)}\<\lox v_{\si(n)}\,,\qquad
v^-_\bla=\,\frac1D\,\sum_{\si\in S_n\!}\,
(-1)^\si\,v_{\si(1)}\<\lox v_{\si(n)}\,.
\vv.1>
\eeq
Thus \;$w_{(\ij)}\>v^\pm_\bla=\,\pm\,v^\pm_\bla$ \,for any \,$i\ne j$.
\qed
\end{lem}

For \,$N=n$ \,and \,$\bla=(1\lc1)$, denote the Chern roots
\,$\ga_{1,1}\lc\<\ga_{n,1}$ of the corresponding line bundles
over \,$\Fla$ by \,$\xxx$, respectively. Then
\beq
\label{H11}
H_\bla=\,\C[\zzz]^{\>S_n}\!\ox\Cx\ox\C[h]\,\Big/\Bigl\bra\,
\prod_{i=1}^n\,(u-x_i)\,=\,\prod_{i=1}^n\,(u-z_i)\Bigr\ket\,,
\vv.3>
\eeq
see \Ref{Hrel}. The assignment
\,$f(\zzz\<\>;\xxx\<\>;h)\mapsto f(\xxx\<\>;\xxx\<\>;h)$
\vvn.2>
\,defines an isomorphism \,$H_\bla$ \>with \,$\Cx\ox\C[h]$\>.

\vsk.3>
Consider the following two \,$\Hh$-actions \,$\tau^\pm$ on
\,${\Cx\ox\C[h]\simeq H_\bla}$. For any \,$i=1\lc N$, \,$\tau^\pm(y_i)$ \,is
the multiplication by \,$x_i$, and \,$\tau^\pm(\cc)$ is the multiplication by
\,$\pm\>h$\>. The elements $\sd_1\lc\sd_{n-1}$ act by the rule
\begin{align*}
\tau^\pm(\sd_i):f(\xxx,h)\,\mapsto\,
\frac{x_i-x_{i+1}\mp h}{x_i-x_{i+1}}\; & f(\xxii,h)\,\pm{}
\\[3pt]
{}\pm\,\frac{h}{x_i-x_{i+1}}\;&f(\xxi,h)\,,
\\[-14pt]
\end{align*}
cf.~\Ref{sh}. The actions of \,$\sd_1\lc\sd_{n-1}$ are uniquely determined
by relations in \,$\Hh$, see \Ref{Hecke}, and the property
\,$\tau^\pm(\sd_i):1\mapsto\>1$ \,for all \,$i=1\lc n-1$.

\begin{thm}
\label{thm Hecke}
For \,$i=1\lc n$, we have
\vvn.3>
\be
\rho^\pm(\Xk_i)\,=\,\tau^\pm(\Yk_i)
\,=\,x_i\pm\>h\>\sum_{j=1}^{i-1}\,
\frac{\kk_i}{\kk_i-\kk_j}\,\tau^\pm(\sd_{\ij})\>\pm\>
h\!\sum_{j=i+1}^n\frac{\kk_j}{\kk_i-\kk_j}\,\tau^\pm(\sd_{\ij})\,.
\vv.1>
\ee
\end{thm}
\begin{proof}
Since \;$\rho^\pm(X)=(\nu^\pm_\bla)^{-1}\,\pho^\pm(X)\;\<\nu^\pm_\bla$
\vvn.1>
\,for any \,$X\in\ty$, the statement follows from formula \Ref{XY},
Lemma \ref{+-}, and the definition of \,$\tau^\pm$.
\end{proof}

In particular, by Theorem \ref{thm Hecke} and formula \Ref{Xkm},
for any \,$i=1\lc n$, we get
\vvn.2>
\beq
\label{rhox}
\rho^-(\Xkm_{\bla\<\>,\<\>i})\,=\,x_i-\>h\>\sum_{j=1}^{i-1}\,
\frac{\kk_i}{\kk_i-\kk_j}\,\bigl(\tau^-(\sd_{\ij})-1\bigr)\>-\>h\!
\sum_{j=i+1}^n\frac{\kk_j}{\kk_i-\kk_j}\,\bigl(\tau^-(\sd_{\ij})-1\bigr)\,.
\eeq

\subsection{Relations in \,$\Hck_\bla$ and trigonometric Calogero-Moser model}

For \,$N=n$ \,and \,$\bla=(1\lc1)$, recall the algebra \,$\Hck_\bla$ defined in
\Ref{HK},
\vvn-.3>
\be
\Hck_\bla\>=\,\C[\zzz]^{\>S_n}\!\ox\Cx\ox\C[h]\,\Big/
\Bigl\bra \Wk(u)\>=\<\!\prod_{1\le i<j\le n}\!(\kk_i-\kk_j)\,
\prod_{a=1}^n\,(u-z_a)\Bigr\ket\,.
\vv.3>
\ee
where \,$\xxx$ \,denote the Chern roots \,$\ga_{1,1}\lc\<\ga_{n,1}$ \,and
\vvn.3>
\be
\Wk(u)\,=\,
\det\>\Bigl(\kk_i^{n-j}\,\bigl(u-x_{i}+h\>(i-j)\bigr)\!\Bigr)_{i,j=1}^n\>.
\vv.3>
\ee

\begin{lem}
\label{betan}
The isomorphism \;$\beta^-_\bla:H_\bla\<\to\Hck_\bla$ is such that
\;$\beta^-_\bla(x_i)=x_i$, \,$i=1\lc n$.
\qed
\end{lem}

Proposition \ref{beta-} for the case \,$N=n$, \,$\bla=(1\lc1)$
\,reads as follows.
\begin{prop}
\label{chib}
The assignment \;$\chib\!:x_i\mapsto x_i\,\bul$, \,$i=1\lc n$, extends uniquely
to an isomorphism \;$\Hck_\bla\<\to\tbkm_\bla$ \,of \;$\CZH$-algebras. The maps
\;$\chib_\bla$ and \;$\beta^-_\bla$ identify the \,$\tbkm_\bla\!$-module
\,$H_\bla$ with the regular representation of the algebra \,$\Hck_\bla$.
\qed
\end{prop}

Define the \;$n\,{\times}\,n$ \,matrix \,$C$ with entries
\vvn.3>
\begin{gather*}
C_{\ii}\>=\,x_i-\>h\>\sum_{j=1}^{i-1}\,\frac{\kk_i}{\kk_i-\kk_j}\>-
\>h\!\sum_{j=i+1}^n\frac{\kk_j}{\kk_i-\kk_j}\,,\qquad i=1\lc n\,,\kern-2em
\\[6pt]
C_{\ij}\>=\,\frac{h\>\kk_i}{\kk_i-\kk_j}\,,\qquad i,j=1\lc n\,,\quad i\ne j\,.
\kern-2em
\end{gather*}
Notice that
\vvn-.3>
\be
C_{\ii}\>=\,\rhkm_\bla(\Xk_i)\,,\qquad i=1\lc n\,,\kern-2em
\vv.4>
\ee
where \,$\Xk_1\lc\Xk_n$ \>are the dynamical Hamil\-ton\-ians \Ref{XX},
see \Ref{Xkm} and Lemma \ref{lrhoXK}.

\begin{lem}
\label{WuC}
\quad$\dsize\Wk(u)\,=\,
\det\>(u-C\<\>)\prod_{1\le i<j\le n}^{}\!(\kk_i-\kk_j)$\,.
\end{lem}
\begin{proof}
The verification is an exercise in linear algebra,
see \cite[Section 6.1]{MTV6}.
\end{proof}

Recall the function \,$\Whk(u,x)$, see \Ref{Wh}.
Define the diagonal matrix \,$Q=\diag\>(\kkn)$.

\begin{lem}
\label{WhCQ}
\quad$\dsize\Whk(u)\,=\,\det\>\bigl((u-C\<\>)\>(x-Q)-h\>Q\bigr)
\prod_{1\le i<j\le n}^{}\!(\kk_i-\kk_j)$\,.
\end{lem}
\begin{proof}
The statement follows from Lemma \ref{WuC} for \,$(n+1)\,{\times}\,(n+1)$
matrices by taking an appropriate limit, see similar calculation in
\cite[Lemma 4.1]{MTV7} for the ordinary Wronskian.
\end{proof}

\begin{cor}
\label{relations11}
For \,$N=n$, \,$\bla=(1\lc1)$, \,we have
\vvn-.2>
\be
\Hck_\bla\>=\,\C[\zzz]^{\>S_n}\!\ox\Cx\ox\C[h]\,\Big/
\Bigl\bra \det\>\bigl((u-C\<\>)\>(x-Q)-h\>Q\bigr)\>=\>
\prod_{a=1}^n\,(u-z_a)\Bigr\ket\,.
\ee
\end{cor}

\noindent
Cf.~this statement with part (4) of Theorem 3.2 in \cite{BMO}.

\vsk.5>
The coefficients of the determinant \;$\det(u-C)$ \,in Lemma \ref{WuC} form
a complete set of integrals in involution of the trigonometric
Calogero-Moser model, see for example \cite{E}. There \,$\kkn$ are
exponentials of coordinates, and the diagonal entries \,$C_{1,1}\lc C_{n,n}$
are momenta. The Hamiltonian of the trigonometric Calogero-Moser model equals
\vvn.3>
\be
\frac12\>\tr\>C^2\>=\,\frac 12\,\sum_{i=1}^n\,C_{\ii}^2\<\>+\>
\frac{h^2}2\!\sum_{1\le i<j\le n}\,\frac{q_i\>q_j}{(q_i-q_j)^2}\;.
\ee

\vsk.4>
The matrices \;$Q$ \>and \;$C$ \,satisfy the relation
\;$\rank\>(C\>Q-Q\>C-h\>Q\<\>)=1$. Thus if $x_1\lc x_n, h$ are numbers,
the matrices \;$Q$ \,and \;$C/h$ \,define a point of the trigonometric
Calogero-Moser space. Notice that the matrices \;$Q$ \>and
\;$\Ct=Q^{-1}C/h$ \,satisfy the relation \;$\rank\>(\Ct\>Q-Q\>\Ct-1)=1$.
So the correspondence \,$(Q,C/h)\mapsto(Q,\Ct)$ \,describes an embedding of
the trigonometric Calogero-Moser space into the rational Calogero-Moser space
as a submanifold of points with invertible matrix \;$Q$.

\subsection{Bethe algebra and quantum equivariant cohomology}
\label{sec Okounk}

In \cite{MO}, the quantum multiplication on $H^*_{T^n\times\C^*}(\tfl)$ is
described. Previously, in \cite{BMO}, A.\,Braverman, D.\,Maulik and
A.\,Okounkov considered the case $N=n$, \,$\bla=(1\lc1)$, of the equivariant
cohomology \,$H^*_{SL_n\times\C^*}(\tfl)$ \,of the full flag variety
and described the associated quantum connection on the trivial bundle
with fiber $H^*_{SL_n\times\C^*}(\tfl;\C)$ over the space $(\C^*)^{n-1}$,
which is the maximal torus of the group \,$SL_n$.

\vsk.2>
Let \>$N=n$ \,and \,$\bla=(1\lc1)$. Consider the quotient algebras
\vvn.4>
\begin{gather*}
\Hb_\bla\<=H_\bla\>\zno\,,
\\[4pt]
\Bbk(\DL)=\<\>\tbk(\DL)\zno\,,\qquad \tbbm_\bla\!\<=\tbkm_\bla\!\<\zno\,.
\\[-12pt]
\end{gather*}
We identify the notation of \cite{BMO} with our notation as follows:
\;$H^*_{SL_n\times\C^*}(\tfl;\C)=\Hb_\bla$, \;$\kp=-\>1$, \;$t=-\>h$,
\;$q^{\al_i^\vee}\!=\<\>\kk_{i+1}/\kk_i$, \,$i=1\lc n-1$, where
\vvn.06>
\,$\al_1\lc\al_{n-1}$ \,are simple roots of the Lie algebra \,$\slnn$,
\,and all functions of \,$\kkn$ are of homogeneous degree zero.
Then the quantum connection \,$\bmo\<$ of \cite{BMO} takes the form
\vvn.3>
\beq
\label{bmo}
\bmo_{\al_i}=\,\ddk_i-\ddk_{i+1}\>-\>(x_i-x_{i+1})\,\bul\,,
\qquad i=1\lc n\,,\kern-3.6em
\vv.4>
\eeq
see \Ref{gax}, \Ref{rhox}.

\begin{lem}
\label{8.6}
The algebra \;$\tbbm_\bla\!$ is generated by the elements
\;$\rho^-_\bla(\Xkm_{\bla\<\>,\<\>i}\<-\Xkm_{\bla\<\>,\<\>i+1})$,
\vvn.06>
\,$i=1\lc n-1$, and \;$\CZH$\>.
\end{lem}

\begin{proof}
By Theorem \ref{thm:generate}, the algebra \,$\Bbk(\DL)$ is generated by
the elements \,$\pho^-(\Xkm_{\bla\<\>,\<\>i})$\>, \,$i=1\lc n$, and \;$\CZH$.
The lemma follows from the identity
\,$\pho^-(\Xkm_{\bla\<\>,1}\<\lsym+\Xkm_{\<\bla\<\>,\<\>n})=
z_1\<\lsym+z_n$, see formula \Ref{Xkm}.
\end{proof}

The algebra of quantum multiplication on \,$\Hb_\bla$ \,determined
in \cite{BMO} is generated by the elements $(x_i-x_{i+1})\,\bul$\,,
\,$i=1\lc n-1$, see \cite{BMO}. Thus by Corollary \ref{betax} and
Lemma \ref{8.6}, we have the following result verifying Theorem \ref{conj}
for \>$N=n$ \,and \,$\bla=(1\lc1)$.

\begin{thm}
\label{thm identification}
For \>$N=n$, \,$\bla=(1\lc1)$, the quantum multiplication on $\Hb_\bla$
determined in \cite{BMO} coincides with the multiplication $\,\bul\,$ on
\,$\Hb_\bla$ introduced in \Ref{mult}.
\end{thm}

Consider the quotient algebra \;$\Hbck_\bla\<=\Hck_\bla\zno$. Since
\,$x_1\<\lsym+x_n=z_1\<\lsym+z_n$, the algebra \,$\Hbck_\bla$ is generated
by the elements \,$x_i-x_{i+1}$, \,$i=1\lc n-1$. By Proposition \ref{chib}
we have the following result.

\begin{cor}
For $N=n$, \,$\bla=(1\lc1)$, the space \,$H_\bla$ as the module over
the algebra of quantum multiplication on \,$H_\bla$ defined in \cite{BMO}
is isomorphic to the regular representation of the algebra \,$\Hbck_\bla$
\,by the isomorphisms \,$x_i-x_{i+1}\mapsto\<\>x_i-x_{i+1}$,
\,$(x_i-x_{i+1})\,\bul\,\mapsto\<\>x_i-x_{i+1}$, \,$i=1\lc n-1$.
\end{cor}

\vsk.3>
It is known that the flat sections of the trigonometric dynamical connection
are given by multidimensional hypergeometric integrals, see \cite{MV},
cf.~\cite{TV2}. These hypergeometric integrals provide flat sections
of the quantum connection defined by \Ref{gauge}, see \cite{TV3}.
This presentation of flat sections of the quantum connection as
hypergeometric integrals is in the spirit of the mirror symmetry, see Candelas
et~al.~\cite{COGP}, Givental \cite{G1,G2}, and \cite{BCK,BCKS,GKLO,I,JK}.

\begin{example}
Let \,$N=n=2$ \,and \,$\bla=(1,1)$. Set \,$x=x_1-x_2$, \,$z=z_1-z_2$,
\vvn.4>
\;$q=\kk_2/\kk_1$\>. Then
\be
\Hb_\bla=\>\C[\<\>x,z^2\<,h\<\>]\>\big/\bra\>x^2\<\<=z^2\>\ket\,,\qquad
\Hbck_\bla=\>\C[\<\>x,z^2\<,h\<\>]\>\Big/\Big\bra
x^2\<-\>\frac{4\>h\>q}{1-q}\,(x+h)=z^2\<\>\Big\ket\,,
\ee
and
\vvn-.5>
\be
x\,\bul\,=\,x\>-\>\frac{2\>h\>q}{1-q}\,\bigl(\tau^-(s)-1\bigr)
\vv.3>
\ee
is the operator of quantum multiplication by \,$x$ \>acting on \,$\Hb_\bla$.
Here the operator \,$\tau^-(s)$ \,is
\vvn.4>
\be
\tau^-(s):f(x,z^2\<,h)\,\mapsto\,
\frac{x+h}x\,f(-\>x,z^2\<,h)\>-\>\frac hx\,f(x,z^2\<,h)\,.
\vv.3>
\ee
The isomorphism of the algebra of quantum multiplication on \,$\Hb_\bla$
and \,$\Hbck_\bla$ is \;$x\,\bul\,\mapsto\<\>x$\>, that is, the operator
\,$x\,\bul\,$ \,acting on \,$\Hb_\bla$ satisfies the equation
\vvn.3>
\be
(x\,\bul\,)^2-\>\frac{4\>h\>q}{1-q}\,(h+x\,\bul\,)\,=\,z^2\>.
\vv.3>
\ee
The quantum differential equation is
\beq
\label{qeq}
-\>2\>\kp\>q\>\frac\der{\der\>q}\,I\,=\,x\bul I\,.
\vvn.4>
\eeq
The $q\<\>$-hypergeometric solutions of this equation are given by the formula
\vvn.4>
\be
I\>=\,q^{\>z/2\<\>\kp}\,(x+h)\,\int_C\>q^{\>u/\<\kp}\;
\Ga\Bigl(\frac{u-h}\kp\Bigr)\,\Ga\Bigl(\frac{u+z-h}\kp\Bigr)\;
\Ga\Bigl(\<-\>\frac u\kp\>\Bigr)\,\Ga\Bigl(\<-\>\frac{u+z}\kp\Bigr)\,
(2\>u+z-x)\;d\<\>u\,,
\vv.3>
\ee
where the integration contour \;$C$ \,is a deformation of
\vvn.06>
the imaginary line separating the poles of
\;$\Ga\bigl((u-h)/\kp\bigr)\,\Ga\bigl((u+z-h)/\kp\bigr)$ \,from those of
\;$\Ga(\<-\>u/\kp)\,\Ga\bigl(\<-\>(u+z)/\kp\bigr)$\>. Linearly independent
solutions can be obtained by choosing various branches of \,$q^{\>u/\kp}$.

\vsk.3>
Alternatively, a basis of solutions can be obtained by taking the sum
of residues of the integrand at \,$u\in\kp\>\Z_{\ge0}$ \>or at
\,$u\in-\>z+\kp\>\Z_{\ge0}$. The sum of residues at \,$u\in\kp\>\Z_{\ge0}$
\,gives
\vvn.5>
\begin{align*}
J_1\>=\,(x+h){} &\;\Ga\Bigl(\<-\>\frac h\kp\>\Bigr)\,
\Ga\Bigl(\<-\>\frac z\kp\>\Bigr)\,\Ga\Bigl(\frac{z-h}\kp\Bigr)\,\times{}
\\[6pt]
{}\times\,{}& \sum_{d=0}^\infty\;q^{\>d\>+z/2\<\>\kp}\,\,
\frac{z-x+2\<\>\kp\<\>d}{\kp^{\<\>d-1}\,d\>!}\;\;
\prod_{i=0}^{d-1}\,\frac{(h-\kp\>i)\>(h-z-\kp\>i)}{z+\kp\>(i+1)}\;,
\\[-12pt]
\end{align*}
and the sum of residues at \,$u\in-\>z+\kp\>\Z_{\ge0}$ \,gives
\vvn.5>
\begin{align*}
J_2\>=\,(x+h){} &\;\Ga\Bigl(\<-\>\frac h\kp\>\Bigr)\,
\Ga\Bigl(\>\frac z\kp\>\Bigr)\,\Ga\Bigl(\<-\>\frac{z+h}\kp\Bigr)\,\times{}
\\[6pt]
&{}\times\,\sum_{d=0}^\infty\;q^{\>d-z/2\kp}\,\,
\frac{(2\<\>\kp\<\>d-z-x)\>}{\kp^{\<\>d-1}\,d\>!}
\,\,\prod_{i=0}^{d-1}\,\frac{(h-\kp\>i)\>(h+z-\kp\>i)\>}{\kp\>(i+1)-z}\;.
\end{align*}

\vsk.6>\noindent
The hypergeometric solutions of equation \Ref{qeq} are given by the formula
\vvn.5>
\be
I\>=\,q^{z/2\<\>\kp}\,(1-q)^{2\<\>h/\<\kp}
\int_C u^{(h-z)/\<\kp}\>(u-1)^{-h/\<\kp}\>(u-q)^{-h/\<\kp}\,
\Bigl(\>\frac{x-z+2\>h}{u-1}+\frac{x+z}{u-q}\>\Bigr)\;d\<\>u\,,
\vv.4>
\ee
where \;$C$ \,is an appropriate integration contour. We plan to discuss
$q\<\>$-hypergeometric and hypergeometric formulae for flat sections of
the quantum connection in a separate paper.
\end{example}

\vsk.2>
The dynamical connection commutes with the difference \qKZ/ connection
as explained in Section \ref{sec qkz}. For $N=n$, $\bla=(1\lc1)$,
the dynamical connection is identified with the quantum connection by Theorem
\ref{thm identification}. Then the \qKZ/ connection of Section \ref{sec qkz}
induce on $H_\bla$ a difference connection that commutes with the quantum
differential connection in the sense of Section \ref{sec qkz}. This discrete
structure was discussed in \cite{BMO} under the name of shift operators.
We will discuss the \qKZ/ discrete connection in this situation in more details
in another paper.

\vsk.3>
The isomorphism of the Bethe algebra and quantum multiplication establishes
a correspondence between the eigenvectors of the operators of the Bethe algebra
and idempotents of the quantum algebra. Hence we may find the idempotents of
the quantum algebra by the \XXX/ Bethe ansatz.

\begin{example}
Let $N=n=2$, \,$\bla=(1,1)$. If \,$u_1,u_2$ \,are solutions of the Bethe ansatz
equation
\vvn.3>
\be
(u-z)\>(u+z)\,=\,q\>(u-z+2\>h)\>(u+z+2\>h)\,,
\ee
then the elements
\be
w_1\>=\,\frac{x-u_2}{u_1\<-u_2}\,,\qquad w_2\>=\,\frac{x-u_1}{u_2\<-u_1}
\vv.4>
\ee
of \,$\Hb_\bla$ \>satisfy the equations \;$w_i\bul w_j=\>\dl_{\ij}\>w_i$.
\end{example}

\end{document}